\documentclass[leqno,11pt,a4paper]{amsart}

\usepackage[left=2.5cm,right=2.5cm,top=2.5cm,bottom=2.5cm]{geometry}
\usepackage{amsmath}
\usepackage{amsthm}
\usepackage{amssymb}
\usepackage{mathtools}
\usepackage{IEEEtrantools}
\usepackage[shortlabels]{enumitem}
\usepackage{hyperref}
\usepackage{cite}

\usepackage{xcolor}

\theoremstyle{plain}
\newtheorem{theorem}{Theorem}[section]
\newtheorem{prop}[theorem]{Proposition}
\newtheorem{cor}[theorem]{Corollary}
\newtheorem{lem}[theorem]{Lemma}

\theoremstyle{definition}

\newtheorem{remark}[theorem]{Remark}
\newtheorem{definition}[theorem]{Definition}
\newtheorem{question}[theorem]{Question}
\newtheorem{conjecture}[theorem]{Conjecture}
\newtheorem{construction}[theorem]{Construction}

\numberwithin{equation}{section}

\newcommand{\BC}{{\mathbb{C}}}
\newcommand{\BD}{{\mathbb{D}}}

\newcommand{\BP}{{\mathbb{P}}}

\newcommand{\BR}{{\mathbb{R}}}

\newcommand{\BT}{{\mathbb{T}}}

\newcommand{\BZ}{{\mathbb{Z}}}

\newcommand{\MA}{\mathcal{A}}

\newcommand{\ML}{\mathcal{L}}

\newcommand{\MT}{\mathcal{T}}
\newcommand{\MU}{\mathcal{U}}

\newcommand{\on}[1]{\operatorname{#1}}

\newcommand{\vol}{\operatorname{vol}}
\newcommand{\Amin}{\operatorname{A_{min}}}
\newcommand{\AHopf}{\operatorname{A_{Hopf}}}
\newcommand{\cZ}{\operatorname{c_Z}}

\newcommand{\cG}{\operatorname{c_G}}
\newcommand{\cHZ}{\operatorname{c_{HZ}}}
\newcommand{\cSH}{\operatorname{c_{SH}}}
\newcommand{\cEH}[1]{\operatorname{c^{EH}_{#1}}}
\newcommand{\cCH}[1]{\operatorname{c^{CH}_{#1}}}
\newcommand{\cECH}[1]{\operatorname{c^{ECH}_{#1}}}

\newcommand{\Symp}{\operatorname{Symp}}
\newcommand{\Ham}{\operatorname{Ham}}
\newcommand{\Diff}{\operatorname{Diff}}
\newcommand{\Cal}{\operatorname{Cal}}

\newcommand{\identity}{\operatorname{id}}

\newcommand{\interior}{\operatorname{int}}

\begin{document}

\author[O.~Edtmair]{O.~Edtmair}
\address{Department of Mathematics\\University of California at Berkeley\\Berkeley, CA\\94720\\USA}
\email{oliver\_edtmair@berkeley.edu}
%-------------------------------------------------------------------------------
\title[Disk-like surfaces of section and symplectic capacities]{Disk-like surfaces of section and symplectic capacities}

\begin{abstract}
We prove that the cylindrical capacity of a dynamically convex domain in $\BR^4$ agrees with the least symplectic area of a disk-like global surface of section of the Reeb flow on the boundary of the domain. Moreover, we prove the strong Viterbo conjecture for all convex domains in $\BR^4$ which are sufficiently $C^3$ close to the round ball. This generalizes a result of Abbondandolo-Bramham-Hryniewicz-Salom\~{a}o establishing a systolic inequality for such domains. 
\end{abstract}

\maketitle

\tableofcontents

\section{Introduction}

\subsection{Symplectic capacities}

A {\it symplectic capacity} is a function $c$ that assigns numbers $c(X,\omega)\in [0,\infty]$ to symplectic manifolds $(X,\omega)$ of a certain dimension $2n$. Symplectic capacities are required to be monotonic under symplectic embeddings and behave linearly with respect to scalings of the symplectic form. More precisely, one requires:
\begin{itemize}
\item (Monotonicity) If $(X,\omega)$ symplectically embeds into $(X',\omega')$, then $c(X,\omega)\leq c(X',\omega')$.
\item (Conformality) For every $r>0$, we have $c(X,r\omega)=rc(X,\omega)$.
\end{itemize}
We will be mainly concerned with symplectic capacities of domains in Euclidean space $\BR^{2n}=\BC^n$ equipped with the standard symplectic form
\begin{equation*}
\omega_0 \coloneqq \sum\limits_{j=1}^n dx_j\wedge dy_j.
\end{equation*}
We define the ball $B(a)$ and cylinder $Z(a)$ of symplectic width $a>0$ to be the sets
\begin{equation*}
B(a) \coloneqq \left\{z\in\BC^n\mid \pi |z|^2\leq a\right\}\qquad \text{and} \qquad Z(a) \coloneqq \left\{z\in\BC^n \mid \pi |z_1|^2\leq a\right\}.
\end{equation*}
A symplectic capacity is called {\it normalized} if it satisfies
\begin{itemize}
\item (Normalization) $c(B(\pi))= c(Z(\pi)) = \pi$.
\end{itemize}
Two examples of normalized symplectic capacities which are easy to define are the {\it Gromov width} $\cG$ and the {\it cylindrical capacity} $\cZ$. We use the notation $A \overset{s}{\hookrightarrow} B$ to indicate that there exists a {\it symplectic embedding} of $A$ into $B$, i.e. a smooth embedding preserving the symplectic structure. Gromov width and cylindrical capacity are given by
\begin{equation*}
\cG(X)\coloneqq \sup \{a \mid B(a) \overset{s}{\hookrightarrow} X\} \qquad \text{and}\qquad
\cZ(X)\coloneqq \inf \{a \mid X \overset{s}{\hookrightarrow} Z(a) \}.
\end{equation*}
Note, however, that it is highly non-trivial to show that $\cG$ and $\cZ$ are indeed normalized capacities. In fact, this is equivalent to the celebrated Gromov non-squeezing theorem \cite{Gr85}. There is a whole collection of symplectic capacities whose definition involves Hamiltonian dynamics. Examples of normalized capacities arising this way are the Hofer-Zehnder capacity $\cHZ$ introduced in \cite{HZ90} and the Viterbo capacity $\cSH$ defined in \cite{Vit99} using symplectic homology. Other capacities come in families parametrized by positive integers. Examples are the Ekeland-Hofer capacities $\cEH{k}$ defined in \cite{EH89} and \cite{EH90} and the equivariant capacities $\cCH{k}$ constructed by Gutt and Hutchings in \cite{GH18} from $S^1$-equivariant symplectic homology. The first capacities $\cEH{1}$ and $\cCH{1}$ in these families are normalized. In dimension $4$, there exists a sequence of capacities $\cECH{k}$ defined by Hutchings in \cite{Hu11} using embedded contact homology. Again, the first capacity $\cECH{1}$ is normalized. For more information on symplectic capacities, we refer the reader to Cieliebak-Hofer-Latschev-Schlenk \cite{CHLS07} and the references therein.\\

Recall that a contact form on an odd dimensional manifold is a nowhere vanishing $1$-form $\alpha$ such that the restriction of $d\alpha$ to the hyperplane field $\xi\coloneqq\ker\alpha$ is non-degenerate. A contact form $\alpha$ induces a Reeb vector field $R=R_\alpha$ characterized by the equations
\begin{equation*}
\iota_{R}\alpha = 1\qquad \text{and}\qquad \iota_R d\alpha = 0.
\end{equation*}
Studying the dynamical properties of Reeb flows, such as the existence of periodic orbits, is a topic of great interest in symplectic geometry. Contact forms naturally arise on the boundaries of convex or, more generally, star-shaped domains $X\subset\BR^{2n}$. We equip $\BR^{2n}$ with the radial Liouville vector field $Z_0$ and the associated Liouville $1$-form $\lambda_0$
\begin{equation}
\label{eq:liouville_vector_field_and_form}
Z_0 = \sum\limits_{j=1}^n (x_j\partial_{x_j} + y_j\partial_{y_j}) = \frac{1}{2}r\partial_r\qquad\text{and}\qquad \lambda_0 = \frac{1}{2}\sum\limits_{j=1}^n (x_jdy_j-y_jdx_j).
\end{equation}
They are related to the symplectic form $\omega_0$ via $d\lambda_0 = \omega_0$ and $\iota_{Z_0}\omega_0=\lambda_0$. Consider a closed, connected hypersurface $Y\subset\BR^{2n}$. The restriction of $\lambda_0$ to $Y$ is a contact form if and only if the Liouville vector field $Z_0$ is transverse to $Y$. If $Y$ has this property, we call it a star-shaped hypersurface and the domain bounded by $Y$ a star-shaped domain. Note that all star-shaped hypersurfaces are contactomorphic to the sphere $S^{2n-1}$ equipped with its standard contact structure. Moreover, any contact form on $S^{2n-1}$ defining the standard contact structure is strictly contactomorphic to the restriction of $\lambda_0$ to some star-shaped hypersurface. Thus studying star-shaped hypersurfaces is equivalent to studying contact forms on $S^{2n-1}$ defining the standard contact structure.\\

It was proved by Rabinowitz in \cite{Rab78} that there exists a periodic Reeb orbit on the boundary of any star-shaped domain $X\subset\BR^{2n}$. If $\gamma$ is a periodic orbit on $\partial X$, we define its action $\MA(\gamma)$ to be
\begin{equation*}
\MA(\gamma)\coloneqq\int_\gamma\lambda_0.
\end{equation*}
The capacities $\cHZ$, $\cSH$, $\cEH{k}$ and $\cCH{k}$ have the following important property: Their value on a star-shaped domain $X\subset\BR^{2n}$ is equal to the action $\MA(\gamma)$ of some (possibly multiply covered) periodic orbit $\gamma$ on $\partial X$. The capacities $\cECH{k}$ have a similar property. Their values can be represented as the sum of the actions of finitely many periodic orbits.

\subsection{Viterbo's conjecture}

In \cite{Vit00}, Viterbo stated the following fascinating conjecture concerning normalized symplectic capacities.

\begin{conjecture}[Viterbo conjecture]
\label{conjecture:weak_viterbo_conjecture}
Let $X\subset\BR^{2n}$ be a convex domain. Then any normalized symplectic capacity $c$ satisfies the inequality
\begin{equation}
\label{eq:weak_viterbo_conjecture}
c(X) \leq (n!\vol(X))^{\frac{1}{n}}.
\end{equation}
\end{conjecture}

Note that inequality \eqref{eq:weak_viterbo_conjecture} holds for the Gromov width $\cG$. This is an easy consequence of the fact that symplectomorphisms are volume preserving. For all other normalized capacities introduced above, Conjecture \ref{conjecture:weak_viterbo_conjecture} is open. It was proved by Artstein-Avidan-Karasev-Ostrover \cite{AKO14} that Conjecture \ref{conjecture:weak_viterbo_conjecture} implies Mahler's conjecture, an old conjecture in convex geometry. This is one of the reasons for the recent increase in interest in Viterbo's conjecture. There is a stronger version of Viterbo's conjecture.

\begin{conjecture}[Strong Viterbo conjecture]
\label{conjecture:strong_viterbo_conjecture}
Let $X\subset \BR^{2n}$ be a convex domain. Then all normalized symplectic capacities agree on $X$.
\end{conjecture}

The strong Viterbo conjecture together with the above observation that Conjecture \ref{conjecture:weak_viterbo_conjecture} holds for $\cG$ immediately implies that Conjecture \ref{conjecture:weak_viterbo_conjecture} is true for all normalized symplectic capacities. It is an easy consequence of the definitions that any normalized symplectic capacity $c$ satisfies $\cG\leq c\leq \cZ$. Thus Conjecture \ref{conjecture:strong_viterbo_conjecture} is equivalent to saying that Gromov width $\cG$ and cylindrical capacity $\cZ$ agree on convex domains. The convexity assumption in Viterbo's conjectures is essential. Even within the class of star-shaped domains there exist domains $X$ with arbitrarily small volume such that the cylindrical capacity satisfies $\cZ(X)\geq 1$ (see Hermann's paper \cite{Her98}). We refer to Gutt-Hutchings-Ramos \cite{GHR20} for a recent account on Viterbo's conjectures.

\subsection{Embeddings into cylinders}

Except for the Gromov width $\cG$ and the cylindrical capacity $\cZ$, the construction of most, if not all, known normalized capacities is based on Hamiltonian dynamics. There has been a significant amount of work showing that many of these dynamical capacities agree. For example, it follows from work of Ekeland, Hofer, Zehnder, Abbondandolo, Kang and Irie that the values of the capacities $\cHZ$, $\cSH$, $\cEH{1}$ and $\cCH{1}$ on a convex domain $X\subset \BR^{2n}$ all agree with $\Amin(X)$, the minimal action $\MA(\gamma)$ of a periodic orbit $\gamma$ on $\partial X$. We refer to Theorem 1.12 in \cite{GHR20} for a summary. On the other hand, except for the obvious inequalities $\cG\leq c\leq \cZ$, almost nothing is known about the relationship between the dynamical capacities and the embedding capacities $\cG$ and $\cZ$. The purpose of our work is to bridge the gap between dynamics and symplectic embeddings.\\

While it is a well established strategy to use dynamics to obstruct symplectic embeddings, in this paper we go in the opposite direction and use dynamical information to construct symplectic embeddings. The dynamical information is given in terms of {\it global surfaces of section}, an important concept in dynamics going back to Poincar\'{e}. Let $\alpha$ be a contact form on a closed $3$-manifold $Y$. We call an embedded surface (with boundary) $\Sigma\subset Y$ a global surface of section for the Reeb flow if the boundary $\partial \Sigma$ is embedded and consists of closed, simple Reeb orbits, the Reeb vector field $R$ is transverse to the interior $\interior (\Sigma)$, and every trajectory not contained in $\partial\Sigma$ meets $\interior (\Sigma)$ infinitely often forward and backward in time. Surfaces of section are an extremely useful tool in three dimensional Reeb dynamics. For example, they have been used to show that every (non-degenerate) Reeb flow on a closed contact $3$-manifold must have either two or infinitely many periodic orbits (see \cite{HWZ98}, \cite{CHP19} and \cite{CDR20}). In this paper, we will be concerned with disk-like global surfaces of section, i.e. the case that $\Sigma$ is diffeomorphic to the $2$-dimensional closed disk. This implies that the underlying contact manifold must be the $3$-sphere $S^3$ with its unique tight contact structure. For more details on surfaces of section we refer to section \ref{subsection:disk_like_global_surfaces_of_section}.\\

Let us begin by stating the following general dynamical criterion guaranteeing the existence of symplectic embeddings into a cylinder.

\begin{theorem}
\label{theorem:area_of_surface_of_section_bounds_cylindrical_capacity}
Let $X\subset\BR^4$ be a star-shaped domain. Let $\Sigma\subset\partial X$ be a $\partial$-strong (see Definition \ref{def:non_degenerate_surface_of_section}) disk-like global surface of section of the natural Reeb flow on $\partial X$ of symplectic area
\begin{equation*}
a\coloneqq \int_\Sigma \omega_0 = \MA(\partial\Sigma).
\end{equation*}
Then there exists a symplectic embedding $X\overset{s}{\hookrightarrow} Z(a)$. In particular, we have $\cZ(X)\leq a$.
\end{theorem}

The boundary of a general star-shaped domain need not possess a disk-like global surface of section (see van Koert's paper \cite{VK20}). In this case, Theorem \ref{theorem:area_of_surface_of_section_bounds_cylindrical_capacity} is vacuous. However, Theorem \ref{theorem:area_of_surface_of_section_bounds_cylindrical_capacity} is particularly useful when applied to the important class of {\it dynamically convex} domains because for such domains there are general existence theorems for disk-like global surfaces of section due to Hofer-Wysocki-Zehnder \cite{HWZ98}, Hryniewicy-Salom\~{a}o \cite{HS11} and Hryniewicz \cite{Hry14}. Ever since the notion of dynamical convexity was first introduced in \cite{HWZ98}, it has played a significant role in numerous papers on Reeb dynamics (see e.g. \cite{Hry14}, \cite{AM17}, \cite{AM17b}, \cite{GG20}, \cite{Zho20}, \cite{HN16}, just to name a few). We recall the definition of dynamical convexity from \cite{HWZ98}.

\begin{definition}[{\cite[Definition 1.2]{HWZ98}}]
\label{definition:dynamical_convexity}
A contact form $\alpha$ on $S^3$ defining the unique tight contact structure is called {\it dynamically convex} if every periodic Reeb orbit $\gamma$ of $\alpha$ has Conley-Zehnder index $\on{CZ}(\gamma)$ at least $3$. A star-shaped domain $X\subset\BR^4$ is called dynamically convex if the restriction of the standard Liouville $1$-form $\lambda_0$ (see equation \eqref{eq:liouville_vector_field_and_form}) to the boundary $\partial X$ is dynamically convex.
\end{definition}

\begin{remark}
The Conley-Zehnder index $\on{CZ}(\gamma)$ of a periodic Reeb orbit depends on the choice (up to homotopy) of a symplectic trivialization of the contact structure along the orbit. On $S^3$ every contact structure admits a unique global trivialization up to homotopy. This is the trivialization used in Definition \ref{definition:dynamical_convexity}. Let us also point out that usually the Conley-Zehnder index is only defined for non-degenerate orbits. If $\gamma$ is degenerate, then $\on{CZ}(\gamma)$ in Definition \ref{definition:dynamical_convexity} refers to the lower semicontinuous extension of the Conley-Zehnder index.
\end{remark}

It is proved in \cite{HWZ98} that every convex domain $X\subset\BR^4$ whose boundary has positive definite second fundamental form is dynamically convex.\\

Let us introduce the following terminology.
\begin{definition}
A simple closed orbit of a tight Reeb flow on $S^3$ is called a {\it Hopf orbit} if it is unknotted and has self-linking number equal to $-1$ when viewed as a transverse knot.
\end{definition}
The reason for this terminology is that the fibers of the Hopf fibration on $S^3$ are unknotted and have self-linking number $-1$. It is shown by Hofer-Wysocki-Zehnder \cite{HWZ96} that the boundary of every star-shaped domain carries at least one Hopf orbit. Given a star-shaped $X\subset\BR^4$, let us therefore define
\begin{equation}
\label{eq:def_of_ahopf}
\AHopf(X) \coloneqq \inf \left\{ \MA(\gamma) \mid \gamma \enspace \text{is Hopf orbit on} \enspace \partial X \right\} \in (0,\infty).
\end{equation}
The significance of Hopf orbits to our discussion is the following. By work of Hryniewicz-Salom\~{a}o \cite{HS11} and Hryniewicz \cite{Hry14}, a simple periodic orbit of a dynamically convex Reeb flow on $S^3$ bounds a disk-like global surface of section if and only if it is a Hopf orbit.\\

Our main result provides a dynamical characterization of the cylindrical capacity of $4$-dimensional dynamically convex domains. It can be thought of as a generalization of the Gromov non-squeezing theorem from the ball to arbitrary dynamically convex domains.

\begin{theorem}
\label{theorem:a_hopf_equals_cylindrical_capacity_for_convex_domains}
Let $X\subset\BR^4$ be a dynamically convex domain and let $a>0$. Then there exists a symplectic embedding $X\overset{s}{\hookrightarrow} Z(a)$ if and only if $a\geq \AHopf(X)$. In particular, this implies that
\begin{equation*}
\cZ(X) = \AHopf(X).
\end{equation*}
Moreover, the infimum in the definition of $\cZ(X)$ is attained, i.e. there exists an optimal embedding of $X$ into a smallest cylinder.
\end{theorem}

Let us point out that sharp symplectic embedding results such as Theorem \ref{theorem:a_hopf_equals_cylindrical_capacity_for_convex_domains} are rather rare. Moreover, most known results concern highly symmetric toric domains with integrable flows on their boundaries. In contrast to this, the Reeb dynamics of dynamically convex domains can be extremely rich. It is also worth mentioning that while the cylindrical capacity $\cZ$ is apriori rather elusive, the quantity $\operatorname{A_{Hopf}}$ can in principle be computed numerically given an explicit domain.\\

Hryniewicz-Hutchings-Ramos show in \cite{HHR21} that $\AHopf(X)$ agrees with the first embedded contact homology capacity $\cECH{1}(X)$ for all dynamically convex domains $X\subset\BR^4$. We obtain the following corollary.

\begin{cor}
For all dynamically convex domains $X\subset\BR^4$ we have
\begin{equation*}
\cZ(X) = \cECH{1}(X).
\end{equation*}
\end{cor}

\subsection{The local strong Viterbo conjecture}

Abbondandolo-Bramham-Hryniewicz-Salom\~{a}o \cite{ABHS18} proved that for all domains $X\subset\BR^{4}$ whose boundary $\partial X$ is smooth and sufficiently close to the unit sphere with respect to the $C^3$-topology, the minimal action $\Amin(X)$ satisfies inequality \eqref{eq:weak_viterbo_conjecture}. This result was generalized to small perturbations of more general $3$-dimensional Zoll Reeb flows by Benedetti-Kang \cite{BK21} and to arbitrary dimension by Abbondandolo-Benedetti in \cite{AB20}. A consequence of these works is that the capacities $\cHZ$, $\cSH$, $\cEH{1}$ and $\cCH{1}$ satisfy Conjecture \ref{conjecture:weak_viterbo_conjecture} in a $C^3$-neighbourhood of the round ball. Our second main result significantly strengthens this in the $4$-dimensional case. We prove the full strong Viterbo conjecture (Conjecture \ref{conjecture:strong_viterbo_conjecture}) near the ball.

\begin{theorem}
\label{theorem:strong_viterbo_near_round_ball}
Let $X\subset\BR^4$ be a convex domain. If $\partial X$ is sufficiently close to the unit sphere $S^3\subset\BR^4$ with respect to the $C^3$-topology, then all normalized symplectic capacities agree on $X$.
\end{theorem}

\subsection{Systoles and Hopf orbits}

The following question was first raised by Hofer-Wysocki-Zehnder in \cite{HWZ98}.

\begin{question}
Let $X\subset \BR^4$ be a (dynamically) convex domain. Must a systole of $X$, i.e. a Reeb orbit on $\partial X$ of least action, be a Hopf orbit and therefore bound a disk-like global surface of section?
\end{question}

This question is particularly interesting in view of Theorem \ref{theorem:a_hopf_equals_cylindrical_capacity_for_convex_domains}. An affirmative answer would imply that $\operatorname{A_{min}}(X) = \cZ(X)$ for all (dynamically) convex domains. This would force all normalized capacities which are bounded from below by $\operatorname{A_{min}}(X)$ to be equal to the cylindrical capacity. The number of known normalized capacities would be cut down to just two: the Gromov width and the cylindrical capacity.\\

Equality of $\Amin$ and $\AHopf$ was proved by Hainz \cite{Hai07} (see also \cite{HH11}) under certain curvature assumptions.

\begin{theorem}[Hainz]
\label{theorem:unknottedness_of_index_tree_orbits}
Let $X\subset\BR^4$ be a strictly convex domain. Assume that the principal curvatures $a\geq b\geq c$ of the boundary $\partial X$ satisfy the pointwise pinching condition $a\leq b+c$. Then any periodic Reeb orbit $\gamma$ on $\partial X$ of Conley-Zehnder index $3$ is a Hopf orbit.
\end{theorem}

It follows from Ekeland's book \cite{Ek90} (see in particular Theorem 3 and Proposition 9 in chapter V) that for convex domains $X$ with strictly positively curved boundary a Reeb orbit of minimal action has Conley-Zehnder index $3$. Thus we have $\Amin(X)=\AHopf(X)$ if $X$ satisfies the curvature assumptions in Theorem \ref{theorem:unknottedness_of_index_tree_orbits}.

\subsection{Overview of the proofs}

Let us explain the main ideas. Consider the unit disk $\BD\subset\BC$ equipped with the standard symplectic form $\omega_0 = dx\wedge dy$. Let
\begin{equation*}
H:\BR/\BZ \times \BD\rightarrow\BR
\end{equation*}
be a $1$-periodic Hamiltonian vanishing on the boundary $\partial\BD$. Consider {\it time-energy extended phase space}
\begin{equation*}
\widetilde{\BD} \coloneqq \BR_s \times (\BR/\BZ)_t \times \BD
\end{equation*}
equipped with the symplectic form
\begin{equation*}
 \widetilde{\omega_0} = ds \wedge dt + \omega_0.
\end{equation*}
Let
\begin{equation*}
\Gamma(H)\coloneqq \{(H(t,z),t,z)\in\widetilde{\BD}\mid (t,z)\in\BR/\BZ\times\BD\}
\end{equation*}
be the graph of $H$. This is a hypersurface in $\widetilde{\BD}$ of codimension $1$. Hence the symplectic form $\widetilde{\omega_0}$ induces a characteristic foliation on $\Gamma(H)$. It is an easy computation (see Lemma \ref{lem:characteristic_foliation_on_graph}) that the vector field
\begin{equation*}
R\coloneqq X_{H_t}(z) + \partial_t + \partial_tH(z,t)\cdot \partial_s
\end{equation*}
is tangent to the characteristic foliation on $\Gamma(H)$. Observe that the projection of the flow of $R$ to the disk $\BD$ agrees with the Hamiltonian flow $\phi_H^t$ on $\BD$ induced by $H$. In particular, we see that the image of the map
\begin{equation*}
f:\BD\rightarrow\Gamma(H)\qquad z\mapsto (H(0,z),0,z)
\end{equation*}
is a disk-like surface of section of the flow on $\Gamma(H)$ and that the first return map is given by $\phi_H^1$.\\

\noindent{\bf Main construction.} Assume that the Hamiltonian $H$ is strictly positive in the interior $\interior(\BD)$ of the disk and vanishes on the boundary $\partial\BD$. We abbreviate
\begin{equation*}
\widetilde{\BD}_+\coloneqq \BR_{\geq 0}\times\BR/\BZ\times\BD\qquad\text{and}\qquad\widetilde{\BD}_0\coloneqq \{0\}\times\BR/\BZ\times\BD.
\end{equation*}
Consider the map
\begin{equation*}
\Phi: \widetilde{\BD}_+ \rightarrow \BC^2
\qquad \Phi(s,t,z) \coloneqq \left(z\enspace,\enspace\sqrt{\frac{s}{\pi}}\cdot e^{2\pi i t}\right).
\end{equation*}
Note that the image of $\Phi$ is precisely the cylinder $Z(\pi)$. We observe that $\Phi$ restricts to a symplectomorphism
\begin{equation*}
\Phi : (\widetilde{\BD}_+\setminus \widetilde{\BD}_0,\widetilde{\omega_0}) \rightarrow (Z(\pi)\setminus (\BD\times\{0\}),\omega_0).
\end{equation*}
The image $\Phi(\Gamma(H))\subset\BC^2$ is a smooth hypersurface away from the circle $\partial\BD\times\{0\}$. Under suitable assumptions on the boundary behaviour of $H$, it is smooth everywhere. In order to keep the introduction simple, let us ignore this issue for now. Let $A(H)$ denote the domain bounded by $\Phi(\Gamma(H))$. Since $\Phi$ restricts to a symplectomorphism on $\widetilde{\BD}_+\setminus\widetilde{\BD}_0$, it maps the characteristic foliation on $\Gamma(H)$ to the characteristic foliation on $\partial A(H)$. Thus $\Phi\circ f$ parametrizes a disk-like surface of section of the characteristic foliation on $\partial A(H)$. The first return map is given by $\phi_H^1$. Note that $\partial A(H)$ need not be star-shaped or even of contact type.\\

\noindent{\bf Embeddings into the cylinder.} Now suppose that $X\subset\BC^2$ is a star-shaped domain and that the boundary $\partial X$ admits a disk-like surface of section $\Sigma\subset\partial X$. After scaling, we can always assume that the symplectic area of $\Sigma$ is equal to $\pi$. Suppose that $g:\BD\rightarrow\Sigma$ is a parametrization of $\Sigma$ such that $g^*\omega_0=\omega_0$. Here $\omega_0$ denotes the standard symplectic form on both $\BC^2$ and $\BC$. Let $\phi\in\Ham(\BD,\omega_0)$ be the first return map (see equation \eqref{eq:first_return_map}). In section \ref{subsection:disk_like_global_surfaces_of_section}, we explain how to lift $\phi$ to an element $\widetilde{\phi}\in\widetilde{\Ham}(\BD,\omega_0)$ of the universal cover. Such a lift depends on a choice of trivialization which, roughly speaking, is an identification of $\partial X\setminus\partial\Sigma$ with the solid torus. In section \ref{subsection:disk_like_global_surfaces_of_section} we classify isotopy classes of such trivializations via an integer-valued function called degree. Let $\widetilde{\phi}$ denote the lift of $\phi$ with respect to a trivialization of degree $0$. Suppose that $\widetilde{\phi}$ can be generated by a $1$-periodic Hamiltonian $H$ which vanishes on the boundary and is strictly positive in the interior. The above discussion shows that $\partial X$ and $\partial A(H)$ admit disk-like surfaces of section whose first return maps agree. In fact, one can show more: The lifts of the first return maps to $\widetilde{\Ham}(\BD,\omega_0)$ (with respect to trivializations of degree $0$) agree as well. We use a well-known result of Gromov and McDuff (Theorem \ref{theorem:gromov_mcduff_theorem}) to show that this in fact implies that the domains $X$ and $A(H)$ must be symplectomorphic. By definition, $A(H)$ is contained in the cylinder $Z(\pi)$. Therefore, we obtain an embedding of $X$ into the cylinder $Z(\pi)$. We make these arguments precise in Theorem \ref{theorem:embedding_result}. Unfortunately, we do not know whether the lift of the first return map of a disk-like surface of section with respect to a trivialization of degree $0$ can always be generated by a Hamiltonian which vanishes on the boundary and is strictly positive in the interior. In order to resolve this issue, let us observe that if $H$ and $G$ are Hamiltonians vanishing on the boundary $\partial\BD$ and strictly positive in the interior $\interior(\BD)$, then the inequality $H\leq G$ implies the inclusion $A(H)\subset A(G)$. Roughly speaking, this says that we can increase the Hamiltonian generating the first return map by making the domain bigger. More precisely, in Proposition \ref{prop:modify_hypersurface_such_that_return_map_is_generated_by_positive_hamiltonian} we prove that any star-shaped domain with a disk-like surface of section in its boundary can be symplectically embedded into a bigger star-shaped domain whose boundary admits a disk-like surface of section of the same area and with the property that the degree $0$ lift of the first return map can be generated by a positive Hamiltonian. Theorem \ref{theorem:area_of_surface_of_section_bounds_cylindrical_capacity} is an easy consequence of Theorem \ref{theorem:embedding_result} and Proposition \ref{prop:modify_hypersurface_such_that_return_map_is_generated_by_positive_hamiltonian}.\\

\noindent{\bf Ball embeddings.} Suppose that the degree $0$ lift of the first return map of the surface of section $\Sigma\subset\partial X$ can be generated by a Hamiltonian $H$ which satisfies the inequality
\begin{equation}
\label{eq:lower_bound_hamiltonian_in_introduction}
H(t,z)\geq \pi(1-|z|^2).
\end{equation}
Then the domain $A(H)$ is squeezed between the ball $B(\pi)$ and the cylinder $Z(\pi)$, i.e.
\begin{equation*}
B(\pi)\subset A(H)\subset Z(\pi).
\end{equation*}
Since $X$ is symplectomorphic to $A(H)$, this implies that $\cG(X)=\cZ(X)$. The strategy of the proof of Theorem \ref{theorem:strong_viterbo_near_round_ball} is to show that if $\partial X$ is sufficiently close to the round sphere, then the shortest Reeb orbit on $\partial X$ must bound a disk-like surface of section with the property that the degree $0$ lift of the first return map can be generated by a Hamiltonian satisfying \eqref{eq:lower_bound_hamiltonian_in_introduction}. This is the subject of section \ref{section:a_positivity_criterion_for_hamiltonian_diffeomorphisms} and makes use of generalized generating functions as introduced in \cite{ABHS18}. Let us sketch the main ideas in a special case. We assume that $g:\BD\rightarrow\Sigma\subset \partial X$ parametrizes a surface of section whose boundary orbit $\partial\Sigma$ has minimal action among all closed Reeb orbits on $\partial X$. Moreover, assume that $g^*\omega_0=\omega_0$. Let $\widetilde{\phi}\in\widetilde{\Ham}(\BD,\omega_0)$ be the degree $0$ lift of the first return map $\phi$ to the universal cover. The periodic orbits on $\partial X$ different from the boundary orbit $\partial\Sigma$ correspond to the periodic points of $\phi$. As explained in section \ref{subsection:area_preserving_maps_of_the_disk}, any fixed point $p$ of $\phi$ has a well-defined action $\sigma_{\widetilde{\phi}}(p)$ depending on the lift $\widetilde{\phi}$ to the universal cover. Lifts with respect to a trivialization of degree $0$ have the property that $\sigma_{\widetilde{\phi}}(p)$ is equal to the action of the corresponding closed Reeb orbit on $\partial X$ (see Lemma \ref{lem:action_equals_first_return_time}). Since $\partial\Sigma$ is assumed to have minimal action, this implies that
\begin{equation}
\label{eq:action_inequality_in_overview_of_proofs}
\sigma_{\widetilde{\phi}}(p)\geq \MA(\partial\Sigma)=\pi
\end{equation}
for all fixed points $p$ of $\phi$. If $\partial X$ is the unit sphere $S^3$, then the degree $0$ lift of the first return map $\widetilde{\phi}$ is equal to the counter-clockwise rotation by angle $2\pi$. Let us denote this rotation by $\widetilde{\rho}\in\widetilde{\Ham}(\BD,\omega_0)$. If $\partial X$ is sufficiently close to $S^3$ with respect to the $C^3$-topology, then $\widetilde{\phi}$ is $C^1$-close to $\widetilde{\rho}$. This is proved in \cite{ABHS18} and explained in section \ref{section:from_reeb_flows_to_disk_like_surfaces_of_section_and_approximation_results}. In order to simplify the discussion, let us assume that $\widetilde{\phi}$ is actually equal to $\widetilde{\rho}$ in a small neighbourhood of the boundary $\partial\BD$. Therefore, we can regard
\begin{equation*}
\psi\coloneqq \widetilde{\rho}^{-1}\circ\widetilde{\phi}
\end{equation*}
as an element of $\Ham_c(\BD,\omega_0)$, the group of Hamiltonian diffeomorphisms compactly supported in the interior $\interior(\BD)$. It is $C^1$-close to the identity and it follows from \eqref{eq:action_inequality_in_overview_of_proofs} that the action $\sigma_\psi(p)$ is non-negative for all fixed points $p$. The following result is a special case of Corollary \ref{cor:positivity_criterion_for_diffeomorphisms_close_to_the_identity} in section \ref{section:a_positivity_criterion_for_hamiltonian_diffeomorphisms}.

\begin{prop}[Special case of Corollary \ref{cor:positivity_criterion_for_diffeomorphisms_close_to_the_identity}]
\label{prop:special_case_of_corollary_positivity_criterion}
Let $\psi\in\Ham_c(\BD,\omega_0)$ be a Hamiltonian diffeomorphism compactly supported in the interior $\interior(\BD)$. Suppose that all fixed points of $\psi$ have non-negative action and that $\psi$ is close to the identity with respect to the $C^1$-topology. Then $\psi$ can be generated by a non-negative Hamiltonian $H$ with support contained in $\interior(\BD)$.
\end{prop}

We apply Proposition \ref{prop:special_case_of_corollary_positivity_criterion} to the Hamiltonian diffeomorphism $\psi=\widetilde{\rho}^{-1}\circ\widetilde{\phi}$. Let $G$ denote the resulting Hamiltonian. We may assume that $G_t$ vanishes for time $t$ close to $0$ or $1$. Let us define the Hamiltonian $K$ by the formula
\begin{equation*}
K(t,z)\coloneqq \pi(1-|z|^2).
\end{equation*}
This Hamiltonian generates the rotation $\widetilde{\rho}$. Now set
\begin{equation*}
H_t\coloneqq (K\#G)_t\coloneqq K_t + G_t\circ(\phi_K^t)^{-1}.
\end{equation*}
This defines a $1$-periodic Hamiltonian. Its time-$1$-flow represents $\widetilde{\phi}$. Since $G$ is non-negative, $H$ satisfies inequality \eqref{eq:lower_bound_hamiltonian_in_introduction}. As explained above, this implies that $B(\pi)\subset A(H)\subset Z(\pi)$ and hence $\cG(X)=\cZ(X)$.\\

\noindent{\bf Existence of non-negative Hamiltonians.} Let us sketch the proof of Proposition \ref{prop:special_case_of_corollary_positivity_criterion}. It follows the same basic idea as the proof of Corollary \ref{cor:positivity_criterion_for_diffeomorphisms_close_to_the_identity}. The advantage of our simplified setting here is that we can work with standard generating functions (see e.g. chapter 9 in \cite{MS17}) and do not have to appeal to the generalized ones from \cite{ABHS18}. Let $\psi=(X,Y)$ denote the components of $\psi$. There exists a unique generating function $W:\BD\rightarrow\BR$, compactly supported in $\interior(\BD)$, such that
\begin{equation*}
\begin{cases}
X-x = \enspace\partial_2W(X,y)\\
Y-y = -\partial_1W(X,y)
\end{cases}
\end{equation*}
The fixed points of $\psi$ are precisely the critical points of $W$. Moreover, the action of a fixed point is equal to the value of $W$ at the fixed point. Since all fixed points are assumed to have non-negative action, this implies that $W$ takes non-negative values at all its critical points. In particular, this implies that $W$ is non-negative. For $t\in [0,1]$, let us define the generating function $W_t\coloneqq t\cdot W$. Let $\psi_t$ denote the compactly supported symplectomorphism generated by $W_t$. This defines an arc in $\Ham_c(\BD,\omega_0)$ from the identity to $\psi$. Let $H$ be the unique compactly supported Hamiltonian generating the arc $\psi_t$. Our goal is to show that $H$ is non-negative. A direct computation shows that $H_0$, the Hamiltonian $H$ at time $0$, is equal to $W$ and in particular non-negative. The Hamiltonian $H$ need not be autonomous. However, the following is true. For every fixed $t\in[0,1]$, the set of critical points of $H_t$ is equal to the set of critical points of $W$. Moreover, $W$ and $H_t$ agree on this set. Hence $H_t$ takes non-negative values on its critical points. Therefore, the Hamiltonian $H$ must be non-negative.\\

\noindent{\bf Approximation results.} In general, the first return map of a disk-like surface of section need not be equal to the identity in any neighbourhood of $\partial\BD$. Nevertheless, it will be convenient to assume that the Reeb flow in a small neighbourhood of the boundary orbit $\partial\Sigma$ has a specific simple form. More precisely, we want to assume that the local first return map of a small disk transverse to the orbit $\partial\Sigma$ is smoothly conjugated to a rotation. The main purpose of section \ref{section:from_reeb_flows_to_disk_like_surfaces_of_section_and_approximation_results} is to prove that we may approximate a given contact form with contact forms having this property. This is slightly subtle because we need to keep track of a certain number of higher order derivatives of the Reeb vector field in order to be able to apply the results from section \ref{section:a_positivity_criterion_for_hamiltonian_diffeomorphisms} to the first return map.\\

\noindent{\bf Organization.} The rest of the paper is structured as follows:\\
\indent In \S\ref{section:preliminaries} we review some preliminary material on area preserving disk maps (\S\ref{subsection:area_preserving_maps_of_the_disk}) and global surfaces of section (\S\ref{subsection:disk_like_global_surfaces_of_section}).\\
\indent The main results of section \ref{section:from_disk_like_surfaces_of_section_to_symplectic_embeddings}, namely the embedding result Theorem \ref{theorem:embedding_result} and Proposition \ref{prop:modify_hypersurface_such_that_return_map_is_generated_by_positive_hamiltonian} on modifications of star-shaped domains, are stated in \S\ref{subsection:embedding_results}. The construction of the domain $A(H)$ is explained in \S\ref{subsection:main_construction}. Proofs are given in \S\ref{subsection:proof_of_embedding_result} and \S\ref{subsection:proof_of_modification_result}. Note that the reader only interested in Theorems \ref{theorem:area_of_surface_of_section_bounds_cylindrical_capacity} and \ref{theorem:a_hopf_equals_cylindrical_capacity_for_convex_domains} on the cylindrical embedding capacity and not in the local version of the strong Viterbo conjecture (Theorem \ref{theorem:strong_viterbo_near_round_ball}) may skip \S\ref{section:a_positivity_criterion_for_hamiltonian_diffeomorphisms} and \S\ref{section:from_reeb_flows_to_disk_like_surfaces_of_section_and_approximation_results} and directly move on to \S\ref{section:proofs_of_the_main_results}, where we prove our main results.\\
\indent The main results of \S\ref{section:a_positivity_criterion_for_hamiltonian_diffeomorphisms} are Theorem \ref{theorem:positivity_criterion_for_radially_monotone_diffeomorphisms} and Corollary \ref{cor:positivity_criterion_for_diffeomorphisms_close_to_the_identity} guaranteeing the existence of non-negative Hamiltonians generating certain Hamiltonian diffeomorphisms. They are stated in \S\ref{subsection:statement_of_the_positivity_criterion}. In \S\ref{subsection:lifts_to_the_strip} and \S\ref{subsection:generalized_generating_functions} we review material from \cite{ABHS18} on generalized generating functions. The only result that is not also explicitly explained in \cite{ABHS18} is Proposition \ref{prop:correspondence_symplectomorphisms_generating_functions}. The proofs of the main results of \S\ref{section:a_positivity_criterion_for_hamiltonian_diffeomorphisms} are given in \S\ref{subsection:proof_of_the_positivity_criterion_for_radially_monotone_diffeomorphisms} and \S\ref{subsection:proof_of_the_positivity_criterion_for_diffeomorphisms_close_to_the_identity}.\\
\indent \S\ref{section:from_reeb_flows_to_disk_like_surfaces_of_section_and_approximation_results} is slightly technical in nature. The main result that is needed outside of this section is Proposition \ref{prop:contact_forms_in_C_3_neighbourhood_may_be_approx_by_forms_satisfying_criterion} on certain approximations of contact forms.\\
\indent In \S\ref{section:proofs_of_the_main_results} we give proofs of the main results of our paper.\\

\noindent{\bf Acknowledgments.} We are deeply indepted Umberto Hryniewicz, whose talk on \cite{HHR21} inspired this paper. We also thank Michael Hutchings for his suggestion to prove the strong Viterbo conjecture near the round ball and Julian Chaidez for countless stimulating discussions.

\section{Preliminaries}
\label{section:preliminaries}

\subsection{Area preserving maps of the disk}
\label{subsection:area_preserving_maps_of_the_disk}

In this section, we recall some basic concepts and results concerning area preserving diffeomorphisms of the disk. Most of the material is taken from Abbondandolo-Bramham-Hryniewicz-Salom\~{a}o \cite[sections 2.1 and 2.2]{ABHS18}. Let $\omega$ be a smooth $2$-form on the closed unit disk $\BD\subset \BC$. We assume that $\omega$ is positive in the interior $\interior(\BD)$. On the boundary, $\omega$ is allowed to vanish. We let $\Diff^+(\BD)$ denote the group of orientation preserving diffeomorphisms of $\BD$. Let
\begin{equation*}
\pi:\widetilde{\Diff}(\BD)\rightarrow \Diff^+(\BD)\qquad \widetilde{\phi}\mapsto\phi
\end{equation*}
be the universal cover. We define $\Diff(\BD,\omega)\subset\Diff^+(\BD)$ to be the subgroup of all diffeomorphisms preserving $\omega$. Let $\widetilde{\Diff}(\BD,\omega)$ denote the preimage of $\Diff(\BD,\omega)$ under the universal covering map $\pi$. If $\omega$ is nowhere vanishing on the boundary $\partial\BD$, then this agrees with the actual universal cover of $\Diff(\BD,\omega)$. However, in general it need not agree with the universal cover (see Remark 2.1 in \cite{ABHS18}). Elements $\widetilde{\phi}\in\widetilde{\Diff}(\BD,\omega)$ can be represented by arcs $(\phi_t)_{t\in [0,1]}$ in $\Diff^+(\BD)$ which start at the identity and end at $\phi_1=\pi(\widetilde{\phi})\in\Diff(\BD,\omega)$. Two such arcs are equivalent in $\widetilde{\Diff}(\BD,\omega)$ if they are isotopic in $\Diff^+(\BD)$ with fixed end points.\\

Consider a primitive $\lambda$ of $\omega$ and an element $\widetilde{\phi} = [(\phi_t)_{t\in [0,1]}]\in\widetilde{\Diff}(\BD,\omega)$. Then there exists a unique smooth function $\sigma_{\widetilde{\phi},\lambda}\in C^\infty(\BD,\BR)$ such that
\begin{equation}
\label{eq:action_on_disk_first_defining_identity}
\phi^*\lambda - \lambda = d\sigma_{\widetilde{\phi},\lambda}
\end{equation}
and
\begin{equation}
\label{eq:action_on_disk_second_defining_identity}
\sigma_{\widetilde{\phi},\lambda}(z) = \int_{\{t\mapsto\phi_t(z)\}}\lambda
\end{equation}
for all $z\in\partial\BD$ (see \cite[section 2.1]{ABHS18}). We call $\sigma_{\widetilde{\phi},\lambda}$ the {\it action} of $\widetilde{\phi}$ with respect to $\lambda$. We recall the following basic result \cite[Lemma 2.2]{ABHS18}.

\begin{lem}
\label{lem:basic_properties_of_action}
Let $\widetilde{\phi},\widetilde{\psi}\in \widetilde{\Diff}(\BD,\omega)$. Let $\lambda$ be a primitive of $\omega$ and let $u$ be a smooth real-valued function on $\BD$. Then:
\begin{enumerate}
\item $\sigma_{\widetilde{\phi},\lambda+du} = \sigma_{\widetilde{\phi},\lambda} + u\circ\phi - u$
\item $\sigma_{\widetilde{\psi}\circ\widetilde{\phi},\lambda} = \sigma_{\widetilde{\psi},\lambda}\circ\phi + \sigma_{\widetilde{\phi},\lambda}$
\item $\sigma_{\widetilde{\phi}^{-1},\lambda} = -\sigma_{\widetilde{\phi},\lambda}\circ\phi^{-1}$
\end{enumerate}
\end{lem}

In particular, item (1) in Lemma \ref{lem:basic_properties_of_action} implies that the value $\sigma_{\widetilde{\phi},\lambda}(p)$ at a fixed point $p$ of $\phi$ is independent of the choice of primitive $\lambda$ and we will occasionally denote this value by $\sigma_{\widetilde{\phi}}(p)$.\\

The {\it Calabi invariant} $\Cal(\widetilde{\phi})$ is defined to be the integral
\begin{equation*}
\Cal(\widetilde{\phi}) = \int_\BD \sigma_{\widetilde{\phi},\lambda}\cdot \omega.
\end{equation*}
It follows from item (1) in Lemma \ref{lem:basic_properties_of_action} that this is independent of the choice of primitive $\lambda$ and from item (2) that
\begin{equation*}
\Cal:\widetilde{\Diff}(\BD,\omega)\rightarrow\BR
\end{equation*}
is a group homomorphism.\\

Let $(\phi_t)_{t\in [0,1]}$ be an arc in $\Diff(\BD,\omega)$. Let $X_t$ be the vector field generating $\phi_t$. Since $\phi_t$ preserves $\omega$, the interior product $\iota_{X_t}\omega$ is a closed $1$-form. Since $\BD$ is simply connected, there exists a smooth function $H_t$ on $\BD$, unique up to addition of a constant, such that $dH_t=\iota_{X_t}\omega$. The vector field $X_t$ is tangent to the boundary $\partial\BD$. This implies that $dH_t$ vanishes on tangent vectors of $\partial\BD$. Thus $H_t$ is constant on the boundary. We will always use the normalization $H_t|_{\partial\BD}=0$. This uniquely specifies $H_t$. Conversely, if we are given a family of smooth functions $H_t$ vanishing on the boundary, there exists a unique vector field $X_{H_t}$ in the interior $\interior(\BD)$ satisfying $\iota_{X_{H_t}}\omega = dH_t$. If $\omega$ does not vanish on $\partial\BD$, then $X_{H_t}$ smoothly extends to a vector field on the closed disk which is tangent to the boundary. Note that this is not necessarily true if $\omega$ vanishes on the boundary. So while every arc in $\Diff(\BD,\omega)$ is generated by a family of Hamiltonians vanishing on the boundary $\partial \BD$, not every family of such Hamiltonians generates an arc in $\Diff(\BD,\omega)$. The following result \cite[Proposition 2.6]{ABHS18} expresses the action of $\widetilde{\phi}= [(\phi_t)_{t\in [0,1]}]$ in terms the Hamiltonian $H_t$.
\begin{lem}
\label{lem:action_in_terms_of_hamiltonian}
Suppose that $(\phi_t)_{t\in [0,1]}$ is an arc in $\Diff(\BD,\omega)$ generated by a family of Hamiltonians $H_t$ vanishing on the boundary $\partial \BD$. Let $\widetilde{\phi}\in \widetilde{\Diff}(\BD,\omega)$ be the element represented by the arc $(\phi_t)_{t\in [0,1]}$. Then
\begin{equation*}
\sigma_{\widetilde{\phi},\lambda}(z) = \int_{\{t\mapsto\phi_t(z)\}}\lambda + \int_0^1 H_t(\phi_t(z))dt
\end{equation*}
for all $z\in \BD$.
\end{lem}

\subsection{Global surfaces of section}
\label{subsection:disk_like_global_surfaces_of_section}

Let $Y^3$ be a closed oriented $3$-manifold equipped with a nowhere vanishing vector field $R$. Let $\phi^t$ denote the flow generated by $R$. Let $\Sigma\subset Y$ be an embedded compact surface, possibly with boundary, which we also assume to be embedded. We call $\Sigma$ a {\it global surface of section} for the flow $\phi^t$ if the boundary $\partial\Sigma$ consists of simple periodic orbits of $\phi^t$, the vector field $R$ is transverse to $\interior(\Sigma)$ and every trajectory of $\phi^t$ which is not contained in $\partial\Sigma$ meets $\interior(\Sigma)$ infinitely often forward and backward in time. We will always orient surfaces of section such that $R$ is positively transverse to $\Sigma$, i.e. the orientation of $R$ followed by the orientation of $\Sigma$ agrees with the orientation of $Y$. Consider a boundary orbit $\gamma$ of $\Sigma$. We call $\gamma$ {\it positive} if the orientation of $\gamma$ given by $R$ agrees with the boundary orientation of $\Sigma$ and {\it negative} otherwise. We define the {\it first return time} and {\it first return map} by
\begin{equation}
\label{eq:first_return_time}
\sigma:\interior(\Sigma)\rightarrow\BR_{>0}\qquad \sigma(p)\coloneqq \inf\{t>0\mid \phi^t(p)\in \Sigma\}
\end{equation}
and
\begin{equation}
\label{eq:first_return_map}
\phi:\interior(\Sigma)\rightarrow \interior(\Sigma)\qquad \phi(p)\coloneqq \phi^{\sigma(p)}(p).
\end{equation}
Studying the dynamics of the flow $\phi^t$ is equivalent to studying the discrete dynamics of the diffeomorphism $\phi$. Let $\Sigma'$ be a second global surface of section with the same boundary orbits as $\Sigma$, i.e. $\partial\Sigma'=\partial\Sigma$. Then the respective first return maps $\phi$ and $\phi'$ are smoothly conjugated. To see this, we define a transfer map $\psi:\interior(\Sigma)\rightarrow\interior(\Sigma')$ as follows. Let $z_0\in \interior(\Sigma)$ and let $\tau(z_0)$ denote a real number such that $\phi^{\tau(z_0)}(z_0)\in\interior(\Sigma')$. Then there exists a unique smooth extension of $\tau$ to a real-valued function on $\interior(\Sigma)$ such that $\phi^{\tau(z)}(z)\in\interior(\Sigma')$ for all $z\in\interior(\Sigma)$. We define $\psi(z)\coloneqq \phi^{\tau(z)}(z)$. This is a diffeomorphism. The first return maps of $\Sigma$ and $\Sigma'$ are related via $\phi = \psi^{-1}\circ\phi'\circ\psi$.\\
In general, the first return time $\sigma$ and map $\phi$ need not smoothly extend to the boundary $\partial\Sigma$. In order to describe the boundary behaviour, we recall a blow-up construction due to Fried \cite{Fr82}. Our exposition follows Florio-Hryniewicz \cite{FH21}. We define the vector bundle $\xi\coloneqq TY/\langle R\rangle$ on $Y$, where $\langle R\rangle$ is the subbundle of $TY$ spanned by $R$. Moreover, we define the circle bundle $\BP_+\xi\coloneqq (\xi\setminus 0)/\BR_+$. The linearization of $\phi^t$ induces a lift $d\phi^t$ of the flow $\phi^t$ to $\xi$. This lift $d\phi^t$ descends to the bundle $\BP_+\xi$. Now consider a simple closed orbit $\gamma$ of $\phi^t$. Then the torus $\BT_\gamma\coloneqq \BP_+\xi|_\gamma$ is invariant under the projective linearized flow $d\phi^t$. As a set, the blow-up of $Y$ at $\gamma$ is equal to the disjoint union $\overline{Y}\coloneqq (Y\setminus\gamma) \sqcup \BT_\gamma$. It carries the structure of a compact smooth manifold with boundary $\BT_\gamma$. The natural projection $\pi:\overline{Y}\rightarrow Y$ is smooth. The pullback of the restriction of the vector field $R$ to $Y\setminus\gamma$ is a smooth vector field $\overline{R}$ on the interior of $\overline{Y}$. It smoothly extends to all of $\overline{Y}$ (see e.g. \cite[Lemma A.1]{FH21}). The resulting flow $\overline{\phi}^t$ on $\overline{Y}$ lifts the flow $\phi^t$ and its restriction to the boundary $\BT_\gamma$ agrees with the projective linearized flow $d\phi^t$. Consider a surface of section $\Sigma\subset Y$. Let $\overline{Y}$ be the simultaneous blow-up of $Y$ at all the boundary orbits of $\Sigma$. The surface $\Sigma$ lifts to a properly embedded surface $\overline{\Sigma}\subset\overline{Y}$ with boundary $\partial\overline{\Sigma}$ contained in $\partial\overline{Y}$. We recall the following definition from \cite{FH21}.

\begin{definition}
\label{def:non_degenerate_surface_of_section}
The global surface of section $\Sigma$ is called {\it $\partial$-strong} if the lifted surface $\overline{\Sigma}\subset \overline{Y}$ is a global surface of section for the lifted flow $\overline{\phi}^t$, i.e. if $\overline{R}$ is transverse to $\overline{\Sigma}$ and all trajectories of $\overline{\phi}^t$ meet $\overline{\Sigma}$ infinitely often forward and backward in time.
\end{definition}
Since $\Sigma$ is a surface of section, $\overline{R}$ is clearly transverse to $\overline{\Sigma}$ in the interior of $\overline{Y}$. Moreover, all trajectories in the interior meet $\overline{\Sigma}$ forward and backward in time. Thus the condition for being $\partial$-strong is equivalent to requiring that $\overline{\Sigma}\cap\BT_\gamma$ is a surface of section for the projective linearized flow $d\phi^t$ on $\BT_\gamma$ for all boundary orbits $\gamma$ of $\Sigma$.

\begin{lem}
Suppose that $\Sigma\subset Y$ is a $\partial$-strong global surface of section. Then the first return time extends to a smooth function $\sigma:\Sigma\rightarrow \BR_{>0}$ and the first return map extends to a diffeomorphism $\phi:\Sigma\rightarrow\Sigma$. If $\Sigma'$ is a second $\partial$-strong global surface of section with the same boundary orbits, then any transfer map extends to a diffeomorphism $\psi:\Sigma\rightarrow\Sigma'$.
\end{lem}

\begin{proof}
We simply observe that $\Sigma$ being $\partial$-strong implies that the first return time and map of the lifted surface $\overline{\Sigma}$ are also defined on the boundary $\partial\overline{\Sigma}$ and smooth. The same argument applies to a transfer map $\psi$.
\end{proof}

In this paper we will be mainly concerned with {\it disk-like global surfaces of section}, i.e. the case that $\Sigma$ is diffeomorphic to the closed unit disk $\BD$. The manifold $Y$ is then necessarily diffeomorphic to $S^3$. Suppose that $\Sigma$ is a $\partial$-strong disk-like global surface of section. It will be useful to lift the first return map $\phi:\Sigma\rightarrow\Sigma$ to an element $\widetilde{\phi}\in\widetilde{\Diff}(\Sigma)$ of the universal cover of the space $\Diff^+(\Sigma)$ of orientation preserving diffeomorphisms of $\Sigma$. Such a lift depends on a choice of trivialization. Let $\pi:\overline{Y}\rightarrow Y$ be the blow-up of $Y$ at the boundary orbit of $\Sigma$. A {\it trivialization} of $\overline{Y}$ is a diffeomorphism $\tau:\BR/\BZ\times \Sigma\rightarrow \overline{Y}$ such that the composition
\begin{equation*}
\Sigma\cong 0\times\Sigma\subset\BR/\BZ\times \Sigma \overset{\tau}{\rightarrow} \overline{Y} \overset{\pi}{\rightarrow} Y
\end{equation*}
is simply the inclusion of $\Sigma$. Moreover, we require that $\iota_{\tau^*\overline{R}}dt>0$, where $t$ denotes the coordinate on $\BR/\BZ$. Since $\Diff^+(\Sigma)$ is connected, the space of trivializations is non-empty. Let $\MT$ denote the set of isotopy classes of trivializations of $\overline{Y}$. It is an affine space over $\pi_1(\Diff^+(\Sigma))\cong\BZ$. We exhibit an explicit bijection $\on{deg}:\MT\rightarrow\BZ$ as follows. Let $\tau$ be a trivialization and $p\in\partial\Sigma$ a point in the boundary. Then the degree $d$ of the map
\begin{equation*}
S^1\cong \BR/\BZ \rightarrow \partial\Sigma\cong S^1\qquad t\mapsto \pi(\tau(t,p))
\end{equation*}
is independent of the choice of $p$ and only depends on the isotopy class of $\tau$. Here $\partial\Sigma$ is oriented as the boundary of $\Sigma$. We define the {\it degree} of $\tau$ to be $\on{deg}(\tau)\coloneqq d$. Given a trivialization $\tau$, there is a natural lift $\widetilde{\phi}$ of $\phi$ to $\widetilde{\Diff}(\Sigma)$ constructed as follows. Let $X$ denote the unique (positive) rescaling of the pullback vector field $\tau^*\overline{R}$ on $\BR/\BZ\times \Sigma$ such that $\iota_Xdt=1$. The flow of $X$ yields an arc in $\Diff^+(\Sigma)$ from the identity to $\phi$. Clearly, the element $\widetilde{\phi}\in\widetilde{\Diff}(\Sigma)$ represented by this arc only depends on the isotopy class of $\tau$. Let us explain the dependence of the lift on the choice of trivialization. Consider integers $d$ and $e$ and let $\widetilde{\phi}_d$ and $\widetilde{\phi}_e$ denote the lifts of $\phi$ with respect to trivializations of degrees $d$ and $e$, respectively. Let $\widetilde{\rho}\in\widetilde{\Diff}(\Sigma)$ be one full positive rotation of $\Sigma$. Then the lifts $\widetilde{\phi}_d$ and $\widetilde{\phi}_e$ are related by the identity
\begin{equation}
\label{eq:relationship_lift_trivialization}
\widetilde{\rho}^{e-d}\circ \widetilde{\phi}_e = \widetilde{\phi}_d.
\end{equation}

Let us now specialize our discussion of global surfaces of section to Reeb flows. Let $\alpha$ be a contact form on $Y$ and let $R$ be the induced Reeb vector field. We abbreviate $\omega\coloneqq d\alpha|_\Sigma$. This is a closed $2$-form on $\Sigma$. It vanishes on the boundary $\partial\Sigma$ and is a positive area form in the interior $\interior(\Sigma)$. Note that by Stokes' theorem $\Sigma$ must possess at least one positive boundary orbit. In particular, if $\Sigma$ is a disk, then its boundary orbit must be positive. Let $\lambda$ denote the restriction of $\alpha$ to $\Sigma$. This defines a primitive of $\omega$. The first return time $\sigma$ and map $\phi$ satisfy the identity
\begin{equation}
\label{eq:first_return_time_identity}
\phi^*\lambda - \lambda = d\sigma.
\end{equation}
This implies that $\phi$ preserves the area form $\omega$. Similarly, one can show that a transfer map $\psi$ between two global surfaces of section $\Sigma$ and $\Sigma'$ with the same boundary orbits is area preserving.

\begin{lem}
\label{lem:action_equals_first_return_time}
Suppose that $\Sigma$ is a $\partial$-strong disk-like global surface of section and let $\widetilde{\phi}\in\widetilde{\Diff}(\Sigma,\omega)$ denote the lift of the first return map $\phi$ with respect to a trivialization of degree $0$. Then the action $\sigma_{\widetilde{\phi},\lambda}$ agrees with the first return time $\sigma$.
\end{lem}

\begin{proof}
We need to check that the first return time $\sigma$ satisfies \eqref{eq:action_on_disk_first_defining_identity} and \eqref{eq:action_on_disk_second_defining_identity}. The first identity is true by \eqref{eq:first_return_time_identity}. Let $(\phi_t)_{t\in [0,1]}$ be any arc in $\Diff^+(\Sigma)$ representing $\widetilde{\phi}$. Let $z\in\partial\Sigma$ be a point in the boundary and let $\gamma:[0,1]\rightarrow\partial\Sigma$ be the path defined by $\gamma(t)\coloneqq \phi_t(z)$. Let $\delta:[0,\sigma(z)]\rightarrow\overline{Y}$ be the trajectory of $\overline{\phi}^t$ starting at $z\in\partial \overline{\Sigma}\cong\partial\Sigma$. We can express $\sigma_{\widetilde{\phi},\lambda}(z)$ and $\sigma(z)$ as
\begin{equation*}
\sigma_{\widetilde{\phi},\lambda}(z) = \int_\gamma\lambda\qquad\text{and}\qquad \sigma(z) = \int_\delta \pi^*\alpha = \int_{\tau^{-1}\circ\delta} \tau^*\pi^*\alpha.
\end{equation*}
In order to see that these two numbers agree, we regard $\gamma$ as a path in $\BR/\BZ\times\Sigma$ via the inclusion $\Sigma\cong 0\times\Sigma\subset \BR/\BZ\times\Sigma$ and form the concatenation $\epsilon\coloneqq (\tau^{-1}\circ\delta) \# \overline{\gamma}$. This defines a loop in $\BR/\BZ\times\partial\Sigma$ which is homotopic to the loop $\BR/\BZ\times z$. The restriction $\beta\coloneqq (\tau^*\pi^*\alpha)|_{\BR/\BZ\times\partial\Sigma}$ is the pullback of the restriction of $\alpha$ to $\partial\Sigma$. Hence $\beta$ is a closed $1$-form. Since $\tau$ has degree $0$, the loop $\BR/\BZ\times z$ is mapped to a contractible loop in $\partial\Sigma$ by $\pi\circ\tau$. Thus the integral of $\beta$ over the loop $\BR/\BZ\times z$ vanishes. Since $\epsilon$ is homotopic to $\BR/\BZ\times z$, we obtain
\begin{equation*}
0 = \int_\epsilon \beta = \int_{\tau^{-1}\circ\delta} \tau^*\pi^*\alpha - \int_{\gamma} \lambda
\end{equation*}
where we have used that the restriction of $\beta$ to $0\times\partial\Sigma$ is equal to $\lambda$. This concludes the proof.
\end{proof}

\section{From disk-like surfaces of section to symplectic embeddings}
\label{section:from_disk_like_surfaces_of_section_to_symplectic_embeddings}

\subsection{Embedding results}
\label{subsection:embedding_results}

The following theorem says, roughly speaking, that if the boundary of a star-shaped domain $X\subset\BR^4$ admits a disk-like global surface of section of symplectic area $a$ such that the lift of the first return map with respect to a trivialization of degree $0$ can be generated by a positive Hamiltonian, then the domain $X$ can be symplectically embedded into the cylinder $Z(a)$. The second part of the theorem states that if the lift of the first return map with respect to a trivialization of degree $1$ can still be generated by a positive Hamiltonian, then the ball $B(a)$ embeds into $X$.

\begin{theorem}
\label{theorem:embedding_result}
Let $X\subset\BR^4$ be a star-shaped domain. Let $\Sigma\subset\partial X$ be a $\partial$-strong disk-like global surface of section of the natural Reeb flow on $\partial X$. Assume that the local first return map of a small disk transverse to the boundary orbit $\partial\Sigma$ is smoothly conjugated to a rotation. Set $\omega\coloneqq \omega_0|_{\Sigma}$ and let
\begin{equation*}
a\coloneqq \int_\Sigma \omega
\end{equation*}
be the symplectic area of the surface of section.
\begin{enumerate}
\item Let $\widetilde{\phi}_0\in \widetilde{\Diff}(\Sigma,\omega)$ be the lift of the first return map with respect to a trivialization of degree $0$. Suppose that there exists a Hamiltonian $H:\BR/\BZ\times\Sigma\rightarrow\BR$ with the following properties:
\begin{enumerate}
\item $H$ is strictly positive in the interior $\operatorname{int}(\Sigma)$ and vanishes on the boundary $\partial\Sigma$.
\item $H$ is autonomous in some neighbourhood of $\partial\Sigma$.
\item The Hamiltonian vector field $X_{H_t}$ defined by $\iota_{X_{H_t}}\omega = dH_t$ in the interior $\operatorname{int}(\Sigma)$ smoothly extends to the closed disk $\Sigma$ and is tangent to $\partial \Sigma$.
\item The arc $(\phi_H^t)_{t\in [0,1]}$ represents $\widetilde{\phi}_0$.
\end{enumerate}
\noindent Then $X\overset{s}{\hookrightarrow} Z(a)$.
\item Let $\widetilde{\phi}_1\in \widetilde{\Diff}(\Sigma,\omega)$ be the lift of the first return map with respect to a trivialization of degree $1$. Assume that there exists a Hamiltonian $G:\BR/\BZ\times\Sigma\rightarrow\BR$ satisfying properties (a)-(c) above such that the arc $(\phi_G^t)_{t\in [0,1]}$ represents $\widetilde{\phi}_1$. Then $B(a)\overset{s}{\hookrightarrow}X\overset{s}{\hookrightarrow} Z(a)$.
\end{enumerate}
\end{theorem}

Given a star-shaped domain $X$ and a disk-like surface of section $\Sigma\subset\partial X$, we do not know whether it is always possible to generate the lift $\widetilde{\phi}_0$ by a Hamiltonian which vanishes on the boundary $\partial\Sigma$ and is positive in the interior. The following Proposition says that we may always symplectically embed $X$ into a bigger domain satisfying the hypotheses of Theorem \ref{theorem:embedding_result}.

\begin{prop}
\label{prop:modify_hypersurface_such_that_return_map_is_generated_by_positive_hamiltonian}
Let $X\subset\BR^4$ be a star-shaped domain. Let $\Sigma\subset\partial X$ be a $\partial$-strong disk-like surface of section of the natural Reeb flow on $\partial X$. Then there exist a star-shaped domain $X'$ and a $\partial$-strong disk-like surface of section $\Sigma'$ of the natural Reeb flow on the boundary $\partial X'$ such that $\Sigma$ and $\Sigma'$ have the same symplectic areas, $X$ symplectically embeds into $X'$ and the tuple $(X',\Sigma')$ satisfies all hypotheses of the first assertion of Theorem \ref{theorem:embedding_result}.
\end{prop}

\subsection{Main construction}
\label{subsection:main_construction}

Given a $1$-periodic Hamiltonian
\begin{equation*}
H:\BR/\BZ\times\BD\rightarrow\BR
\end{equation*}
which is positive in the interior $\interior(\BD)$ and vanishes on the boundary $\partial\BD$, we construct a domain
\begin{equation*}
A(H)\subset\BC^2.
\end{equation*}
We show that the characteristic foliation on the boundary $\partial A(H)$ possesses a disk-like surface of section with first return map given by $\phi_H^1$.

\begin{lem}
\label{lem:characteristic_foliation_on_graph}
Let $(M,\omega)$ be a symplectic manifold. Let $\widetilde{M}\coloneqq \BR_s\times (\BR/\BZ)_t \times M$ denote time-energy extended phase space equipped with the symplectic form $\widetilde{\omega}\coloneqq ds\wedge dt + \omega$. Consider a periodic Hamiltonian
\begin{equation*}
H:\BR/\BZ\times M\rightarrow\BR
\end{equation*}
and let
\begin{equation*}
\Gamma(H)\coloneqq \{(H(t,p),t,p)\mid (t,p)\in\BR/\BZ\times M\} \subset\widetilde{M}
\end{equation*}
denote its graph inside time-energy extended phase space $\widetilde{M}$. Then the characteristic foliation on $\Gamma(H)$ induced by the symplectic form $\widetilde{\omega}$ is spanned by the vector field
\begin{equation*}
X_{H_t}(p) + \partial_t + \partial_tH(t,p)\cdot \partial_s.
\end{equation*}
\end{lem}

\begin{proof}
We define the autonomous Hamiltonian $\widetilde{H}$ on time-energy extended phase space by
\begin{equation*}
\widetilde{H}: \widetilde{M} \rightarrow\BR \qquad \widetilde{H}(s,t,p)\coloneqq H(t,p)-s.
\end{equation*}
The graph $\Gamma(H)$ is given by the regular level set $\widetilde{H}^{-1}(0)$. Thus the characteristic foliation on $\Gamma(H)$ is spanned by the restriction of the Hamiltonian vector field $X_{\widetilde{H}}$ to $\Gamma(H)$. We compute
\begin{equation*}
d\widetilde{H}(s,t,p) = d H_t(p) + \partial_t H(t,p)\cdot dt - ds = \iota_{X_{H_t} + \partial_t + \partial_tH\cdot \partial_s}(ds\wedge dt+\omega)
\end{equation*}
and conclude that
\begin{equation*}
X_{\widetilde{H}} = X_{H_t} + \partial_t + \partial_tH \cdot \partial_s.
\end{equation*}
\end{proof}

\begin{construction}
\label{construction:spun_hypersurface}
Consider $\BC$ equipped with the standard symplectic form $\omega_0 = dx\wedge dy$. Let
\begin{equation*}
\widetilde{\BC}\coloneqq \BR_s\times(\BR/\BZ)_t\times\BC \qquad \widetilde{\omega_0}\coloneqq ds\wedge dt+\omega_0
\end{equation*}
denote time-energy extended phase space and abbreviate
\begin{equation*}
\widetilde{\BC}_+ \coloneqq \BR_{\geq 0}\times\BR/\BZ\times\BC \qquad\text{and}\qquad \widetilde{\BC}_0 \coloneqq \{0\}\times\BR/\BZ\times\BC.
\end{equation*}
Consider the map
\begin{equation*}
\Phi : \widetilde{\BC}_+\rightarrow \BC^2 \qquad \Phi(s,t,z)\coloneqq \left(z\enspace,\enspace\sqrt{\frac{s}{\pi}}\cdot e^{2\pi i t}\right).
\end{equation*}
$\Phi$ restricts to a diffeomorphism between $\widetilde{\BC}_+\setminus\widetilde{\BC}_0$ and $\BC^2 \setminus (\BC\times 0)$. Moreover $\Phi^*\omega_0 = \widetilde{\omega_0}$. For $a>0$, let $B^2(a)\subset\BC$ denote the closed $2$-dimensional disk of area $a$. Let
\begin{equation*}
H:\BR/\BZ\times B^2(a) \rightarrow\BR
\end{equation*}
be a smooth function. Assume that:
\begin{enumerate}
\item $H$ is strictly positive in the interior $\operatorname{int}(B^2(a))$.
\item There exists a constant $C>0$ such that in some neighbourhood of $\partial B^2(a)$ the function $H$ is given by
\begin{equation}
\label{eq:special_form_of_H_near_boundary}
H(t,z) = C\cdot (a-\pi|z|^2).
\end{equation}
\end{enumerate}
Let
\begin{equation*}
\Gamma_-(H)\coloneqq \{(s,t,z)\in \widetilde{\BC}_+ \mid z \in B^2(a) \enspace \text{and}\enspace 0\leq s\leq H(t,z)\}
\end{equation*}
denote the subgraph of $H$. We define the subset $A(a,H)\subset \BC^2$ by
\begin{equation*}
A(a,H) \coloneqq \Phi(\Gamma_-(H)).
\end{equation*}
\end{construction}

\begin{lem}[Basic properties]
\label{lem:basic_properties_of_spun_hypersurfaces}
The set $A(a,H)\subset\BC^2$ defined in Construction \ref{construction:spun_hypersurface} satisfies the following basic properties:
\begin{enumerate}
\item $A(a,H)$ has smooth boundary and is diffeomorphic to the closed ball $D^4$.
\item $A(a,H)\subset Z(a)$
\item If $H(t,z)\geq a-\pi|z|^2$, then $B(a)\subset A(a,H)$.
\item The map
\begin{equation*}
f:B^2(a)\rightarrow \partial A(a,H) \quad z\mapsto \Phi(H(0,z),0,z)
\end{equation*}
is a parametrization of a disk-like surface of section of the characteristic foliation on $\partial A(a,H)$. We have $f^*\omega_0 = \omega_0$. Consider the lift $\widetilde{\phi}_0$ of the first return map with respect to a trivialization of degree $0$. We regard $\widetilde{\phi}_0$ as an element of $\widetilde{\Diff}(B^2(a),\omega_0)$ via the parametrization $f$. It is represented by $(\phi_H^t)_{t\in [0,1]}$.
\end{enumerate}
\end{lem}

\begin{remark}
The parametrization $f$ in item (4) of Lemma \ref{lem:basic_properties_of_spun_hypersurfaces} is only smooth in the interior $\interior(\BD)$. At the boundary $\partial\BD$, the radial derivative $\partial_rf$ blows up. Since $H$ has the special form \eqref{eq:special_form_of_H_near_boundary} and $\phi_H^t$ is a rotation near the boundary, this does not cause problems.
\end{remark}

\begin{proof}
Clearly, the boundary $\partial A(a,H)$ is smooth away from the circle $\partial B^2(a)\times\{0\}$. Near this circle, the Hamiltonian $H$ has the special form \eqref{eq:special_form_of_H_near_boundary}. Thus $\partial A(a,H)$ can be described by the equation
\begin{equation*}
C|z_1|^2 + |z_2|^2 = \frac{Ca}{\pi}
\end{equation*}
near $\partial B^2(a)\times\{0\}$. The solution set of this equation is the boundary of an ellipsoid and in particular smooth. Let $G$ denote the Hamiltonian which is given by formula \eqref{eq:special_form_of_H_near_boundary} on the entire disk $B^2(a)$. The set $A(a,G)$ is an ellipsoid and in particular diffeomorphic to the closed ball $D^4$. Clearly there exists a diffeomorphism
\begin{equation*}
\psi:\Gamma_-(H)\rightarrow\Gamma_-(G)
\end{equation*}
between the subgraphs of $H$ and $G$. In fact, we can choose $\psi$ to agree with the identity map on all points $(s,t,z)\in \Gamma_-(H)$ such that $z$ is close to the boundary $\partial B^2(a)$ or $s$ is close to $0$. If $\psi$ has these properties, then $\Phi\circ\psi\circ\Phi^{-1}$ defines a diffeomorphism between $A(a,H)$ and $A(a,G)$. Thus $A(a,H)$ is diffeomorphic to $D^4$. By construction, $A(a,H)$ is contained in the cylinder $Z(a)$. If $H(t,z) = a-\pi|z|^2$, then $A(a,H)$ is the $4$-dimensional ball $B(a)$ of area $a$. Since $H\leq G$ implies $A(a,H)\subset A(a,G)$, assertion (3) is an immediate consequence. In order to prove assertion (4), let us first observe that Lemma \ref{lem:characteristic_foliation_on_graph} implies that the map
\begin{equation*}
g:B^2(a)\rightarrow \Gamma(H)\subset \widetilde{\BC}_+ \qquad g(z)\coloneqq (H(0,z),0,z)
\end{equation*}
parametrizes a disk-like surface of section of the characteristic foliation on the graph $\Gamma(H)$. The first return map of this surface of section is given by $\phi_H^1$. The symplectomorphism $\Phi$ in Construction \ref{construction:spun_hypersurface} maps the characteristic foliation on the graph $\Gamma(H)$ to the characteristic foliation on $\partial A(a,H)$. Thus the first return map of the surface of section parametrized by $f$ is equal to $\phi_H^1$ as well. In order to show that $(\phi_H^t)_{t\in [0,1]}$ represents the correct lift, simply observe that the composition of the trivialization
\begin{equation*}
\tau:\BR/\BZ\times B^2(a)\rightarrow\Gamma(H)\qquad \tau(t,z)\coloneqq (H(t,z),t,z)
\end{equation*}
of $\Gamma(H)$ with $\Phi$ yields a trivialization of $\partial A(a,H)$ of degree $0$. This proves assertion (4).
\end{proof}

\subsection{Proof of Theorem \ref{theorem:embedding_result}}
\label{subsection:proof_of_embedding_result}

Throughout this section, we fix the setup of Theorem \ref{theorem:embedding_result}. We let $X\subset (\BR^4,\omega_0)$ denote a star-shaped domain and $\Sigma\subset\partial X$ a $\partial$-strong disk-like global surface of section of the natural Reeb flow on $\partial X$ induced by the restriction of the standard Liouville $1$-form $\lambda_0$ defined in \eqref{eq:liouville_vector_field_and_form}. We assume that the local first return map of a small disk transverse to the boundary orbit $\partial\Sigma$ is smoothly conjugated to a rotation. Let $a>0$ denote the symplectic area of $\Sigma$.\\
Our strategy, roughly speaking, is to show that if the degree $0$ lift $\widetilde{\phi}_0$ of the first return map can be generated by a Hamiltonian $H$ which is positive in the interior $\interior(\Sigma)$ and vanishes on the boundary $\partial \Sigma$, then $X$ is symplectomorphic to the domain $A(a,H)$ constructed in section \ref{subsection:main_construction}. We construct a symplectomorphism between $X$ and $A(a,H)$ in two steps. In Proposition \ref{prop:spun_hypersurface_isomorphic_starshaped_hypersurface} we show that there exists a diffeomorphism $\psi:\partial A(a,H)\rightarrow\partial X$ which pulls back $\omega_0|_{\partial X}$ to $\omega_0|_{\partial A(a,H)}$. Then we use a result of Gromov and McDuff (Theorem \ref{theorem:gromov_mcduff_theorem}) to extend $\psi$ to a symplectomorphism between $A(a,H)$ and $X$. This is done in Corollary \ref{cor:corollary_of_gromov_mcduff}.\\

We begin with the following auxiliary lemma on the existence of a convenient parametrization of a tubular neighbourhood of the boundary orbit.

\begin{lem}
\label{lem:neighbourhood_theorem_boundary_orbit_conj_to_rot}
Let $\epsilon>0$ be sufficiently small and let $\BD_\epsilon\subset\BC$ denote the disk of radius $\epsilon$. There exist a $\partial$-strong disk-like global surface of section $\Sigma'\subset \partial X$ with the same boundary orbit as $\Sigma$ and a parametrization $F:\BR/\BZ\times\BD_\epsilon\rightarrow \partial X$ of a tubular neighbourhood of $\partial\Sigma'$ such that the following is true.
\begin{equation}
\label{eq:neighbourhood_theorem_boundary_orbit_conj_to_rot_a}
F^{-1}(\Sigma') = \{(t,r e^{i\theta})\in \BR/\BZ \times \BD_\epsilon \mid 0\leq r\leq \epsilon\enspace\text{and}\enspace \theta=0 \}
\end{equation}
and
\begin{equation}
\label{eq:neighbourhood_theorem_boundary_orbit_conj_to_rot_b}
F^*\lambda_0 = \frac{1}{2}r^2\cdot d\theta + (a-\pi b r^2)\cdot dt
\end{equation}
where $b$ is a positive real number.
\end{lem}

\begin{proof}
Consider a small disk $D$ transverse to $\partial\Sigma$ whose local first return map is smoothly conjugated to a rotation. We may choose a parametrization $f:\BD_\epsilon\rightarrow D$ such that the local first return map, regarded as a diffeomorphism of $\BD_\epsilon$ via $f$, is a rotation of $\BD_\epsilon$. By an equivariant version of Moser's argument, after modifying the parametrization $f$ we may in addition assume that $f^*\omega_0=\omega_0$ where $\omega_0$ denotes the standard symplectic form on both $\BR^4$ and $\BD_\epsilon$. The primitives $f^*\lambda_0$ and $\lambda\coloneqq\frac{1}{2}r^2d\theta$ of the area form $\omega_0$ on $\BD_\epsilon$ differ by an exact $1$-form, i.e. $\lambda = f^*\lambda_0 + d\alpha$ for a smooth function $\alpha$ on $\BD_\epsilon$. We may normalize $\alpha$ such that $\alpha(0)=0$. We define
\begin{equation*}
f':\BD_\epsilon\rightarrow \partial X\quad f'(z)\coloneqq \phi^{\alpha(z)}(f(z))
\end{equation*}
where $\phi^t$ denotes the Reeb flow on $\partial X$. This parametrizes a small disk $D'$ transverse to $\partial\Sigma$. A direct computation shows that $f'^*\lambda_0 = f^*\lambda_0+d\alpha = \lambda$ and the local first return map is still a rotation of $\BD_\epsilon$. Let $\rho$ denote this rotation. Since $\rho^*\lambda=\lambda$, it follows from \eqref{eq:first_return_time_identity} that the first return time of $D'$ is constant and equal to $a$, the action of the orbit $\partial\Sigma$. Let us define the immersion
\begin{equation*}
F:\BR\times\BD_\epsilon\rightarrow\partial X \quad F(t,z)\coloneqq \phi^t(f'(z)).
\end{equation*}
We have $F^*\lambda_0 = dt+\lambda$ and $F$ is invariant under the diffeomorphism $\psi$ of $\BR\times\BD_\epsilon$ defined by $\psi(t,z) \coloneqq (t-a,\rho(z))$. Thus $F$ descends to a strict contactomorphism between the quotient $(\BR\times\BD_\epsilon)/\sim$ of $(\BR\times\BD_\epsilon,dt+\lambda)$ by the action of $\psi$ and a tubular neighbourhood of $\partial\Sigma$ in $\partial X$. It is a direct computation to check that we may choose a diffeomorphism $\tau:\BR/\BZ\times\BD_\epsilon\cong (\BR\times\BD_\epsilon)/\sim$ such that the contact form $dt+\lambda$ on $(\BR\times\BD_\epsilon)/\sim$ pulls back to a contact form on $\BR/\BZ\times\BD_\epsilon$ of the form \eqref{eq:neighbourhood_theorem_boundary_orbit_conj_to_rot_b} for some real number $b$. By slight abuse of notation, let $F:\BR/\BZ\times\BD_\epsilon\rightarrow\partial X$ denote the resulting parametrization of a tubular neighbourhood of $\partial \Sigma$. For appropriate choice of diffeomorphism $\tau$, the preimage $F^{-1}(\Sigma)$ is non-winding, i.e. isotopic to the annulus \eqref{eq:neighbourhood_theorem_boundary_orbit_conj_to_rot_a} in $\BR\times\BD_\epsilon$. In fact, since $\Sigma$ is a $\partial$-strong surface of section, it is easy to see that we may replace $\Sigma$ by an isotopic disk-like global surface of section $\Sigma'$ with the same boundary orbit such that $F^{-1}(\Sigma')$ is equal to the annulus \eqref{eq:neighbourhood_theorem_boundary_orbit_conj_to_rot_a}. It remains to show that the constant $b$ must be positive. This is a consequence of the fact that the boundary orbit $\partial\Sigma'$ is positive, i.e. the boundary orientation of $\Sigma'$ agrees with the orientation induced by the Reeb vector field. Since the Reeb vector field of \eqref{eq:neighbourhood_theorem_boundary_orbit_conj_to_rot_b} is simply given by $\frac{1}{a}(\partial_t+2\pi b\partial_\theta)$, this means that $b$ must be positive.
\end{proof}

The surface of section $\Sigma'$ constructed in Lemma \ref{lem:neighbourhood_theorem_boundary_orbit_conj_to_rot} still satisfies the assumptions of Theorem \ref{theorem:embedding_result} because we can simply use a transfer map $\psi:\Sigma\rightarrow\Sigma'$ transport the positive Hamiltonians $H$ and $G$ generating the lifts $\widetilde{\phi}_0$ and $\widetilde{\phi}_1$ of the first return map of $\Sigma$ to obtain positive Hamiltonians on $\Sigma'$. Thus we may replace $\Sigma$ by $\Sigma'$ and assume in addition that there exists a parametrization $F$ of a tubular neighbourhood of $\partial\Sigma$ satisfying \eqref{eq:neighbourhood_theorem_boundary_orbit_conj_to_rot_a} and \eqref{eq:neighbourhood_theorem_boundary_orbit_conj_to_rot_b}.\\
Our next step is to construct a special parametrization $f:B^2(a)\rightarrow\Sigma$ such that $f^*\omega = \omega_0$ where $\omega_0$ denotes the standard symplectic form on $B^2(a)$. For $re^{i\theta}$ near the boundary $\partial B^2(a)$ we define
\begin{equation*}
f (r e^{i\theta}) = F\left(\frac{\theta}{2\pi} \enspace,\enspace \sqrt{\frac{1}{b\pi}(a-\pi r^2)}\right).
\end{equation*}
A direct computation involving \eqref{eq:neighbourhood_theorem_boundary_orbit_conj_to_rot_b} shows that this pulls back $\omega$ to $\omega_0$. Since both $(B^2(a),\omega_0)$ and $(\Sigma,a)$ have area $a$, we can use a Moser type argument to extend $f$ to an area preserving map $f:(B^2(a),\omega_0)\rightarrow (\Sigma,\omega)$.\\
Consider the degree $0$ lift $\widetilde{\phi}_0$ of the first return map of $\Sigma$. Via the parametrization $f$ we can regard $\widetilde{\phi}_0$ as an element of $\widetilde{\Diff}(B^2(a),\omega_0)$. It follows from \eqref{eq:neighbourhood_theorem_boundary_orbit_conj_to_rot_b} and a short computation that $\widetilde{\phi}_0$ is a rotation by angle $2\pi/b$ near the boundary $\partial B^2(a)$. By the assumptions in the first assertion of Theorem \ref{theorem:embedding_result}, there exists a Hamiltonian $H:\BR/\BZ\times B^2(a)\rightarrow\BR$ which vanishes on the boundary, is positive in the interior, is autonomous near the boundary and generates $\widetilde{\phi}_0$. We argue that we may in addition assume that
\begin{equation}
\label{eq:boundary_behaviour_of_ham}
H(t,z) = \frac{1}{b}(a-\pi |z|^2)\qquad \text{for}\enspace z\enspace\text{sufficiently close to}\enspace \partial B^2(a).
\end{equation}
Since $H$ is autonomous near the boundary, it is invariant under $\phi_H^1$, which is a rotation by $2\pi/b$. If $b$ is irrational, then this implies that $H$ is invariant under arbitrary rotations and it is not hard to see that $H$ must in fact be given by \eqref{eq:boundary_behaviour_of_ham}. If $b$ is rational, then $H$ need not be invariant under arbitrary rotations near the boundary. We show that we may replace $H$ by a Hamiltonian which is rotation invariant. There exists a symplectomorphism $g$ of $(B^2(a),\omega_0)$ supported in a small neighbourhood of $\partial B^2(a)$ which commutes with the rotation by angle $2\pi/b$ such that the level sets of $H\circ g$ near the boundary are circles centred at the centre of $B^2(a)$. The time-$1$-map of $H\circ g$ is given by $g^{-1}\circ \phi_H^1 \circ g$. Away from a neighbourhood of $\partial B^2(a)$ this agrees with $\phi_H^1$ because $g$ is equal to the identity. Near the boundary, $\phi_H$ is a rotation by angle $2\pi/b$ and thus commutes with $g$. Hence $g^{-1}\circ \phi_H^1\circ g$ agrees with $\phi_H^1$ on all of $B^2(a)$. Now simply replace $H$ by $H\circ g$. Then $H$ is rotation invariant near the boundary and again it follows that it must be given by \eqref{eq:boundary_behaviour_of_ham}. 

\begin{prop}
\label{prop:spun_hypersurface_isomorphic_starshaped_hypersurface}
There exists a diffeomorphism $\psi : \partial A(a,H)\rightarrow \partial X$ such that $\psi^*(\omega_0|_{\partial X}) = \omega_0|_{\partial A(a,H)}$.
\end{prop}

\begin{proof}
After possibly shrinking $\epsilon$, we may define
\begin{equation*}
F':\BR/\BZ\times \BD_\epsilon \rightarrow \partial A(a,H)\quad F'(t,z)\coloneqq 
\left( \sqrt{\frac{a}{\pi}- b |z|^2}\cdot e^{2\pi i t}\enspace , \enspace z \right).
\end{equation*}
It follows from \eqref{eq:boundary_behaviour_of_ham} and the definition of $A(a,H)$ that the image of $F'$ is contained in $\partial A(a,H)$ for $\epsilon$ sufficiently small. Moreover, we define
\begin{equation*}
f':B^2(a)\rightarrow \partial A(a,H)\quad f'(z)\coloneqq \left(z,\sqrt{\frac{H(0,z)}{\pi}}\right).
\end{equation*}
A direct computation shows that the following two diffeomorphisms between submanifolds of $\partial A(a,H)$ and $\partial X$ agree on their overlap and pull back $\omega_0|_{\partial X}$ to $\omega_0|_{\partial A(a,H)}$:
\begin{equation*}
F\circ F'^{-1} : \operatorname{im}(F')\rightarrow \operatorname{im}(F)
\end{equation*}
and
\begin{equation*}
f\circ f'^{-1} : \operatorname{im}(f')\rightarrow \operatorname{im}(f)
\end{equation*}
On $\operatorname{im}(F')\cup \operatorname{im}(f')$ we may therefore define $\psi$ to agree with $F\circ F'^{-1}$ and $f\circ f'^{-1}$, respectively. Next, we explain how to extend to a diffeomorphism between $\partial A(a,H)$ and $\partial X$. We pull back the Reeb vector field $R$ on $\partial X$ via $F\circ F'^{-1}$ to obtain a vector field $R'$ on $\operatorname{im}(F')$ which is tangent to the characteristic foliation. We smoothly extend to a vector field on $\partial A(a,H)$, still denoted by $R'$, which is everywhere tangent to the characteristic foliation. The embedding $f'$ parametrizes a surface of section of the flow generated by $R'$. By Lemma \ref{lem:basic_properties_of_spun_hypersurfaces}, the lift of the first return map with respect to a trivialization of degree $0$ is represented by $(\phi_H^t)_{t\in [0,1]}$, which also represents the degree $0$ lift of the first return map of the surface of section of $\partial X$ parametrized by $f$. After replacing $R'$ by a positive scaling $\chi\cdot R'$ for a suitable smooth function $\chi:\partial A(a,H)\rightarrow \BR_{>0}$, we may assume that the first return times of the two surfaces of section $f$ and $f'$ agree as well. This allows us to extend $\psi$ to a diffeomorphism $\psi:\partial A(a,H)\rightarrow \partial X$ by requiring that $\psi$ intertwines the flows on $\partial A(a,H)$ and $\partial X$ generated by $R'$ and $R$, respectively. Set $\omega\coloneqq \psi^*(\omega_0|_{\partial X})$. We need to show that $\omega = \omega_0|_{\partial A(a,H)}$. By construction of $\psi$, the pull-back of the characteristic foliation on $\partial X$ via $\psi$ is equal to the characteristic foliation on $A(a,H)$. Thus $\omega_0$ and $\omega$ induce the  same characteristic foliations on $\partial A(a,H)$. Moreover, the restrictions of $\omega_0$ and $\omega$ to $\operatorname{im}(F')$ and $\operatorname{im}(f')$ agree. Cartan's formula implies
\begin{equation*}
\ML_{R'}\omega_0 = \iota_{R'}d\omega_0 + d\iota_{R'}\omega_0 = 0
\end{equation*}
where we use that $\omega_0$ is closed and $R'$ is contained in its kernel. Similarly, $\ML_{R'}\omega=0$. Let $p\in\partial A(a,H)\setminus \operatorname{im}(F')$ and let $v,w\in T_p\partial A(a,H)$. Our goal is to show that $\omega_0(v,w)=\omega(v,w)$. The trajectory of $R'$ through $p$ intersects the surface of section $\operatorname{im}(f')$ after finite time. Let $\widetilde{p}$ be the first intersection point. We transport $v$ and $w$ via the flow of $R'$ to obtain vectors $\widetilde{v},\widetilde{w}\in T_{\widetilde{p}}\partial A(a,H)$. Since $\ML_{R'}\omega_0$ and $\ML_{R'}\omega$ vanish, we have $\omega_0(v,w) = \omega_0(\widetilde{v},\widetilde{w})$ and $\omega(v,w) = \omega(\widetilde{v},\widetilde{w})$. After replacing $(p,v,w)$ by $(\widetilde{p},\widetilde{v},\widetilde{w})$, we can therefore assume w.l.o.g. that $p$ is contained in the surface of section $\operatorname{im}(f')$. In addition we may assume that $v$ and $w$ are tangent to $\operatorname{im}(f')$. Indeed, replacing $v$ and $w$ by their projections onto $T_p\operatorname{im}(f')$ along the characteristic foliation does not change $\omega_0(v,w)$ and $\omega(v,w)$ because the kernels of $\omega_0$ and $\omega$ are tangent to the characteristic foliation. Now we simply use that the restrictions $\omega_0|_{\operatorname{im}(f')}$ and $\omega|_{\operatorname{im}(f')}$ agree by construction of $\psi$.
\end{proof}

We recall the following well-known theorem due to Gromov and McDuff (see Theorem 9.4.2 in \cite{MS12}).

\begin{theorem}[Gromov-McDuff]
\label{theorem:gromov_mcduff_theorem}
Let $(M,\omega)$ be a connected symplectic $4$-manifold and $K\subset M$ be a compact subset such that the following holds.
\begin{enumerate}
\item There is no symplectically embedded $2$-sphere $S\subset M$ with self-intersection number $S\cdot S=-1$.
\item There exists a symplectomorphism $\psi:\BR^4\setminus V\rightarrow M\setminus K$, where $V\subset\BR^4$ is a star-shaped compact set.
\end{enumerate}
Then $(M,\omega)$ is symplectomorphic to $(\BR^4,\omega_0)$. Moreover, for every open neighbourhood $U\subset M$ of $K$, the symplectomorphism can be chosen equal to $\psi^{-1}$ on $M\setminus U$.
\end{theorem}

\begin{cor}
\label{cor:corollary_of_gromov_mcduff}
For $j\in\{1,2\}$, let $A_j\subset\BR^4$ be a compact submanifold diffeomorphic to the closed disk $D^4$. Assume that there exists a diffeomorphism $\psi:\partial A_1\rightarrow\partial A_2$ such that $\psi^*(\omega_0|_{\partial A_2}) = \omega_0|_{\partial A_1}$. Then the boundary $\partial A_1$ is of contact type if and only if $\partial A_2$ is of contact type. In this case, $\psi$ extends to a symplectomorphism
\begin{equation}
\psi: (A_1,\omega_0|_{A_1})\rightarrow (A_2,\omega_0|_{A_2}).
\end{equation}
\end{cor}

\begin{proof}
By the uniqueness part of the coisotropic neighbourhood theorem in \cite{Got82}, there exist open neighbourhoods $U_j$ of $\partial A_j$ such that $\psi$ extends to a symplectomorphism
\begin{equation*}
\psi: (U_1,\omega_0|_{U_1})\rightarrow (U_2,\omega_0|_{U_2}).
\end{equation*}
Being of contact type is a property that only depends on a small neighbourhood of a hypersurface. Thus $\partial A_1$ is if contact type if and only if $\partial A_2$ is. Suppose now that this is the case. After possibly shrinking $U_1$, we can find a Liouville vector field $Z_1$ defined on $U_1$ and transverse to $\partial A_1$. Let $\lambda_1$ denote the associated Liouville $1$-form defined by $\lambda_1 = \iota_{Z_1}\omega_0$. Let $Z_2$ and $\lambda_2$ denote the push-forwards via $\psi$. For $j\in\{1,2\}$ let $(\widehat{A_j},\omega_j)$ be the symplectic completion of $(A_j,\omega_0|_{A_j})$ obtained by attaching a cylindrical end using the Liouville vector field $Z_j$. Let $\widehat{U_j}$ denote the union of $U_j$ with the cylindrical end attached to $A_j$. Clearly, $\psi$ extends to a symplectomorphism
\begin{equation*}
\psi: (\widehat{U_1},\omega_1)\rightarrow (\widehat{U_2},\omega_2).
\end{equation*}
The contact manifold $(\partial A_1,\ker \lambda_1)$ is fillable. Hence it follows from Eliashberg's paper \cite{Eli91} that it is contactomorphic to $S^3$ equipped with the standard tight contact structure. Thus we can find a star-shaped domain $V\subset\BR^4$ and a strict contactomorphism from $(\partial V,\lambda_0)$ to $(\partial A_1,\lambda_1)$, where $\lambda_0$ denotes the standard Liouville $1$-form on $\BR^4$. There exists an open neighbourhood $U_0$ of $\BR^4\setminus V$ such that this strict contactomorphism extends to a symplectomorphism
\begin{equation*}
\phi: (U_0,\omega_0)\rightarrow (\widehat{U_1},\omega_1).
\end{equation*}
By Theorem \ref{theorem:gromov_mcduff_theorem}, there exists a symplectomorphism
\begin{equation*}
\phi_1:(\BR^4,\omega_0)\rightarrow (\widehat{A_1},\omega_1)
\end{equation*}
which agrees with $\phi$ on the complement of $V$. Similarly, applying Theorem \ref{theorem:gromov_mcduff_theorem} to the composition
\begin{equation*}
\psi\circ\phi:(U_0,\omega_0)\rightarrow (\widehat{U_2},\omega_2)
\end{equation*}
we obtain a symplectomorphism
\begin{equation*}
\phi_2:(\BR^4,\omega_0)\rightarrow (\widehat{A_2},\omega_2)
\end{equation*}
agreeing with $\psi\circ\phi$ on the complement of $V$. The composition $\phi_2\circ\phi_1^{-1}$ restricts to a symplectomorphism $(A_1,\omega_0)\rightarrow (A_2,\omega_0)$ extending the given diffeomorphism $\psi:\partial A_1\rightarrow \partial A_2$.
\end{proof}

\begin{proof}[Proof of Theorem \ref{theorem:embedding_result}]
We prove the first assertion. By Proposition \ref{prop:spun_hypersurface_isomorphic_starshaped_hypersurface} and Corollary \ref{cor:corollary_of_gromov_mcduff}, $X$ is symplectomorphic to $A(a,H)$. By the second item of Lemma \ref{lem:basic_properties_of_spun_hypersurfaces} we have $A(a,H)\subset Z(a)$. Thus $X\overset{s}{\hookrightarrow} Z(a)$.\\
We prove the second assertion. We define the Hamiltonian
\begin{equation*}
K:\BR/\BZ\times B^2(a) \rightarrow \BR\quad K(t,z)\coloneqq a-\pi|z|^2.
\end{equation*}
This Hamiltonian generates $\widetilde{\rho}$, the full positive rotation of $B^2(a)$ by angle $2\pi$. Our goal is to show that we may assume that the Hamiltonian $H:\BR/\BZ\times B^2(a)\rightarrow\BR$ generating $\widetilde{\phi}_0$ satisfies in addition $K\leq H$. Then it follows from Proposition \ref{prop:spun_hypersurface_isomorphic_starshaped_hypersurface}, Corollary \ref{cor:corollary_of_gromov_mcduff} and the third item in Lemma \ref{lem:basic_properties_of_spun_hypersurfaces} that $B(a){\hookrightarrow} X{\hookrightarrow} Z(a)$. We regard the lift $\widetilde{\phi}_1$ as an element of $\widetilde{\Diff}(B^2(a),\omega_0)$ via the parametrization $f$. It follows from \eqref{eq:relationship_lift_trivialization} that $\widetilde{\phi}_1 = \widetilde{\rho}^{-1}\circ \widetilde{\phi}_0$. In particular, $\widetilde{\phi}_1$ is a rotation by angle $2\pi(1/b-1)$ near the boundary. By assumption, there exists a Hamiltonian $G:\BR/\BZ\times B^2(a)\rightarrow\BR$ generating $\widetilde{\phi}_1$ which is positive in the interior, vanishes on the boundary and is autonomous near the boundary. It follows as in the case of the Hamiltonian $H$ generating $\widetilde{\phi}_0$ that we can assume that $G$ is given by
\begin{equation*}
G(z,t) = (\frac{1}{b}-1)\cdot (a-\pi |z|^2)
\end{equation*}
near the boundary. Define
\begin{equation*}
H_t\coloneqq (K\#G)_t = K_t + G_t\circ (\phi_K^t)^{-1}.
\end{equation*}
This defines a $1$-periodic Hamiltonian generating $\widetilde{\phi}_0$ which has the special form \eqref{eq:boundary_behaviour_of_ham} near $\partial B^2(a)$ and is bounded below by $K_t(z)=a-\pi|z|^2$.
\end{proof}

\subsection{Proof of Proposition \ref{prop:modify_hypersurface_such_that_return_map_is_generated_by_positive_hamiltonian}}
\label{subsection:proof_of_modification_result}

We begin with some auxiliary lemmas. The first lemma concerns {\it positive paths} in the linear symplectic group (see Lalonde-McDuff \cite{LaMD97}). Let $\on{Sp}(2n)$ denote the group of all linear symplectomorphisms of $\BR^{2n}$. Moreover, let $J_0$ denote the matrix representing the standard complex structure on $\BR^{2n}\cong\BC^n$. We recall that, for every path $S:[0,1]\rightarrow\BR^{2n\times 2n}$ of symmetric matrices, the solution to the initial value problem
\begin{equation*}
\dot{\Phi}(t) = J_0 S(t)\Phi(t)\qquad \text{and}\qquad \Phi(0)=\on{id}
\end{equation*}
is an arc $\Phi$ in $\on{Sp}(2n)$ starting at the identity. Conversely, every such arc $\Phi$ arises this way for a unique path $S$ of symmetric matrices. A path $\Phi$ in $\on{Sp}(2n)$ is called {\it positive} if $S(t)$ is positive definite for all $t$. Lemma \ref{lem:positive_paths_rotation_number} below characterizes the elements $\widetilde{\Phi}$ of the universal cover $\widetilde{\on{Sp}}(2)$ which can be represented by positive arcs. This characterization is given in terms of the rotation number
\begin{equation*}
\rho : \widetilde{\on{Sp}}(2)\rightarrow\BR.
\end{equation*}
We give two equivalent definitions of $\rho$. The first one involves eigenvalues of symplectic matrices. We begin by defining a function $\overline{\rho}:\on{Sp}(2) \rightarrow\BR/\BZ$. The complex eigenvalues of a matrix $A\in\on{Sp}(2)$ are either given by $\lambda,\lambda^{-1}$ for $\lambda\in\BR\setminus\{0\}$ or by $e^{\pm 2\pi i\theta}$ for $\theta\in (0,1/2)$. In the former case, we define
\begin{equation*}
\overline{\rho}(A)\coloneqq \begin{cases} 0 & \text{if}\enspace \lambda>0 \\
1/2 & \text{if}\enspace\lambda<0 \end{cases}
\end{equation*}
In the latter case, we fix an arbitrary vector $v\in\BR^2\setminus\{0\}$ and define
\begin{equation*}
\overline{\rho}(A)\coloneqq \begin{cases} \theta & \text{if}\enspace \langle J_0v,Av\rangle>0 \\
-\theta & \text{if}\enspace \langle J_0v,Av \rangle <0\end{cases}
\end{equation*}
The rotation number $\rho$ is defined to be the unique lift of $\overline{\rho}$ to the universal cover $\widetilde{\on{Sp}}(2)$ satisfying $\rho(\on{id})=0$.\\
For our second definition of $\rho$, we fix $v\in\BR^2\setminus\{0\}$ and define $\overline{\rho}_v:\on{Sp}(2)\rightarrow\BR/\BZ$ to be the auxiliary function characterized by
\begin{equation*}
Av \in \BR_{>0}\cdot e^{2\pi i \overline{\rho}_v(A)}\cdot v
\end{equation*}
We let $\rho_v:\widetilde{\on{Sp}}(2)\rightarrow\BR$ denote the unique lift of $\overline{\rho}_v$ to the universal cover satisfying $\rho_v(\on{id})=0$ and define
\begin{equation*}
\rho(\widetilde{\Phi})\coloneqq \lim_{k\rightarrow\infty} \frac{\rho_v(\widetilde{\Phi}^n)}{n}.
\end{equation*}
We refer to \cite[appendix A]{CH21} for a proof that these two definitions of $\rho$ coincide.

\begin{lem}
\label{lem:positive_paths_rotation_number}
Let $\widetilde{\Phi}\in\widetilde{\on{Sp}}(2)$. Then $\widetilde{\Phi}$ can be represented by a positive arc in $\on{Sp}(2)$ if and only if the rotation number $\rho(\widetilde{\Phi})$ is strictly positive.
\end{lem}

\begin{proof}
Suppose that $\widetilde{\Phi}$ is represented by a positive arc $(\Phi(t))_{t\in [0,1]}$ in $\on{Sp}(2)$ starting at the identity. Let $S(t)$ denote the path of positive definite matrices generating $\Phi(t)$. We may choose $\epsilon>0$ such that $\langle z,S(t)z \rangle\geq\epsilon$ for all $t\in [0,1]$ and all unit vectors $z\in \BR^2$. Fix $v\in\BR^2\setminus\{0\}$. A direct computation shows that
\begin{equation*}
\frac{d}{dt}\overline{\rho}_v(\Phi(t)) = \frac{\langle \Phi(t)v,S(t)\Phi(t)v \rangle}{|\Phi(t)v|^2} \geq \epsilon.
\end{equation*}
It is immediate from our second definition of $\rho$ that this implies that $\rho(\widetilde{\Phi})\geq \epsilon>0$.\\
Conversely, suppose that $\rho(\widetilde{\Phi})>0$. Our goal is to construct a positive arc in $\on{Sp}(2)$ representing $\widetilde{\Phi}$. We claim that we may reduce ourselves to the case $\rho(\widetilde{\Phi})\in (0,1)$. Indeed, for $\tau\in\BR$, let $\widetilde{R}_\tau\in\widetilde{\on{Sp}}(2)$ denote the element represented by the arc $(e^{it\tau})_{t\in [0,1]}$ in $\on{U}(1)\subset\on{Sp}(2)$. Consider the function $\tau\mapsto\rho((\widetilde{R}_\tau)^{-1}\circ \widetilde{\Phi})$. This function is continuous and it is clear from our second definition of $\rho$ that it is decreasing and that it diverges to $-\infty$ as $\tau\rightarrow +\infty$. Thus we may pick $\tau>0$ such that $\rho((\widetilde{R}_\tau)^{-1}\circ \widetilde{\Phi})\in (0,1)$. Since $\widetilde{R}_\tau$ is represented by a positive arc by definition, it suffices to show that the same is true for $(\widetilde{R}_\tau)^{-1}\circ \widetilde{\Phi}$. Hence, after replacing $\widetilde{\Phi}$ by $(\widetilde{R}_\tau)^{-1}\circ \widetilde{\Phi}$, we may assume w.l.o.g. that $\rho(\widetilde{\Phi})\in (0,1)$.\\
Let $\Phi\in\on{Sp}(2)$ denote the projection of $\widetilde{\Phi}$ to $\on{Sp}(2)$. Since $\rho(\widetilde{\Phi})\notin\BZ$, the spectrum of $\Phi$ does not contain positive real numbers. Thus Theorem 1.2 in \cite{LaMD97} implies that there exists a positive arc $(\Phi(t))_{t\in [0,1]}$ in $\on{Sp}(2)$ starting at the identity and ending at $\Phi(1)=\Phi$ such that $\Phi(t)$ does not have any positive real eigenvalues for any $t>0$. Let $[(\Phi(t))_t]$ denote the element of the universal cover represented by the arc $(\Phi(t))_t$. Our goal is to show that $\widetilde{\Phi}=[(\Phi(t))_t]$. Since the projections of these to elements to $\on{Sp}(2)$ agree, it is enough to show that the rotation numbers $\rho(\widetilde{\Phi})$ and $\rho([(\Phi(t))_t])$ coincide. These rotation numbers must agree in $\BR/\BZ$, so it is actually enough to show that $\rho([(\Phi(t))_t])\in (0,1)$. Positivity of the arc $(\Phi(t))_t$ implies that $\rho([(\Phi(t))_t])>0$. This follows from the implication of Lemma \ref{lem:positive_paths_rotation_number} already proved above. Since the spectrum of $\Phi(t)$ does not contain positive real numbers for any $t>0$, we have $\overline{\rho}(\Phi(t))\neq 0$ for all $t>0$. We deduce that $\rho([(\Phi(t))_t])<1$. This concludes our proof that $\widetilde{\Phi}$ can be represented by a positive arc.
\end{proof}

\begin{lem}
\label{lem:neighbourhood_theorem_boundary_orbit}
Let $X\subset\BR^4$ be a star-shaped domain and let $\Sigma\subset\partial X$ be a $\partial$-strong disk-like global surface of section. Let $\epsilon>0$ be sufficiently small and let $\BD_\epsilon\subset\BC$ denote the disk of radius $\epsilon$. Then there exist a $\partial$-strong disk-like global surface of section $\Sigma'$ with the same boundary orbit as $\Sigma$ and a parametrization $F:\BR/\BZ\times\BD_\epsilon\rightarrow \partial X$ of a tubular neighbourhood of $\Sigma$ such that the following is true:
\begin{equation}
\label{eq:neighbourhood_theorem_boundary_orbit_a}
F^{-1}(\Sigma') = \{(t,re^{i\theta})\in \BR/\BZ\times\BD_\epsilon \mid 0\leq r\leq\epsilon\enspace\text{and}\enspace \theta=0\}
\end{equation}
and
\begin{equation}
\label{eq:neighbourhood_theorem_boundary_orbit_b}
F^*\lambda_0 = \frac{1}{2}r^2d\theta + H dt
\end{equation}
where $H:\BR/\BZ\times\BD_\epsilon\rightarrow\BR$ is a Hamiltonian such that $H_t(0)=\int_{\partial\Sigma} \lambda_0$ and the differential $dH_t(0)$ vanishes. Moreover, the Hessian $\nabla^2H_t(0)$ is negative definite.
\end{lem}

\begin{proof}
Let $a\coloneqq\int_{\partial\Sigma}\lambda_0$ be the action of the orbit $\partial\Sigma$. Let $\xi$ denote the contact structure on $\partial X$. Let $\tau:\xi|_{\partial\Sigma}\cong\BR^2$ be a symplectic trivialization of $\xi|_{\partial\Sigma}$ with the property that $\Sigma$ does not wind with respect to $\tau$. Via the trivialization $\tau$, the linearized Reeb flow $d\phi^t$ along $\partial\Sigma$ induces an arc $\Phi:[0,a]\rightarrow\on{Sp}(2)$ representing an element $\widetilde{\Phi}$ of the universal cover $\widetilde{\on{Sp}}(2)$. Since $\partial \Sigma$ is a positive boundary orbit of $\Sigma$, the linearized Reeb flow along $\partial\Sigma$ winds non-negatively with respect to the surface of section $\Sigma$, i.e. $\rho(\widetilde{\Phi})\geq 0$. The fact that $\Sigma$ is $\partial$-strong actually implies that $\rho(\widetilde{\Phi})$ is strictly positive. Hence it follows from Lemma \ref{lem:positive_paths_rotation_number} that $\widetilde{\Phi}$ is represented by a positive arc. We may therefore choose a loop $(S(t))_{t\in\BR/a\BZ}$ of symmetric positive definite matrices generating an arc $\Psi:[0,a]\rightarrow\on{Sp}(2)$ which represents $\widetilde{\Phi}$. After replacing $\tau$ by an isotopic trivialization, we can assume that the arc $\Psi$ is induced by the linearized Reeb flow. Using the trivialization $\tau$, we can choose a parametrization $F:\BR/\BZ\times\BD_\epsilon\rightarrow\partial X$ of a tubular neighbourhood of $\partial\Sigma$ such that
\begin{enumerate}
\item The pullback $F^*\lambda_0$ is given by $a\cdot dt$ on the circle $\BR/\BZ\times 0$.
\item The pullback $F^*\omega_0$ agrees with $\omega_0$ on the the circle $\BR/\BZ\times 0$. Here $\omega_0$ denotes the standard symplectic form on $\BR^4$ and also the $2$-form on $\BR/\BZ\times\BD_\epsilon$ whose restriction to fibres $t\times\BD_\epsilon$ agrees with the standard symplectic form on $\BD_\epsilon$ and which vanishes on the vector field $\partial_t$.
\item The linearized Reeb flow of $F^*\lambda_0$ along the orbit $\BR/\BZ\times 0$ is given by the arc $\Psi$.
\end{enumerate}
The remaining argument proceeds exactly as the proof of Lemma 5.2 in \cite{HM15}. The result is a modification of the parametrization $F$ such that properties (1)-(3) above still hold and such that $F^*\lambda_0$ is of the form \eqref{eq:neighbourhood_theorem_boundary_orbit_b} for some Hamiltonian $H$ which satisfies $H_t(0)=a$ and $dH_t(0)=0$. It follows from the fact that the linearized flow $\Psi$ is generated by symmetric positive definite matrices and our sign conventions that $\nabla^2H_t(0)$ is negative definite. Finally, since $\Sigma$ does not wind with respect to the parametrization $F$, we can achieve \eqref{eq:neighbourhood_theorem_boundary_orbit_a} by isotoping $\Sigma$ and possibly shrinking the tubular neighbourhood.
\end{proof}

\begin{lem}
\label{lem:embedding_time_engergy_extended_phase_space_into_symplectization}
Let $D$ be a closed $2$-dimensional disk. Let $\lambda$ be a Liouville $1$-form on $D$, i.e. $\omega\coloneqq d\lambda$ is a symplectic form and the Liouville vector field $W$ characterized by $\iota_W\omega = \lambda$ is transverse to $\partial D$. Let $I\subset\BR$ be a closed interval and endow $I\times D$ with the contact form $dt+\lambda$. Here $t$ denotes the coordinate on $I$. Set $M\coloneqq \BR_s\times I\times D$. We can regard $M$ as the symplectization of $I\times D$ and equip it with the symplectic form $\omega_M\coloneqq d(e^s(dt+\lambda))$. We can also regard $M$ as time-energy extended phase space of $D$ and endow it with the symplectic form $\widetilde{\omega}\coloneqq ds\wedge dt+\omega$. We abbreviate $M_+\coloneqq \BR_{\geq 0}\times I\times D$ and $M_0\coloneqq 0\times I\times D$. There exists a symplectic embedding
\begin{equation*}
G: (M_+,\widetilde{\omega})\rightarrow (M_+,\omega_M)
\end{equation*}
which restricts to the identity on $M_0$.
\end{lem}

\begin{proof}
We assume w.l.o.g. that $0$ is contained in the interior of $I$. We define the vector field $Y$ on $M$ by
\begin{equation*}
Y\coloneqq \partial_s - W - t\cdot \partial t.
\end{equation*}
Since $W$ is outward pointing at $\partial D$ and $t\cdot\partial_t$ is outward pointing at $\partial I$, the flow of $Y$ is defined for all positive times. We define $G$ to be the embedding which is uniquely determined by requiring $G(0,t,z)= (0,t,z)$ for all $(t,z)\in I\times D$ and $G^*Y = \partial_s$. Let us check that
\begin{equation*}
G^* \omega_M = \widetilde{\omega}.
\end{equation*}
We compute
\begin{IEEEeqnarray}{rCl}
\ML_Y \omega_M & = & \ML_Y d(e^s(\lambda + dt)) \nonumber\\
& = & d ((\ML_Y e^s)(\lambda + dt) + e^s\ML_Y(\lambda + dt)) \nonumber\\
& = & d (e^s(\lambda + dt - d\iota_W\lambda - \iota_W d\lambda - d \iota_{t\partial_t}dt)) \nonumber\\
& = & 0.\nonumber
\end{IEEEeqnarray}
Clearly $\ML_{\partial_s}\widetilde{\omega}=0$ as well. Thus is suffices to check that $G^*\omega_M$ and $\widetilde{\omega}$ agree on the set $M_0$. This is equivalent to showing that the pullbacks to $M_0$ of $G^*\omega_M$ and $\widetilde{\omega}$ and of $\iota_{\partial_s}G^*\omega_M$ and $\iota_{\partial_s}\widetilde{\omega}$ agree. Since the restriction of $G$ to $M_0$ is the identity, the pullback of $G^*\omega_M$ to $M_0$ is equal to $\omega$. This agrees with the pullback of $\widetilde{\omega}$. We compute
\begin{IEEEeqnarray}{rCl}
\iota_{\partial_s}G^*\omega_M & = & \iota_{\partial_s}G^* d(e^s(\lambda + dt))\nonumber\\
& = & G^* \iota_Y e^s(\omega + ds\wedge\lambda + ds\wedge dt) \nonumber\\
& = & G^* e^s (-\lambda + \lambda + dt + tds) \nonumber\\
& = & G^* e^s (dt + tds).\nonumber
\end{IEEEeqnarray}
The pullback of this form to $M_0$ is simply $dt$. This agrees with the pullback of $\iota_s\widetilde{\omega}$. This establishes $G^*\omega_M=\widetilde{\omega}$
\end{proof}

\begin{proof}[Proof of Proposition \ref{prop:modify_hypersurface_such_that_return_map_is_generated_by_positive_hamiltonian}]
Our first step is to construct a star-shaped domain $X'\subset\BR^4$ which contains $X$, agrees with $X$ outside a small neighbourhood of $\partial\Sigma$ and has a $\partial$-strong disk-like surface of section $\Sigma'\subset\partial X'$ of the same symplectic area as $\Sigma$ such that the local first return map of a small disk transverse to the boundary orbit $\partial \Sigma'$ is smoothly conjugated to a rotation. We apply Lemma \ref{lem:neighbourhood_theorem_boundary_orbit}. After replacing $\Sigma$ by a disk-like global surface of section with the same boundary orbit, we may choose a parametrization $F:\BR/\BZ\times \BD_\epsilon \rightarrow \partial X$ of a tubular neighbourhood of $\partial\Sigma$ satisfying \eqref{eq:neighbourhood_theorem_boundary_orbit_a} and \eqref{eq:neighbourhood_theorem_boundary_orbit_b}. Consider time-energy extended phase space $\BR_s\times (\BR/\BZ)_t\times \BD_\epsilon$ equipped with the symplectic form $\widetilde{\omega_0}\coloneqq \omega_0 + ds\wedge dt$ where $\omega_0$ denotes the standard symplectic form on $\BD_\epsilon$. The pullback of $\widetilde{\omega_0}$ via the parametrization
\begin{equation*}
G:\BR/\BZ\times \BD_\epsilon \rightarrow \BR\times\BR/\BZ\times\BD_\epsilon \quad G(t,z)\coloneqq (H_t(z),t,z)
\end{equation*}
of the graph $\Gamma(H)$ is given by $\omega_0 + dH\wedge dt$. It follows from  \eqref{eq:neighbourhood_theorem_boundary_orbit_b} that this agrees with $F^*\omega_0$ where, by slight abuse of notation, $\omega_0$ also denotes the standard symplectic form on $\BR^4$. We set $\psi\coloneqq F\circ G^{-1}$. This defines a diffeomorphism from $\Gamma(H)$ to $\operatorname{im}(F)$ satisfying $\psi^*\omega_0 = \widetilde{\omega_0}|_{\Gamma(H)}$. By the coisotropic neighbourhood theorem in \cite{Got82}, we may extend $\psi$ to a symplectomorphism defined on some open neighbourhood $U$ of the graph $\Gamma(H)$. Note that the push-forward of the vector field $\partial_s$ via $\psi$ is transverse to $\partial X$ and outward pointing. Let $H'$ be a $C^1$-small perturbation of $H$ supported in a small neighbourhood of $\BR/\BZ\times 0$ with the following properties:
\begin{enumerate}
\item $H'\geq H$
\item Near $\BR/\BZ\times 0$ the Hamiltonian $H'$ is given by
\begin{equation}
\label{eq:modified_H_in_prop_modify_hypersurface_such_that_return_map_is_generated_by_positive_hamiltonian}
H'_t(z) = H_t(0) - b |z|^2
\end{equation}
for some positive constant $b$.
\end{enumerate}
This is possible because the Hessian $\nabla^2H_t(0)$ is negative definite. We define $X'$ to be the star-shaped domain which agrees with $X$ outside $\operatorname{im}(\psi)$ and which satisfies
\begin{equation*}
\partial X'\cap \operatorname{im}(\psi) = \psi(\Gamma(H')).
\end{equation*}
The inequality $H'\geq H$ implies that $X$ is contained on $X'$. By \eqref{eq:modified_H_in_prop_modify_hypersurface_such_that_return_map_is_generated_by_positive_hamiltonian}, the orbit $\partial\Sigma$ on $\partial X$ also is an orbit on $\partial X'$. If follows from our construction of $X'$ that there exists an embedding $F':\BR/\BZ\times \BD_\epsilon \rightarrow\partial X'$ which agrees with $F$ near $\BR/\BZ\times \partial\BD_\epsilon$ such that
\begin{equation*}
F'^*\omega_0 = \omega_0 + dH'\wedge dt.
\end{equation*}
We define $\Sigma'\subset\partial X'$ to agree with $\Sigma$ outside $\operatorname{im}(F')$ and to be given by
\begin{equation*}
\{F'(t,r)\mid 0\leq r\leq \epsilon\enspace \text{and}\enspace t\in\BR/\BZ\}
\end{equation*}
inside $\operatorname{im}(F')$. This clearly defines a disk-like surface of section of the Reeb flow on $\partial X'$. Its symplectic area agrees with the symplectic area of $\Sigma$ because $\partial\Sigma'=\partial\Sigma$. It follows from \eqref{eq:modified_H_in_prop_modify_hypersurface_such_that_return_map_is_generated_by_positive_hamiltonian} that the local return map of the orbit $\partial\Sigma'$ is smoothly conjugated to a rotation.\\
Let us replace $(X,\Sigma)$ by $(X',\Sigma')$. We may choose a smooth parametrization $f:\BD\rightarrow\Sigma$ such that near the boundary $\partial\BD$ the pullback $\omega\coloneqq f^*\omega_0$ is rotation invariant and the first return map $\phi$ is a rotation. Let $\widetilde{\phi}_0$ denote the lift of $\phi$ with respect to a degree $0$ trivialization. Note that by positivity of the constant $b$ in \eqref{eq:modified_H_in_prop_modify_hypersurface_such_that_return_map_is_generated_by_positive_hamiltonian}, the rotation angle of $\widetilde{\phi}_0$ near $\partial\BD$ must be strictly positive. Thus $\widetilde{\phi}_0$ can be generated by a Hamiltonian $H$ which vanishes on the boundary and is autonomous and strictly positive in some neighbourhood of $\partial\BD$. We do not know whether we can choose $H$ to be strictly positive everywhere in the interior. Our strategy is to construct a domain $X'$ which contains $X$ and agrees with $X$ outside an open neighbourhood of $\Sigma$ such that the degree $0$ lift of the first return map of the Reeb flow on $\partial X'$ can be generated by a Hamiltonian $H'$ which agrees with $H$ near $\partial\BD$ and is strictly positive in the interior. Let $K:\BR/\BZ\times\BD\rightarrow\BR$ be a non-negative Hamiltonian which is compactly supported in the interior $\interior(\BR/\BZ\times\BD)$. The composition $\widetilde{\phi}_0\circ (\phi_K^t)_{t\in [0,1]}\in \widetilde{\Diff}(\BD,\omega)$ is generated by the Hamiltonian
\begin{equation}
\label{eq:composite_hamiltonian_in_prop_modify_hypersurface_such_that_return_map_is_generated_by_positive_hamiltonian}
(H\#K)_t\coloneqq H_t + K\circ (\phi_H^t)^{-1}.
\end{equation}
If $K$ is sufficiently large, then this Hamiltonian is strictly positive. The Hamiltonian $H\#K$ need not be $1$-periodic. This can be remedied as follows: First note that we can assume that $H_t$ is strictly positive in the interior for $t$ in some open neighbourhood of $0\in\BR/\BZ$. Then it suffices to consider $K$ with the property that $K_t$ vanishes for $t$ near $0$. In this situation $H\#K$ smoothly extends to a $1$-periodic Hamiltonian which is strictly positive in the interior and autonomous and rotation invariant near the boundary. Our goal is to construct $X'$ such that the degree $0$ lift of the first return map is given by $\widetilde{\phi}_0\circ (\phi_K^t)_{t\in [0,1]}$.\\
Let $D\subset\operatorname{int}(\BD)$ be a closed disk centred at $0$ and containing the support of $K$. After possibly increasing the radius of $D$, we can assume that $\lambda\coloneqq f^*\lambda_0$ is a Liouville $1$-form on $D$ whose associated Liouville vector field is transverse to $\partial D$. Let $\epsilon>0$ be sufficiently small and let
\begin{equation*}
F:[0,\epsilon]\times D\rightarrow \partial X
\end{equation*}
be the unique embedding such that $F(0,z) = f(z)$ and such that the pullback of the Reeb vector field $R$ on $\partial X$ via $F$ is given by $\partial_t$ where $t$ is the coordinate on $[0,\epsilon]$. This implies that $F^*\lambda_0 = \lambda + dt$. Let $M\coloneqq \BR\times [0,\epsilon]\times D$ be the symplectization of $([0,\epsilon]\times D,\lambda + dt)$ equipped with the symplectic form $\omega_M\coloneqq d(e^s(\lambda + dt))$. Using the radial Liouville vector field $Z_0$ on $\BR^4$, we can extend $F$ to a symplectic embedding $F:(M,\omega_M)\rightarrow (\BR^4,\omega_0)$ mapping $\partial_s$ to $Z_0$. Our modification of $X$ will be supported inside the image of $F$.\\
We apply Lemma \ref{lem:embedding_time_engergy_extended_phase_space_into_symplectization} and obtain a symplectic embedding $G:(M_+,\widetilde{\omega})\rightarrow (M_+,\omega_M)$. Let us reparametrize $K_t$ such that it is compactly supported in the time interval $(0,\epsilon)$ and still generates the same time-$1$-flow. Let $\Gamma_-(K)$ denote the subgraph of $K$, i.e. the set
\begin{equation*}
\Gamma_-(K) = \{(s,t,z)\in M_+\mid s\leq K_t(z)\}.
\end{equation*}
We define $X'$ to be the union
\begin{equation*}
X'\coloneqq X \cup F(G(\Gamma_-(K))).
\end{equation*}
$\Sigma$ is a disk-like global surface of section of the characteristic foliation on $\partial X'$ and it follows from Lemma \ref{lem:characteristic_foliation_on_graph} that the degree $0$ lift of the first return map is given by $\widetilde{\phi'}_0=\widetilde{\phi}_0\circ (\phi_K^t)_{t\in [0,1]}$. By our construction of $K$, the first return map $\widetilde{\phi'}_0$ can be generated by a Hamiltonian satisfying the hypotheses in Theorem \ref{theorem:embedding_result}.\\
The domain $X'$ might not be star-shaped. We argue that $X'$ must be symplectomorphic to a star-shaped domain for appropriate choice of Hamiltonian $K$. Our strategy is to define a contact form $\beta$ on $\partial X'$ such that $d\beta = \omega_0|_{X'}$. This contact form must be tight and there exists a star-shaped domain $X''\subset\BR^4$ such that $(\partial X',\beta)$ is strictly contactomorphic to $(\partial X'',\lambda_0|_{\partial X''})$. Corollary \ref{cor:corollary_of_gromov_mcduff} then implies that $X''$ is symplectomorphic to $X'$. On the complement of the image of $F:M\rightarrow\BR^4$ we simply define $\beta\coloneqq \lambda_0|_{\partial X'}$. We parametrize the intersection $\on{im}(F)\cap\partial X'$ via
\begin{equation*}
F':[0,\epsilon]\times D \rightarrow \on{im}(F)\cap\partial X'\quad (t,z)\mapsto F(G(K_t(z),t,z)).
\end{equation*}
A direct computation shows that $F'^*\omega_0 = \omega + dK_t\wedge dt$. Moreover, we have $F'^*\lambda_0 = dt+\lambda$ near the boundary of $[0,\epsilon]\times D$. Thus we may extend $\beta$ to a smooth $1$-form on all of $\partial X'$ be requiring that $F'^*\beta = dt+\lambda + K_t dt$. The resulting $1$-form clearly satisfies $d\beta = \omega_0|_{\partial X'}$. It remains to check $\beta$ is indeed a contact form on $\on{im}(F)\cap\partial X'$ for appropriate choice of $K$. This amounts to showing that $F'^* (\beta\wedge d\beta)$ is a positive volume form on $[0,\epsilon]\times D$. We compute
\begin{equation*}
F'^* (\beta\wedge d\beta) = dt\wedge ((1+K_t)\omega + \lambda\wedge dK_t).
\end{equation*}
Thus it suffices to construct $K$ such that $(1+K_t)\omega + \lambda\wedge dK_t$ is a positive area form on $D$. The first term clearly is a positive area form since $K_t\geq 0$. Let $W$ denote the Liouville vector field on $D$ induced by $\lambda$. We can guarantee that the second term is non-negative by choosing $K_t$ to be constant outside a small neighbourhood of $\partial D$ and requiring that $dK_t(W)\leq 0$ inside that small neighbourhood. Clearly we still have the freedom to make \eqref{eq:composite_hamiltonian_in_prop_modify_hypersurface_such_that_return_map_is_generated_by_positive_hamiltonian} positive.
\end{proof}

\section{A positivity criterion for Hamiltonian diffeomorphisms}
\label{section:a_positivity_criterion_for_hamiltonian_diffeomorphisms}

The results of this section are inspired by the fixed point theorem stated in \cite[Theorem 3]{ABHS18}. In fact, our proofs rely on the generalized generating functions introduced in \cite[sections 2.3 to 2.6]{ABHS18}.

\subsection{Statement of the positivity criterion}
\label{subsection:statement_of_the_positivity_criterion}

The generalized generating function framework we use applies to area-preserving diffeomorphisms of the disk $\BD$ which are {\it radially monotone} in the sense of the following definition.

\begin{definition}
A diffeomorphism $\phi\in\Diff^+(\BD)$ is called {\it radially monotone} if it fixes the center $0$ and if the image of the radial foliation of $\BD$ under $\phi$ is transverse to the foliation of $\BD$ by circles centred at $0$.
\end{definition}

We state the main result of this section.

\begin{theorem}
\label{theorem:positivity_criterion_for_radially_monotone_diffeomorphisms}
Let $\omega$ be a smooth $2$-form on $\BD$ which is positive in the interior $\operatorname{int}(\BD)$. Moreover, assume that $\omega$ is rotation invariant near the origin and the boundary $\partial\BD$. Let $\widetilde{\phi}\in\widetilde{\Diff}(\BD,\omega)$ and set $\phi\coloneqq \pi(\widetilde{\phi})$. Assume that:
\begin{enumerate}
\item $\phi$ fixes the origin and is radially monotone.
\item The restriction of $\phi$ to a small disk centred at the origin is a rotation.
\item The restriction of $\phi$ to a small annular neighbourhood of $\partial\BD$ is a rotation.
\item The action $\sigma_{\widetilde{\phi}}(p)$ is positive for all fixed points $p$ of $\phi$.
\end{enumerate}
Then there exists a Hamiltonian $H:\BR/\BZ \times \BD \rightarrow \BR$ with the following properties:
\begin{enumerate}
\item $H$ is strictly positive in the interior $\operatorname{int}(\BD)$ and vanishes on the boundary $\partial\BD$.
\item $H$ is autonomous and rotation invariant in some neighbourhood of the origin and $\partial\BD$.
\item The Hamiltonian vector field $X_{H_t}$ defined by $\iota_{X_{H_t}}\omega = dH_t$ in the interior $\operatorname{int}(\BD)$ smoothly extends to the closed disk $\BD$ and is tangent to $\partial \BD$.
\item The arc $(\phi_H^t)_{t\in [0,1]}$ represents $\widetilde{\phi}$.
\end{enumerate}
\end{theorem}

In order to prove Theorem \ref{theorem:strong_viterbo_near_round_ball}, we need to apply Theorem \ref{theorem:positivity_criterion_for_radially_monotone_diffeomorphisms} to area-preserving diffeomorphisms of the disk which are $C^1$-close to the identity. Such diffeomorphisms need not be radially monotone. However, they are smoothly conjugated to radially monotone diffeomorphisms (see \cite[Proposition 2.24]{ABHS18}). We use this observation to deduce the following corollary of Theorem \ref{theorem:positivity_criterion_for_radially_monotone_diffeomorphisms}.

\begin{cor}
\label{cor:positivity_criterion_for_diffeomorphisms_close_to_the_identity}
Let $\omega$ be a smooth $2$-form on $\BD$ which is positive in the interior $\operatorname{int}(\BD)$. Let $\widetilde{\phi}\in\widetilde{\Diff}(\BD,\omega)$ and set $\phi\coloneqq \pi(\widetilde{\phi})$. Assume that:
\begin{enumerate}
\item $\widetilde{\phi}$ is $C^1$-close to the identity $\identity_\BD$.
\item $\phi$ is smoothly conjugated to a rotation in some neighbourhood of the boundary $\partial\BD$.
\item The action $\sigma_{\widetilde{\phi}}(p)$ is positive for all fixed points $p$ of $\phi$.
\end{enumerate}
Then there exists a Hamiltonian $H:\BR/\BZ \times \BD\rightarrow\BR$ with the following properties:
\begin{enumerate}
\item $H$ is strictly positive in the interior $\operatorname{int}(\BD)$ and vanishes on the boundary $\partial\BD$.
\item $H$ is autonomous in some neighbourhood of $\partial\BD$.
\item The Hamiltonian vector field $X_{H_t}$ defined by $\iota_{X_{H_t}}\omega = dH_t$ in the interior $\operatorname{int}(\BD)$ smoothly extends to the closed disk $\BD$ and is tangent to $\partial \BD$.
\item The arc $(\phi_H^t)_{t\in [0,1]}$ represents $\widetilde{\phi}$.
\end{enumerate}
\end{cor}

\subsection{Lifts to the strip}
\label{subsection:lifts_to_the_strip}

In order to represent area preserving diffeomorphisms of the disk $\BD$ by generalized generating functions, it will be convenient to lift them to the strip $S\coloneqq [0,1]\times\BR$. This is carefully explained in \cite[section 2.3]{ABHS18}. Here we summarize the relevant material. Consider the map
\begin{equation*}
p:S\rightarrow\BD\quad (r,\theta)\mapsto r\cdot e^{i\theta}.
\end{equation*}
The restriction of $p$ to $(0,1]\times\BR$ is a covering map to $\BD\setminus\{0\}$. The translation
\begin{equation*}
T:S\rightarrow S\quad (r,\theta)\mapsto (r,\theta+2\pi)
\end{equation*}
generates the group of deck transformations. Any orientation preserving diffeomorphism $\phi\in\Diff^+(\BD)$ fixing the origin lifts to a diffeomorphism $\Phi\in\Diff^+(S)$ (see \cite[Lemma 2.10]{ABHS18}). The lift $\Phi$ commutes with $T$, i.e.
\begin{equation*}
T\circ\Phi=\Phi\circ T.
\end{equation*}
Any two lifts are related by composition with a deck transformation. A lift $\widetilde{\phi}\in\widetilde{\Diff}(\BD)$ of $\phi$ to the universal cover uniquely specifies a lift $\Phi\in\Diff^+(S)$ as follows: Represent $\widetilde{\phi}$ by a smooth arc $(\phi_t)_{t\in [0,1]}$ in $\Diff^+(\BD)$ starting at the identity. This arc uniquely lifts to an arc $(\Phi_t)_{t\in [0,1]}$ in $\Diff^+(S)$ starting at the identity. Now simply set $\Phi\coloneqq \Phi_1$.\\

If the diffeomorphism $\phi\in\Diff^+(\BD)$ is radially monotone, then any lift $\Phi$ is monotone in the sense of the following definition.

\begin{definition}
\label{definition:monotone_diffeomorphism_of_strip}
Let $\Phi\in\Diff^+(S)$ be a diffeomorphism and denote the components of $\Phi$ by $(R,\Theta)$. We call $\Phi$ {\it monotone} if $\partial_1R(r,\theta)>0$ for all $(r,\theta)\in S$.
\end{definition}

The following characterization of monotonicity will be useful in later sections.

\begin{lem}
\label{lem:monotonicity_equivalent_to_diagonal_projection_being_diffeomorphism}
Let $\Phi\in\Diff^+(S)$ be a diffeomorphism preserving the boundary components of $S$. Let $\Gamma(\Phi)\subset S\times S$ denote the graph of $\Phi$. Then $\Phi$ is monotone if and only if
\begin{equation*}
\pi: \Gamma(\Phi)\subset S\times S \rightarrow S\quad (r,\theta,R,\Theta)\mapsto (R,\theta)
\end{equation*}
is a diffeomorphism.
\end{lem}

\begin{proof}
Let $\Phi(r,\theta) = (R(r,\theta),\Theta(r,\theta))$ denote the components of $\Phi$. The map
\begin{equation*}
(\identity_S,\Phi):S\rightarrow \Gamma(\Phi)\quad (r,\theta)\mapsto (r,\theta,R(r,\theta),\Theta(r,\theta))
\end{equation*}
is a parametrization of $\Gamma(\Phi)$. Clearly, $\pi$ is a diffeomorphism if and only if the composition
\begin{equation*}
\pi\circ (\identity_S,\Phi):S\rightarrow S\quad (r,\theta)\mapsto (R(r,\theta),\theta)
\end{equation*}
is a diffeomorphism. This is the case if and only if
\begin{equation}
\label{eq:R_in_proof_of_monotonicity_criterion}
R(\cdot,\theta):[0,1]\rightarrow [0,1]
\end{equation}
is a diffeomorphism for every fixed $\theta\in \BR$. By assumption $\Phi$ preserves the boundary components of $S$. Thus $R(0,\theta) = 0$ and $R(1,\theta)=1$ for every $\theta$. Hence \eqref{eq:R_in_proof_of_monotonicity_criterion} is a diffeomorphism if and only if $\partial_1R>0$, i.e. $\Phi$ is monotone.
\end{proof}

Now suppose that $\BD$ is equipped with a $2$-form $\omega$ which is positive in the interior $\operatorname{int}(\BD)$. Then
\begin{equation*}
\Omega\coloneqq p^*\omega
\end{equation*}
is a $2$-form on $S$ which is positive in the interior and invariant under the translation $T$, i.e.
\begin{equation*}
T^*\Omega = \Omega.
\end{equation*}
If $\phi\in\Diff(\BD,\omega)$ fixes the origin and preserves $\omega$, then any lift $\Phi$ preserves $\Omega$, i.e. $\Phi\in\Diff(S,\Omega)$.

\subsection{Generalized generating functions on the strip}
\label{subsection:generalized_generating_functions}

Throughout this section, let $\Omega$ be a $2$-form on the strip $S$ which is positive in the interior $\interior(S)$ and preserved by $T$. Moreover, let $\Phi\in\Diff(S,\Omega)$ be a diffeomorphism which preserves $\Omega$ and the boundary components of $S$ and which commutes with $T$. We equip the product $S\times S$ with the closed $2$-form $(-\Omega)\oplus\Omega$. The restriction of this $2$-form to the interior of $S\times S$ is a symplectic form. Consider a primitive $1$-form $\alpha$ of $(-\Omega)\oplus\Omega$ whose restriction to the diagonal $\Delta\subset S\times S$ vanishes. The graph $\Gamma(\Phi)$ is a Lagrangian submanifold of $(S\times S,(-\Omega)\oplus\Omega)$, i.e. the restriction of $(-\Omega)\oplus\Omega$ to $\Gamma(\Phi)$ vanishes. Therefore, the restriction of the primitive $\alpha$ to $\Gamma(\Phi)$ is closed. Since $S$ is simply connected, it is also exact, i.e. it can be written as
\begin{equation}
\label{eq:defining_equation_generalized_generating_function}
\alpha|_{\Gamma(\Phi)} = d W_{\Phi,\alpha}
\end{equation}
for a function $W_{\Phi,\alpha}:\Gamma(\Phi)\rightarrow\BR$, which is unique up to addition of a constant. We call $W_{\Phi,\alpha}$ a {\it generalized generating function} for $\Phi$ with respect to $\alpha$ (c.f. \cite[Definition 9.3.8]{MS17}). In Lemma \ref{lem:normalization_generalized_generating_function} below, we specify a preferred normalization of $W_{\Phi,\alpha}$ which we will use throughout this paper. For $z\in \{1\}\times \BR\subset \partial S$, let $\gamma_z$ denote the path
\begin{equation*}
\gamma_z : [0,1]\rightarrow \partial S\times \partial S\qquad \gamma_z(t)\coloneqq (z,(1-t)z+t\Phi(z)).
\end{equation*}

\begin{lem}
\label{lem:normalization_generalized_generating_function}
There exists a unique smooth function $W_{\Phi,\alpha}:\Gamma(\Phi)\rightarrow\BR$ satisfying
\begin{equation}
\label{eq:characterizing_equation_generalized_generating_function_general_primitive}
d W_{\Phi,\alpha} = \alpha|_{\Gamma(\Phi)}
\end{equation}
and
\begin{equation}
\label{eq:normalization_equation_generalized_generating_function_general_primitive}
W_{\Phi,\alpha}(z,\Phi(z)) = \int_{\gamma_z}\alpha\qquad\text{for all}\enspace z\in\{1\}\times \BR.
\end{equation}
\end{lem}

\begin{proof}
Set $z_0\coloneqq (1,0)$. Clearly, there exists a unique function $W_{\Phi,\alpha}$ satisfying \eqref{eq:characterizing_equation_generalized_generating_function_general_primitive} such that \eqref{eq:normalization_equation_generalized_generating_function_general_primitive} holds for $z_0$. We need to check that this function $W_{\Phi,\alpha}$ satisfies \eqref{eq:normalization_equation_generalized_generating_function_general_primitive} for all $z\in \{1\}\times \BR$. Fix $z\in\{1\}\times\BR$. We define two paths $\delta$ and $\epsilon$ by
\begin{equation*}
\delta : [0,1] \rightarrow \partial S \times \partial S \qquad \delta(t) = ((1-t)z_0+tz, \Phi((1-t)z_0+tz))
\end{equation*}
and
\begin{equation*}
\epsilon : [0,1] \rightarrow \partial S \times \partial S \qquad \epsilon(t) = ((1-t)z_0+tz,(1-t)z_0+tz).
\end{equation*}
The path $\delta$ is contained in $\Gamma(\Phi)$ and goes from $(z_0,\Phi(z_0))$ to $(z,\Phi(z))$. The path $\epsilon$ connects $(z_0,z_0)$ to $(z,z)$ and is contained in $\Delta$. The concatenations $\gamma_{z_0} \# \delta$ and $\epsilon \# \gamma_z$ are homotopic inside $\partial S\times\partial S$ with fixed end points. Since the restriction of $(-\Omega)\oplus\Omega$ to $\partial S\times \partial S$ vanishes, the restriction of $\alpha$ to this subspace is closed. Thus
\begin{equation*}
\int_{\gamma_{z_0} \# \delta} \alpha = \int_{\epsilon \# \gamma_z} \alpha.
\end{equation*}
Using \eqref{eq:characterizing_equation_generalized_generating_function_general_primitive} and the fact that $W_{\Phi,\alpha}$ satisfies \eqref{eq:normalization_equation_generalized_generating_function_general_primitive} for $z_0$, the left hand side evaluates to
\begin{equation*}
\int_{\gamma_{z_0} \# \delta} \alpha = \int_{\gamma_{z_0}}\alpha + \int_{\delta}\alpha = W_{\Phi,\alpha}(z_0,\Phi(z_0)) + W_{\Phi,\alpha}(z,\Phi(z)) - W_{\Phi,\alpha}(z_0,\Phi(z_0)) = W_{\Phi,\alpha}(z,\Phi(z)).
\end{equation*}
Since the restriction of $\alpha$ to the diagonal $\Delta$ vanishes, the right hand side is given by
\begin{equation*}
\int_{\epsilon \# \gamma_z} \alpha = \int_{\epsilon}\alpha + \int_{\gamma_z}\alpha =\int_{\gamma_z}\alpha.
\end{equation*}
This concludes our proof that \eqref{eq:normalization_equation_generalized_generating_function_general_primitive} holds for all $z\in\{1\}\times\BR$.
\end{proof}

Let $\beta$ be a second primitive $1$-form of $(-\Omega)\oplus\Omega$ whose restriction to $\Delta$ vanishes. Then the difference $\alpha-\beta$ is exact, i.e. there exists a smooth function $u$ on $S\times S$ such that $\alpha-\beta = du$. Since the restrictions of $\alpha$ and $\beta$ to the diagonal $\Delta$ vanish, the function $u$ must be constant on $\Delta$. Let us normalize $u$ such that $u|_\Delta=0$. The following lemma relates the generalized generating functions of $\Phi$ with respect to $\alpha$ and $\beta$.

\begin{lem}
\label{lem:generalized_generating_functions_change_of_primitive}
Let $W_{\Phi,\alpha}$ and $W_{\Phi,\beta}$ be the generalized generating functions of $\Phi$ with respect to $\alpha$ and $\beta$. Then
\begin{equation*}
W_{\Phi,\alpha} = W_{\Phi,\beta} + u|_{\Gamma(\Phi)}.
\end{equation*}
In particular, since $u$ vanishes on the diagonal $\Delta$, the value of a generalized generating function at a fixed point of $\Phi$ is independent of the choice of primitive $1$-form.
\end{lem}

\begin{proof}
We set $W\coloneqq W_{\Phi,\beta} + u|_{\Gamma(\Phi)}$. We need to check that this function satisfies \eqref{eq:characterizing_equation_generalized_generating_function_general_primitive} and \eqref{eq:normalization_equation_generalized_generating_function_general_primitive}. In order to show \eqref{eq:characterizing_equation_generalized_generating_function_general_primitive}, we compute
\begin{equation*}
dW = dW_{\Phi,\beta} + d u|_{\Gamma(\Phi)} = \beta|_{\Gamma(\Phi)} + (\alpha-\beta)|_{\Gamma(\Phi)} = \alpha|_{\Gamma(\Phi)}.
\end{equation*}
Let $z\in\{1\}\times\BR$. We have
\begin{equation*}
W(z,\Phi(z)) = W_{\Phi,\beta}(z,\Phi(z)) + u(z,\Phi(z)) = \int_{\gamma_z}\beta + \int_{\gamma_z} du = \int_{\gamma_z}\alpha.
\end{equation*}
Here the second equality uses that $u$ vanishes on the diagonal $\Delta$. This shows \eqref{eq:normalization_equation_generalized_generating_function_general_primitive}.
\end{proof}

There are two primitives of $(-\Omega)\oplus\Omega$ whose associated generalized generating functions are of particular importance to our discussion. The first such primitive is given by $(-\Lambda)\oplus\Lambda$ where $\Lambda$ is a primitive of the area form $\Omega$ on $S$. It will be useful to regard the associated generalized generating function $W_{\Phi,(-\Lambda)\oplus\Lambda}$ as a function on $S$ via the parametrization $(\identity_S,\Phi):S\rightarrow\Gamma(\Phi)$ of the graph $\Gamma(\Phi)$. We define
\begin{equation*}
\Sigma_{\Phi,\Lambda} \coloneqq W_{\Phi,(-\Lambda)\oplus\Lambda} \circ (\identity_S,\Phi)
\end{equation*}
and call it the {\it action} of $\Phi$ with respect to $\Lambda$. The characterizing equations \eqref{eq:characterizing_equation_generalized_generating_function_general_primitive} and \eqref{eq:normalization_equation_generalized_generating_function_general_primitive} for the generalized generating function $W_{\Phi,(-\Lambda)\oplus\Lambda}$ can be expressed in terms of the action $\Sigma_{\Phi,\Lambda}$ as
\begin{equation}
\label{eq:characterizing_equation_action_on_strip}
\Phi^*\Lambda - \Lambda = d\Sigma_{\Phi,\Lambda}
\end{equation}
and
\begin{equation}
\label{eq:normalization_equation_action_on_strip}
\Sigma_{\Phi,\Lambda}(1,\theta) = \int_{\delta_\theta}\Lambda
\qquad \text{for all}\enspace \theta\in\BR
\end{equation}
where $\delta_\theta$ denotes the path
\begin{equation*}
\delta_\theta : [0,1]\rightarrow \partial S\quad \delta_\theta(t)\coloneqq (1-t)\cdot(1,\theta)+t\cdot \Phi(1,\theta).
\end{equation*}

The following basic properties of the action $\Sigma_{\Phi,\Lambda}$ will be useful later on.

\begin{lem}
\label{lem:action_on_strip_basic_properties}
\begin{enumerate}
\item Let $(\Phi_t)_{t\in [0,1]}$ be an arc in $\Diff(S,\Omega)$ starting at the identity. Let
\begin{equation*}
H:[0,1] \times S \rightarrow \BR
\end{equation*}
be a Hamiltonian generating this arc. If we normalize $H$ by $H_t(1,\theta)=0$, then the action $\Sigma_{\Phi,\Lambda}$ may be computed via
\begin{equation*}
\Sigma_{\Phi,\Lambda}(z) = \int_{\{t\mapsto\Phi_t(z)\}}\Lambda + \int_0^1 H_t(\Phi_t(z))dt.
\end{equation*}
\item Suppose that $\Lambda=p^*\lambda$ is the pull-back of a primitive $\lambda$ of $\omega$ and that $\Phi$ is the lift of a diffeomorphism $\widetilde{\phi}\in\widetilde{\Diff}(\BD,\omega)$ fixing the origin. Then $\Sigma_{\Phi,\Lambda} = \sigma_{\widetilde{\phi},\lambda}\circ p$.
\end{enumerate}
\end{lem}

\begin{proof}
Statement (1) is the analog of \cite[Proposition 2.6]{ABHS18}, which deals with area-preserving diffeomorphisms of the disk $\BD$. The proof given in \cite{ABHS18} carries over to the case of the strip almost verbatim and we will not repeat it here.\\

We prove (2). We compute
\begin{equation*}
d(\sigma_{\widetilde{\phi},\lambda} \circ p) = p^* d\sigma_{\widetilde{\phi},\lambda} = p^* (\phi^*\lambda - \lambda) = \Phi^*p^*\lambda - p^*\lambda = \Phi^*\Lambda - \Lambda.
\end{equation*}
Here the third equality uses the identity $p\circ\Phi = \phi\circ p$. Let $(\phi_t)_{t\in [0,1]}$ be an arc in $\Diff^+(\BD)$ representing $\widetilde{\phi}$. Then
\begin{equation*}
\sigma_{\widetilde{\phi},\lambda}\circ p(1,\theta) = \int_{t\mapsto\phi_t(p(1,\theta))} \lambda = \int_{p\circ\delta_\theta} \lambda = \int_{\delta_\theta} p^*\lambda = \int_{\delta_\theta} \Lambda = \Sigma_{\Phi,\Lambda}(1,\theta).
\end{equation*}
Here the second equality uses that the restriction of $\lambda$ to $\partial\BD$ is closed and that $t\mapsto\phi_t(p(1,\theta))$ and $p\circ\delta_\theta$ are homotopic in $\partial\BD$ with fixed end points. This shows that $\sigma_{\widetilde{\phi},\lambda} \circ p$ satisfies \eqref{eq:characterizing_equation_action_on_strip} and \eqref{eq:normalization_equation_action_on_strip}. Thus $\sigma_{\widetilde{\phi},\lambda} \circ p = \Sigma_{\Phi,\Lambda}$.
\end{proof}

Let us define a second special primitive of $(-\Omega)\oplus\Omega$. We write
\begin{equation*}
\Omega = F(r,\theta)\cdot dr\wedge d\theta
\end{equation*}
where $F$ is a smooth function on $S$ which is positive in the interior and invariant under $T$. Next, we define functions $A$ and $B$ on $S$ by
\begin{equation*}
A(r,\theta)\coloneqq \int_0^rF(s,\theta)ds \quad\quad \text{and}\quad\quad B(r,\theta) \coloneqq \int_0^\theta F(r,\vartheta) d\vartheta.
\end{equation*}
We let $(r,\theta,R,\Theta)$ denote coordinates on $S\times S$ and define
\begin{equation*}
\Xi \coloneqq (A(R,\theta)-A(r,\theta))\cdot d\theta + (B(R,\theta)-B(R,\Theta))\cdot dR.
\end{equation*}
A direct computation shows that $d\Xi=(-\Omega)\oplus\Omega$ and that the restriction of $\Xi$ to the diagonal $\Delta\subset S\times S$ vanishes. The resulting generalized generating function $W_{\Phi,\Xi}$ is particularly useful if the diffeomorphism $\Phi\in\Diff(S,\Omega)$ is monotone in the sense of Definition \ref{definition:monotone_diffeomorphism_of_strip}. From now on, let us assume that this is the case. Consider the projection
\begin{equation}
\label{eq:diagonal_projection_for_generating_function}
\pi_\Delta: \Gamma(\Phi)\rightarrow S\quad (r,\theta,R,\Theta)\mapsto (R,\theta).
\end{equation}
By Lemma \ref{lem:monotonicity_equivalent_to_diagonal_projection_being_diffeomorphism}, $\pi_\Delta$ is a diffeomorphism. It will be convenient to view $W_{\Phi,\Xi}$ as a function on $S$ via the diffeomorphism $\pi_\Delta$. We abbreviate
\begin{equation*}
W\coloneqq W_{\Phi,\Xi}\circ\pi_\Delta^{-1}.
\end{equation*}
Equation \eqref{eq:characterizing_equation_generalized_generating_function_general_primitive} can be rewritten in terms of $W$ as
\begin{equation}
\label{eq:defining_equation_generating_function}
\begin{cases}
\partial_1W(R,\theta) = B(R,\theta)-B(R,\Theta)\\
\partial_2W(R,\theta) = A(R,\theta) - A(r,\theta)
\end{cases}
\qquad \text{for all}\enspace (r,\theta,R,\Theta)\in\Gamma(\Phi).
\end{equation}
The normalization \eqref{eq:normalization_equation_generalized_generating_function_general_primitive} simply becomes
\begin{equation}
\label{eq:normalization_generating_function}
W|_{\{1\}\times \BR} = 0
\end{equation}
because the restriction of the primitive $\Xi$ to $(\{1\}\times\BR)\times (\{1\}\times\BR)$ vanishes. We summarize the relevant properties of $W$ in the following Proposition (see Proposition 2.15, Lemma 2.16 and Proposition 2.17 in \cite{ABHS18}).

\begin{prop}
\label{prop:generalized_generating_functions}
Suppose that $\Phi\in\Diff(S,\Omega)$ is monotone and commutes with $T$. Then there exists a unique generating function $W:S\rightarrow\BR$ satisfying equations \eqref{eq:defining_equation_generating_function} and the normalization \eqref{eq:normalization_generating_function}. The function $W$ is invariant under $T$ and is constant on the boundary components of $S$. The interior critical points of $W$ are precisely the interior fixed points of $\Phi$. We have $W(p)=\Sigma_{\Phi,\Lambda}(p)$ for all fixed points $p$ of $\Phi$ and any primitive $\Lambda$ of $\Omega$. If the restriction of $\Lambda$ to $\{0\}\times\BR$ vanishes, then $W$ agrees with $\Sigma_{\Phi,\Lambda}$ on $\{0\}\times\BR$.
\end{prop}

\begin{proof}
Existence and uniqueness of $W$ follow from Proposition 2.15 in \cite{ABHS18}. Moreover, this proposition asserts that $W$ is invariant under $T$ and constant on the boundary components of $S$ and that the interior critical points of $W$ are precisely the interior fixed points of $W$. The remaining assertions are proved in \cite[Lemma 2.16 and Proposition 2.17]{ABHS18} in the special case that $\Phi$ is the lift of a diffeomorphism $\widetilde{\phi}\in\widetilde{\Diff}(\BD,\omega)$ and $\Lambda=p^*\lambda$ for a primitive $\lambda$ of $\omega$. Since the statements in Proposition \ref{prop:generalized_generating_functions} are slightly more general, we provide independent proofs. It is a direct consequence of Lemma \ref{lem:generalized_generating_functions_change_of_primitive} that $W(p) = \Sigma_{\Phi,\Lambda}(p)$ for all fixed points $p$ of $\Phi$. Suppose that the restriction of $\Lambda$ to $\{0\}\times\BR$ vanishes. This implies that the restriction of $(-\Lambda)\oplus\Lambda$ to $(\{0\}\times\BR)^2$ vanishes. Similarly, the restriction of the primitive $\Xi$ to this subspace vanishes. It is a direct consequence of \eqref{eq:characterizing_equation_generalized_generating_function_general_primitive} that the generating functions $W_{\Phi,\Xi}$ and $W_{\Phi,(-\Lambda)\oplus\Lambda}$ are both constant on $\Gamma(\Phi)\cap (\{0\}\times\BR)^2$. Let $u$ be the unique smooth function on $S\times S$ whose restriction to the diagonal $\Delta$ vanishes and which satisfies $\Xi = (-\Lambda)\oplus\Lambda + du$. Since both $\Xi$ and $(-\Lambda)\oplus\Lambda$ restrict to zero on $(\{0\}\times\BR)^2$, the function $u$ vanishes on this set. By Lemma \ref{lem:generalized_generating_functions_change_of_primitive}, this implies that $W_{\Phi,\Xi}$ and $W_{\Phi,(-\Lambda)\oplus\Lambda}$ agree on $\Gamma(\Phi)\cap (\{0\}\times\BR)^2$. We conclude that $W$ and $\Sigma_{\Phi,\Lambda}$ agree and are constant on $\{0\}\times\BR$.
\end{proof}

In order to avoid technicalities involving the behaviour of $\Phi$ and $W$ near $\partial S$, let us now assume in addition that $\Omega$ is translation invariant in some small neighbourhood of $\partial S$. In other words, $F(r,\theta)$ does not depend on $\theta$ for $r$ sufficiently close to $0$ or $1$. Moreover, we will restrict our attention to diffeomorphisms $\Phi$ whose restrictions to neighbourhoods of the two boundary components of $S$ are translations. More precisely, we will assume that there exist constants $\theta_0$ and $\theta_1$ such that for $j\in\{0,1\}$
\begin{equation}
\label{eq:translation_condition_on_diffeomorphisms}
\Phi(r,\theta)=(r,\theta+\theta_j) \quad\quad \text{if }r\text{ is sufficiently close to }j.
\end{equation}
Since the following result is not explicitly stated in \cite{ABHS18}, we provide a proof.

\begin{prop}
\label{prop:correspondence_symplectomorphisms_generating_functions}
There exists a bijective correspondence between the set of all diffeomorphisms $\Phi\in\Diff(S,\Omega)$ which are monotone, commute with $T$ and satisfy \eqref{eq:translation_condition_on_diffeomorphisms} and the set of all smooth functions $W:S\rightarrow\BR$ satisfying
\begin{enumerate}
\item\label{item:G_solvable} $0<A(r,\theta)-\partial_2W(r,\theta)<A(1,\theta)$ for all $(r,\theta)\in\operatorname{int}(S)$
\item\label{item:G_derivative_invertible} $\partial_{12}W(r,\theta)<F(r,\theta)$ for all $(r,\theta)\in\operatorname{int}(S)$
\item\label{item:G_periodic} $W\circ T=W$
\item\label{item:G_boundary_behaviour} There exist constants $c_1$, $c_2$ and $c_3$ such that
\begin{equation*}
\begin{cases}
W(r,\theta) = c_1 + c_2\cdot \int_0^rF(s,\theta)ds & \quad\text{if } r \text{ is sufficiently close to 0}\\
W(r,\theta) = c_3\cdot \int_r^1F(s,\theta)ds & \quad\text{if } r \text{ is sufficiently close to 1}.
\end{cases}
\end{equation*}
\end{enumerate}
$\Phi$ and $W$ correspond to each other under this bijection if and only if equations \eqref{eq:defining_equation_generating_function} hold.
\end{prop}

\begin{proof}
Let $\Phi\in\Diff(S,\Omega)$ be a diffeomorphism which is monotone, commutes with $T$ and satisfies \eqref{eq:translation_condition_on_diffeomorphisms}. Let $W$ be the associated generating function satisfying \eqref{eq:defining_equation_generating_function} and the normalization \eqref{eq:normalization_generating_function}. We verify that $W$ satisfies properties (1)-(4). Property (3) actually is a consequence of Proposition \ref{prop:generalized_generating_functions}. It follows from \eqref{eq:translation_condition_on_diffeomorphisms} that any point $(r,\theta,R,\Theta)\in\Gamma(\Phi)$ sufficiently close to the boundary satisfies $r=R$. Thus equation \eqref{eq:defining_equation_generating_function} implies that
\begin{equation*}
\partial_2 W(R,\theta) = A(R,\theta)-A(r,\theta) = 0
\end{equation*}
near $\partial S$. Hence $W$ is independent of $\theta$ in a neighbourhood of the boundary. Near $\partial S$ we also have
\begin{equation*}
B(r,\theta) = \theta\cdot F(r).
\end{equation*}
Here we use that $F(r,\theta)=F(r)$ does not depend on $\theta$ near $\partial S$. Thus \eqref{eq:defining_equation_generating_function} yields
\begin{equation*}
\partial_1W(R,\theta) = B(R,\theta)-B(R,\Theta) = (\theta-\Theta)\cdot F(R).
\end{equation*}
Using \eqref{eq:translation_condition_on_diffeomorphisms} we obtain
\begin{equation*}
\partial_1W(R,\theta) = -\theta_j\cdot F(R)
\end{equation*}
for $R$ close to $j\in\{0,1\}$. Property (4) is an immediate consequence. Next we check property (1). Rearranging the second equation in \eqref{eq:defining_equation_generating_function} gives
\begin{equation*}
A(R,\theta)-\partial_2W(R,\theta) = A(r,\theta).
\end{equation*}
Now we simply use that $A(0,\theta)=0$ and that $A(r,\theta)$ is strictly monotonic in $r$ for fixed $\theta$. It remains to verify (2). By Lemma \ref{lem:monotonicity_equivalent_to_diagonal_projection_being_diffeomorphism}, monotonicity of $\Phi$ implies that the projection $\pi_\Delta$ defined in equation \eqref{eq:diagonal_projection_for_generating_function} is a diffeomorphism. Let us now parametrize $\Gamma(\Phi)$ by the inverse of this diffeomorphism
\begin{equation*}
\pi_\Delta^{-1}: S\rightarrow \Gamma(\Phi)\subset S\times S \quad (R,\theta)\mapsto
\left(\begin{matrix}
r(R,\theta)\\
\theta\\
R\\
\Theta(R,\theta)
\end{matrix}\right).
\end{equation*}
For $j\in\{1,2\}$ let $\pi_j:S\times S\rightarrow S$ be the projection onto the $j$-th factor. The restriction of $\pi_j$ to $\Gamma(\Phi)$ is a diffeomorphism onto $S$. This shows that
\begin{equation*}
(R,\theta)\mapsto (r(R,\theta),\theta)\quad\quad\text{and}\quad\quad (R,\theta)\mapsto (R,\Theta(R,\theta))
\end{equation*}
are both diffeomorphisms of $S$. In fact, these diffeomorphisms are orientation preserving. Thus the linearizations
\begin{equation*}
\left(\begin{matrix}
\partial_1r(R,\theta) & \partial_2r(R,\theta)\\
0 & 1
\end{matrix}\right)
\quad\quad\text{and}\quad\quad
\left(\begin{matrix}
1 & 0\\
\partial_1\Theta(R,\theta) & \partial_2\Theta(R,\theta)
\end{matrix}\right)
\end{equation*}
both have positive determinant. This implies $\partial_1r(R,\theta)>0$ and $\partial_2\Theta(R,\theta)>0$. Let us view both sides of the first equation in \eqref{eq:defining_equation_generating_function} as functions of $R$ and $\theta$. Then differentiating with respect to $\theta$ yields
\begin{equation*}
\partial_{12}W(R,\theta) = F(R,\theta) - F(R,\Theta(R,\theta))\cdot\partial_2\Theta(R,\theta).
\end{equation*}
In the interior of $S$, both $F(R,\Theta)$ and $\partial_2\Theta$ are strictly positive. Thus
\begin{equation*}
\partial_{12}W(R,\theta) < F(R,\theta)
\end{equation*}
proving (2).\\

Let us now prove the converse direction of Proposition \ref{prop:correspondence_symplectomorphisms_generating_functions}. We start with a generating function $W$ satisfying properties (1)-(4). Let us first show that we may solve equations \eqref{eq:defining_equation_generating_function} for $r$ and $\Theta$ and obtain a smooth map
\begin{equation*}
\operatorname{int}(S)\ni (R,\theta)\mapsto (r(R,\theta),\Theta(R,\theta)) \in \operatorname{int}(S).
\end{equation*}
The function $B(R,\cdot)$ is a diffeomorphism of $\BR$ for fixed $R\in (0,1)$. Thus we may apply the implicit function theorem and solve the first equation in \eqref{eq:defining_equation_generating_function} for $\Theta(R,\theta)$ in the interior $\operatorname{int}(S)$. For fixed $\theta$, the function $A(\cdot,\theta)$ is a diffeomorphism from $(0,1)$ onto $(0,A(1,\theta))$. By property (1), $\partial_2W(R,\theta)$ is contained in the image of $A(R,\theta)-A(\cdot,\theta)$. Again we invoke the implicit funciton theorem to solve the second equation in \eqref{eq:defining_equation_generating_function} for $r(R,\theta)$. The map
\begin{equation}
\label{eq:parametrization_iota_of_graph_of_phi}
\iota:\operatorname{int}(S)\rightarrow S\times S\quad (R,\theta)\mapsto \left(\begin{matrix}
r(R,\theta)\\
\theta\\
R\\
\Theta(R,\theta)
\end{matrix}\right)
\end{equation}
parametrizes a smooth submanifold of $S\times S$. We show that this submanifold is in fact the graph of a diffeomorphism of $\operatorname{int}(S)$. Differentiating the first equation in \eqref{eq:defining_equation_generating_function} with respect to $\theta$ and the second equation with respect to $R$ yields:
\begin{equation*}
\begin{cases}
\partial_{12}W(R,\theta) = F(R,\theta) - F(R,\Theta(R,\theta))\cdot \partial_2\Theta(R,\theta)\\
\partial_{12}W(R,\theta) = F(R,\theta) - F(r(R,\theta),\theta)\cdot \partial_1r(R,\theta)
\end{cases}
\end{equation*}
Using property (2), we conclude that $\partial_1r>0$ and $\partial_2\Theta>0$. Hence the composition $\pi_j\circ\iota$ is a diffeomorphism onto its image for $j\in\{1,2\}$. We show that this image actually is all of $\operatorname{int}(S)$. By property (4), we have $\partial_2W=0$ near $\partial S$. Thus the second equation in \eqref{eq:defining_equation_generating_function} implies that $r(R,\theta)=R$ for $R$ near $0$ or $1$. This shows that $r(\cdot,\theta)$ is a diffeomorphism of $(0,1)$ for fixed $\theta$. Therefore the image of $\pi_1\circ\iota$ is $\operatorname{int}(S)$. The function $\partial_1W$ is bounded. For fixed $R\in (0,1)$, the function $B(R,\cdot)$ is an orientation preserving diffeomorphism of $\BR$. Thus the first equation in \eqref{eq:defining_equation_generating_function} implies that $\lim_{\theta\rightarrow\pm\infty}\Theta(R,\theta)=\pm\infty$. Hence the image of $\pi_2\circ\iota$ is $\operatorname{int}(S)$. This shows that the image of $\iota$ is the graph of a diffeomorphism $\Phi\in\Diff(\operatorname{int}(S))$. Equations \eqref{eq:defining_equation_generating_function} imply that the pull-back of $\Xi$ via $\iota$ is closed. Therefore $\Phi$ preserves $\Omega$. Property (4) implies that in a neighbourhood of $\partial S$ equations \eqref{eq:defining_equation_generating_function} become:
\begin{equation*}
\begin{cases}
c\cdot F(R) = (\theta-\Theta)\cdot F(R)\\
0 = A(R,\theta)-A(r,\theta)
\end{cases}
\end{equation*}
We conclude that $\Phi$ satisfies \eqref{eq:translation_condition_on_diffeomorphisms} near $\partial S$. Therefore $\Phi$ smoothly extends to the closed strip and we have $\Phi\in\Diff(S,\Omega)$. In order to check that $\Phi$ commutes with $T$, first note that $A(R,\theta)$ is invariant under $T$ and that $B(R,\theta+2\pi)=B(R,\theta)+B(R,2\pi)$. In combination with invariance of $W$ under $T$, this implies that $r(R,\theta)$ is invariant under $T$ and that $\Theta(R,\theta+2\pi)=\Theta(R,\theta)+2\pi$. Hence the image of $\iota$ is invariant under the diagonal action of $T$ on $S\times S$, which implies that $\Phi$ commutes with $T$. The composition of $\pi_\Delta\circ\iota$ is equal to $\identity_S$. Thus $\pi_\Delta$ is a diffeomorphism, which is equivalent to monotonicity of $\Phi$ by Lemma \ref{lem:monotonicity_equivalent_to_diagonal_projection_being_diffeomorphism}.
\end{proof}

\subsection{Proof of the positivity criterion for radially monotone diffeomorphisms}
\label{subsection:proof_of_the_positivity_criterion_for_radially_monotone_diffeomorphisms}

\begin{proof}[Proof of Theorem \ref{theorem:positivity_criterion_for_radially_monotone_diffeomorphisms}]
Let $\Phi\in\Diff(S,\Omega)$ be the lift of $\widetilde{\phi}$ to the strip. $\Phi$ is monotone and commutes with $T$. Since $\phi$ is a rotation near the origin and the boundary, $\Phi$ is a translation near the two boundary components of $S$, i.e. $\Phi$ satisfies \eqref{eq:translation_condition_on_diffeomorphisms}. It follows from rotation invariance of $\omega$ near $0$ and $\partial\BD$ that $\Omega$ is translation invariant near $\partial S$. Let $W$ be the unique generating function satisfying properties (1)-(4) in Proposition \ref{prop:correspondence_symplectomorphisms_generating_functions}. For $t\in [0,1]$, we define $W_t\coloneqq t\cdot W$. The functions $W_t$ satisfy conditions (1)-(4) for all $t$. Indeed, the set of functions satisfying conditions (1)-(4) is convex and both $W$ and the zero function satisfy these conditions. Proposition \ref{prop:correspondence_symplectomorphisms_generating_functions} therefore yields an isotopy $\Phi_t\in\Diff(S,\Omega)$ starting at the identity and ending at $\Phi$. Let $H:[0,1]\times S\rightarrow\BR$ be the unique Hamiltonian generating $\Phi_t$ which is normalized by $H_t(1,\theta)=0$. Since $\Phi_t$ commutes with $T$ for all $t$, the Hamiltonian $H_t$ is invariant under $T$. Near the boundary components of $S$, the isotopy $\Phi_t$ is a translation at constant speed. Hence $H$ is autonomous and translation invariant near the boundary.\\

Our goal is to show that the restriction of $H$ to the complement of the boundary component $\{1\}\times\BR$ is strictly positive. Our strategy is the following: If we can show that $H_t$ is strictly positive on all its interior critical points and on the boundary component $\{0\}\times\BR$, then it follows that $H_t$ must be strictly positive on the complement of $\{1\}\times\BR$. We begin by showing that for every $t$, the set of interior critical points of $H_t$ is equal to the set of interior critical points of $W$. In other words, if the velocity $\partial_t\Phi_t(z)$ vanishes for some $t\in [0,1]$, then $z$ must be a critical point of $W$ and is fixed by the entire isotopy $\Phi_t$. Let $(R_t,\Theta_t)$ denote the components of $\Phi_t$. The defining equations for the generating function $W_t$ read:
\begin{equation*}
\begin{cases}
\partial_1W_t(R_t,\theta) = B(R_t,\theta)-B(R_t,\Theta_t)\\
\partial_2W_t(R_t,\theta) = A(R_t,\theta) - A(r,\theta)
\end{cases}
\end{equation*}
Differentiating with respect to $t$ yields:
\begin{equation*}
\begin{cases}
\partial_1 W(R_t,\theta) + \partial_{11}W_t(R_t,\theta)\cdot\partial_tR_t = \partial_1B(R_t,\theta)\cdot \partial_tR_t-\partial_1B(R_t,\Theta_t)\cdot \partial_tR_t-\partial_2B(R_t,\Theta_t)\cdot \partial_t\Theta_t \\
\partial_2W(R_t,\theta) + \partial_{12}W_t(R_t,\theta)\cdot\partial_tR_t = \partial_1A(R_t,\theta)\cdot\partial_tR_t
\end{cases}
\end{equation*}
If $z=(r,\theta)$ is a point satisfying $\partial_t\Phi_t(z)=0$ for some $t$, then these equations yield:
\begin{equation}
\label{eq:time_derivative_of_generating_equations_at_stationary_point}
\begin{cases}
\partial_1 W(R_t(r,\theta),\theta)=0\\
\partial_2 W(R_t(r,\theta),\theta)=0
\end{cases}
\end{equation}
This implies that $z$ is a critical point of $W$. Indeed, if $t=0$, then $R_t(r,\theta)=r$ and \eqref{eq:time_derivative_of_generating_equations_at_stationary_point} says that $(r,\theta)$ is a critical point of $W$. If $t>0$, then \eqref{eq:time_derivative_of_generating_equations_at_stationary_point} implies that $(R_t(r,\theta),\theta)$ is a critical point of $W_t$. Hence $z$ is fixed by $\Phi_t$ and therefore a critical point of $W_t$. Since $t>0$, this implies that $z$ is a critical point of $W$. Hence we have verified that every interior critical point of $H_t$ is a critical point of $W$. Conversely, suppose that $z$ is a critical points of $W$. Then $z$ is a critical point of $W_t$ for all $t$. Thus $z$ is fixed by the isotopy $\Phi_t$ and hence a critical point of $H_t$.\\

Let $z$ be an interior critical point of $H_t$. By the above discussion, $z$ is a fixed point of $\Phi$. We show that $H_t(z)=\Sigma_{\Phi}(z)$ for all $t$. This implies that $H_t(z)>0$. Indeed, all fixed points of $\phi$ are assumed to have strictly positive action and the same is true for $\Phi$ by item (3) in Lemma \ref{lem:action_on_strip_basic_properties}. Let $\lambda$ be a primitive of $\omega$ and let $\Lambda$ denote the pull-back to $S$. For every $\tau\in [0,1]$, we can compute the action $\Sigma_{\Phi_\tau}(z)$ via item (2) in Lemma \ref{lem:action_on_strip_basic_properties}
\begin{equation*}
\Sigma_{\Phi_\tau}(z) = \Sigma_{\Phi_\tau,\Lambda}(z) = \int_{\{[0,\tau]\ni t\mapsto\Phi_t(z)\}}\Lambda + \int_0^\tau H_t(\Phi_t(z))dt = \int_0^\tau H_t(z)dt.
\end{equation*}
Here the last equality uses that $z$ is fixed by the isotopy $\Phi_t$. By Proposition \ref{prop:generalized_generating_functions}, the action $\Sigma_{\Phi_\tau}(z)$ agrees with $W_\tau(z)$. We obtain
\begin{equation*}
\tau\cdot W(z) = \int_0^\tau H_t(z)dt.
\end{equation*}
Differentiating with respect to $\tau$ yields $H_\tau(z) = W(z) = \Sigma_{\Phi}(z)>0$.\\

Next we show that $H_t$ is positive on the boundary component $\{0\}\times\BR$. Since $\Lambda$ is given by the pull-back $p^*\lambda$ and $p$ maps the entire boundary component $\{0\}\times \BR$ to the origin $0$ of $\BD$, the restriction of $\Lambda$ to $\{0\}\times \BR$ vanishes. Thus item (2) in Lemma \ref{lem:action_on_strip_basic_properties} yields
\begin{equation*}
\Sigma_{\Phi_\tau,\Lambda}(z) = \int_0^\tau H_t(\Phi_t(z))dt = \int_0^\tau H_t(z) dt
\end{equation*}
for all $z\in \{0\}\times\BR$. Here the second equality uses that $H_t$ is translation invariant near the boundary and in particular constant on $\{0\}\times\BR$. By Proposition \ref{prop:generalized_generating_functions}, the action $\Sigma_{\Phi_\tau,\Lambda}$ agrees with $W_\tau$ on $\{0\}\times \BR$. Thus
\begin{equation*}
\tau\cdot W(z) = \int_0^\tau H_t(z)dt
\end{equation*}
for all $z\in\{0\}\times\BR$. Differentiating with respect to $\tau$ yields $H_\tau(z) = W(z) = \Sigma_{\Phi,\Lambda}(z)$. By item (3) in Lemma \ref{lem:action_on_strip_basic_properties}, the action $\Sigma_{\Phi,\Lambda}(z)$ is equal to the action $\sigma_{\widetilde{\phi}}(0)>0$. Hence $H_\tau$ is strictly positive on $\{0\}\times\BR$ for all $\tau$. This completes the proof that $H_t$ is strictly positive on the complement of $\{1\}\times\BR$.\\

The Hamiltonian $H$ is invariant under $T$ and constant on $\{0\}\times\BR$. Thus it descends to a continuous function $H:[0,1]\times \BD\rightarrow\BR$ which is smooth away from the origin. The Hamiltonian flow of $H$, which is defined on the complement of the origin, is a rotation in some small neighbourhood of the origin. Thus the flow extends to a smooth flow on the entire disk $\BD$. This implies that $H$ is actually smooth everywhere. Clearly $H$ satisfies properties (1)-(4) in Theorem \ref{theorem:positivity_criterion_for_radially_monotone_diffeomorphisms}.\\

There is one detail remaining: We actually want the Hamiltonian $H$ to be $1$-periodic in time. Here is how to fix this. Let $\eta:\BR\rightarrow [0,1]$ be a smooth cut-off function which vanishes in an open neighbourhood of $(-\infty,0]$ and is equal to $1$ in an open neighbourhood of $[1,\infty)$. For $\epsilon>0$ we define
\begin{equation*}
\eta^\epsilon(t)\coloneqq \eta\left(\frac{t-1+\epsilon}{\epsilon}\right).
\end{equation*}
The function $\eta^\epsilon$ vanishes in a neighbourhood of $(-\infty,1-\epsilon]$ and is equal to $1$ in a neighbourhood of $[1,\infty)$. Now define
\begin{equation*}
G^\epsilon:[0,1]\times \BD\rightarrow\BR\quad G^\epsilon(t,z) \coloneqq (1-\eta^\epsilon(t))\cdot H(t,z) + \eta^\epsilon(t)\cdot H(0,z).
\end{equation*}
This actually extends to a smooth $1$-periodic Hamiltonian $G^\epsilon:\BR/\BZ\times \BD\rightarrow\BR$. In an open neighbourhood of $0$ and $\partial\BD$, the Hamiltonian $G^\epsilon_t$ agrees with $H_t$ for all $t\in [0,1]$ because $H$ is autonomous in this region. If $t\in [0,1-\epsilon]$, then $G^\epsilon_t$ agrees with $H_t$ on the entire disk $\BD$. Moreover, $G^\epsilon$ is strictly positive in the interior of $\BD$. The time-$1$-map $\phi_{G^\epsilon}^1$ agrees with $\phi$ in a neighbourhood of $0$ and $\partial\BD$, but it need not agree with $\phi$ on the entire disk. As $\epsilon$ approaches $0$, the diffeomorphism $(\phi_{G^\epsilon}^1)^{-1}\circ\phi$, which is compactly supported in the complement of $0$ and $\partial\BD$, converges to the identity in the $C^1$-topology. Using Lemma \ref{lem:small_symplectomorphism_generated_by_small_hamiltonian}, we can therefore find a Hamiltonian $K^\epsilon:[0,1]\times \BD\rightarrow\BR$, compactly supported in the complement of $0$ and $\partial\BD$ and vanishing for $t$ close to $0$ or $1$, such that $\phi_{K^\epsilon}^1 = (\phi_{G^\epsilon}^1)^{-1}\circ\widetilde{\phi}$ and such that $\|X_{K^\epsilon}\|_{C^0}$ converges to $0$ as $\epsilon$ approaches $0$. Now define
\begin{equation*}
H^\epsilon_t\coloneqq (G^\epsilon \# K^\epsilon)_t = G^\epsilon_t + K^\epsilon_t\circ (\phi_{G^\epsilon}^t)^{-1}
\end{equation*}
for $t\in [0,1]$. This extends to a smooth $1$-periodic Hamiltonian. It agrees with $H$ in a neighbourhood of $0$ and $\partial\BD$ and its time-$1$-map is $\phi_{H^\epsilon}^1$ is equal to $\widetilde{\phi}$. For $\epsilon>0$ sufficiently small, $H^\epsilon$ is strictly positive in the interior $\operatorname{int}(\BD)$. Thus we have constructed a $1$-periodic Hamiltonian satisfying properties (1)-(4).
\end{proof}

\subsection{Proof of the positivity criterion for diffeomorphisms close to the identity}
\label{subsection:proof_of_the_positivity_criterion_for_diffeomorphisms_close_to_the_identity}

The goal of this section is to deduce Corollary \ref{cor:positivity_criterion_for_diffeomorphisms_close_to_the_identity} from Theorem \ref{theorem:positivity_criterion_for_radially_monotone_diffeomorphisms}. Key ingredient is the following lemma.

\begin{lem}
\label{lem:moving_fixed_point_to_origin}
Let $\phi:\BD\rightarrow\BD$ be a diffeomorphism. Assume that:
\begin{enumerate}
\item $\phi$ is sufficiently $C^1$-close to the identity $\operatorname{id}_\BD$.
\item $\phi$ is smoothly conjugated to a rotation near $\partial\BD$.
\item There exists a fixed point $p\in\operatorname{int}(\BD)$ such that $\phi$ is smoothly conjugated to a rotation in a neighbourhood of $p$.
\end{enumerate}
Then there exists a diffeomorphism $\psi:\BD\rightarrow\BD$ such that:
\begin{enumerate}
\item $\psi(0)=p$
\item $\psi^{-1}\circ\phi\circ\psi$ is radially monotone.
\item $\psi^{-1}\circ\phi\circ\psi$ is a rotation near $0$ and $\partial\BD$.
\end{enumerate}
\end{lem}

\begin{proof}
We proceed in four steps.\\
{\bf Step 1:} Let $\psi:\BD\rightarrow\BD$ be a M\"{o}bius transformation such that $\psi(0)=p$. By Proposition 2.24 in \cite{ABHS18}, the diffeomorphism $\psi^{-1}\circ\phi\circ \psi$ fixes the origin and is radially monotone. After replacing $\phi$ by $\psi^{-1}\circ\phi\circ \psi$, we can therefore assume that $\phi$ is radially monotone and smoothly conjugated to rotations near $0$ and $\partial\BD$. Note, however, that we can no longer guarantee that $\phi$ is $C^1$-close to the identity.\\
{\bf Step 2:} We show that we may further reduce to the case that $\phi$ agrees with the linearization $d\phi(0)$ in an entire open neighbourhood of $0$. We choose a $\phi$-invariant neighbourhood $U$ of $0$ and an orientation preserving diffeomorphism $f:(\BD,0)\rightarrow (U,0)$ such that $\rho\coloneqq f^{-1}\circ\phi\circ f$ is a rotation around the center of $\BD$. We may approximate $f$ with respect to the $C^1$-topology by a diffeomorphism $\widetilde{f}:(\BD,0)\rightarrow (U,0)$ which agrees with its linearization $d\widetilde{f}(0)$ near $0$ and with $f$ outside some arbitrarily small neighbourhood of $0$. We set $\psi\coloneqq f\circ\widetilde{f}^{-1}$. This defines a compactly supported diffeomorphism of $U$. We may smoothly extend it to a diffeomorphism of $\BD$ by setting $\psi$ to be equal to the identity outside $U$. The resulting diffeomorphism $\psi:(\BD,0)\rightarrow (\BD,0)$ is $C^1$-close to the identity and supported in a small neighbourhood of $0$. Near $0$ we have
\begin{equation*}
\psi^{-1}\circ\phi\circ\psi = (f\circ\widetilde{f}^{-1})^{-1}\circ\phi\circ (f\circ\widetilde{f}^{-1}) =
\widetilde{f}\circ (f^{-1}\circ\phi\circ f)\circ \widetilde{f}^{-1} =
 d\widetilde{f}(0)\circ \rho \circ d\widetilde{f}(0)^{-1}.
\end{equation*}
This shows that $\psi^{-1}\circ\phi\circ\psi$ is linear in some neighbourhood of $0$. Since $\psi$ is $C^1$-close to the identity and fixes the origin $0$, radial monotonicity is preserved by conjugation by $\psi$. We can therefore replace $\phi$ by $\psi^{-1}\circ\phi\circ\psi$ and assume in addition that $\phi$ agrees with $d\phi(0)$ near $0$.\\
{\bf Step 3:} Since $\phi$ agrees with its linearization near $0$ after performing Step 2, we may choose a linear orientation preserving diffeomorphism $g:(\BD,0)\rightarrow (U,0)$ onto a $\phi$-invariant neighbourhood of $0$ such that $\rho\coloneqq g^{-1}\circ \phi\circ g$ is a rotation of $\BD$. We construct a (non-linear) diffeomorphism $\widetilde{g}:(\BD,0)\rightarrow (U,0)$ satisfying properties (1)-(4) below. Let $(R,\Theta)$ and $(\widetilde{R},\widetilde{\Theta})$ denote the components of $g$ and $\widetilde{g}$ in polar coordinates, respectively.
\begin{enumerate}
\item $\widetilde{g}(r,\theta)$ agrees with $g(r,\theta)$ for $r> \frac{3}{4}$.
\item $\widetilde{\Theta}(r,\theta)$ agrees with $\Theta(r,\theta)$ for $r>\frac{1}{2}$.
\item There exists a constant $C>0$ such that $\widetilde{R}(r,\theta) = C\cdot r$ for $r<\frac{1}{2}$.
\item $\widetilde{\Theta}(r,\theta) = \theta$ for $r< \frac{1}{4}$.
\end{enumerate}
Let us first choose $C>0$ such that the ball $B_C(0)$ is contained in the image $g(B_{\frac{3}{4}}(0))$. We may choose a smooth function $\widetilde{R}(r,\theta)$ which agrees with $C\cdot r$ for $r<\frac{1}{2}$ and with $R(r,\theta)$ for $r>\frac{3}{4}$ and which satisfies $\partial_1\widetilde{R}(r,\theta)>0$. Since $g$ is linear, the function $\Theta(r,\theta)$ is actually independent of $r$ and we denote it by $\Theta(\theta)$. The function $\Theta(\theta)$ is an orientation preserving diffeomorphism of the circle $\BR/2\pi\BZ$. Hence there exists a smooth isotopy from $\Theta$ to the identity. Using such an isotopy, we may define a function $\widetilde{\Theta}(r,\theta)$ such that $\widetilde{\Theta}(r,\theta)=\Theta(\theta)$ for $r>\frac{1}{2}$ and $\widetilde{\Theta}(r,\theta) = \theta$ for $r<\frac{1}{4}$. It is immediate from the construction that $\widetilde{g} = (\widetilde{R},\widetilde{\Theta})$ is a diffeomorphism satisfying properties (1)-(4) above. As before, we define a diffeomorphism $\psi:(\BD,0)\rightarrow (\BD,0)$ which is compactly supported inside $U$ and agrees with $g\circ\widetilde{g}^{-1}$ inside $U$. We have
\begin{equation*}
\psi^{-1}\circ\phi\circ\psi = (g\circ\widetilde{g}^{-1})^{-1}\circ\phi\circ(g\circ\widetilde{g}^{-1}) = \widetilde{g}\circ (g^{-1}\circ\phi\circ g)\circ \widetilde{g}^{-1} = \widetilde{g}\circ \rho\circ \widetilde{g}^{-1}
\end{equation*}
inside $U$. Since $\widetilde{g}$ is simply multiplication by $C$ near $0$, this is an actual rotation near $0$. We need to check that $\psi^{-1}\circ\phi\circ\psi$ is radially monotone. Since this diffeomorphism agrees with $\phi$ outside $U$, we only need to check radial monotonicity of the restriction to $U$, which is given by $\widetilde{g}\circ \rho\circ \widetilde{g}^{-1}$ by the above computation. For $r>\frac{1}{2}$, the function $\widetilde{\Theta}(r,\theta)$ agrees with $\Theta(r,\theta)$ and is therefore independent of $r$. This implies that the restriction of $\widetilde{g}$ to $\BD\setminus B_{\frac{1}{2}}(0)$ preserves the foliations by radial rays. Since $\rho$ is a rotation, the same is true for the restriction of $\widetilde{g}\circ \rho\circ \widetilde{g}^{-1}$ to $\widetilde{g}(\BD\setminus B_{\frac{1}{2}}(0))$. This implies radial monotonicity of $\psi^{-1}\circ\phi\circ\psi$ on the set $\widetilde{g}(\BD\setminus B_{\frac{1}{2}}(0))$. For $r<\frac{1}{2}$, the function $\widetilde{R}$ has the special form $\widetilde{R}(r,\theta)=C\cdot r$. Thus the restriction of $\widetilde{g}$ to $B_{\frac{1}{2}}(0)$ preserves the foliation by circles centered at the origin. Since $\rho$ is a rotation, the same continues to hold for the restriction of $\widetilde{g}\circ \rho\circ \widetilde{g}^{-1}$ to $\widetilde{g}(B_{\frac{1}{2}}(0))$. Radial monotonicity of $\psi^{-1}\circ\phi\circ\psi$ on the set $\widetilde{g}(B_{\frac{1}{2}}(0))$ is a direct consequence. After replacing $\phi$ by $\psi^{-1}\circ\phi\circ\psi$, we can hence assume in addition that $\phi$ is an actual rotation near the origin.\\
{\bf Step 4:} It remains to perform a similar construction to turn $\phi$ into a rotation near $\partial\BD$ while preserving radial monotonicity. Let $\dot{\BD}\coloneqq \BD\setminus\{0\}$ denote the punctured disk. Let $V$ be a $\phi$-invariant neighbourhood of $\partial\BD$ and $f=(R,\Theta):\dot{\BD}\rightarrow V$ an orientation preserving diffeomorphism such that $\rho\coloneqq f^{-1}\circ\phi\circ f$ is a rotation of $\dot{\BD}$ around $0$. Our goal is to define a diffeomorphism $\widetilde{f}=(\widetilde{R},\widetilde{\Theta}):\dot{\BD}\rightarrow V$ agreeing with $f$ in a small neighbourhood of $0$ such that $\widetilde{f}\circ\rho\circ\widetilde{f}^{-1}$ is radially monotone on $V$ and an actual rotation near $\partial\BD$. Once we have such $\widetilde{f}$, we may set $\psi\coloneqq f\circ\widetilde{f}^{-1}$ and extend to a diffeomorphism of $\BD$ by setting $\psi$ to be equal to the identity outside $V$. We have
\begin{equation*}
\psi^{-1}\circ\phi\circ\psi = (f\circ\widetilde{f}^{-1})^{-1}\circ\phi\circ(f\circ\widetilde{f}^{-1}) = \widetilde{f}\circ (f^{-1}\circ\phi\circ f)\circ\widetilde{f}^{-1} = \widetilde{f}\circ \rho\circ\widetilde{f}^{-1}
\end{equation*}
which implies that $\psi^{-1}\circ\phi\circ\psi$ is radially monotone and a rotation near $\partial \BD$. Here is how we construct $\widetilde{f}$. After shrinking $V$ if necessary, we may assume that the image of $f(r,\cdot)$ is $C^1$-close to $\partial\BD$ for all $r$. In particular, $\Theta(r,\cdot)$ is a diffeomorphism of $\BR/2\pi\BZ$ for fixed $r$. Moreover, we can assume that $\partial_1 R(r,\theta)>0$. Let $\eta(r)$ be a smoothing of the function $\min(r,\frac{1}{5})$. Assume that both $\eta(r)$ and $\eta'(r)$ are monotonic and that $\eta(r)$ agrees with $\min(r,\frac{1}{5})$ outside $(\frac{1}{5}-\epsilon,\frac{1}{5}+\epsilon)$ for some small $\epsilon>0$. For $r<\frac{3}{5}$ we set $\widetilde{\Theta}(r,\theta)\coloneqq \Theta(\eta(r),\theta)$. For $r> \frac{4}{5}$ we set $\widetilde{\Theta}(r,\theta)\coloneqq \theta$. 
Choose a smooth isotopy from $\Theta(\frac{1}{5},\cdot)$ to $\identity_{\BR/2\pi\BZ}$ and use it to define $\widetilde{\Theta}$ in the interval $\frac{3}{5}<r<\frac{4}{5}$. We set $\widetilde{R}(r,\theta)$ to be equal to $R(r,\theta)$ for $r<\frac{2}{5}$. We extend $\widetilde{R}$ is such a way that $\partial_1\widetilde{R}(r,\theta)>0$. Moreover, we require that there exists $C>0$ such that $\widetilde{R}(r,\theta) = 1+C\cdot(r-1)$ for $r>\frac{3}{5}$. This finishes the construction of $\widetilde{f}$. Clearly, $\widetilde{f}\circ\rho\circ\widetilde{f}^{-1}$ is an actual rotation near $\partial\BD$. We need to check radial monotonicity. On the set $\widetilde{f}(\{r>\frac{3}{5}\})$, radial monotonicity follows from the special form $\widetilde{R}(r,\theta) = 1+C\cdot(r-1)$. Inside $\widetilde{f}(\{\frac{1}{5}+\epsilon<r<\frac{3}{5}\})$, radial monotonicity follows from the fact that $\widetilde{\Theta}(r,\theta)$ is independent of $r$. For $r<\frac{1}{5}-\epsilon$ the diffeomorphisms $\widetilde{f}$ and $f$ agree, which implies radial monotonicity in $\widetilde{f}(\{r<\frac{1}{5}-\epsilon\})$. It remains to verify radial monotonicity inside $\widetilde{f}(\{\frac{1}{5}-\epsilon<r<\frac{1}{5}+\epsilon\})$. A direct computation shows that for $r<\frac{2}{5}$
\begin{equation}
\label{eq:formula_radial_derivative}
dr\left(\partial_1(f\circ\rho\circ f^{-1})(f(r,\theta))\right) =
\frac{1}{\det df(r,\theta)} \cdot dR(\rho(r,\theta))
\left(\begin{matrix}
\partial_2\Theta(r,\theta) \\
-\partial_1\Theta(r,\theta)
\end{matrix}\right)
\end{equation}
and
\begin{equation}
\label{eq:formula_radial_derivative_tilde}
dr\left(\partial_1(\widetilde{f}\circ\rho\circ\widetilde{f}^{-1})(\widetilde{f}(r,\theta))\right) =
\frac{1}{\det d\widetilde{f}(r,\theta)} \cdot dR(\rho(r,\theta))
\left(\begin{matrix}
\partial_2\Theta(\eta(r),\theta) \\
-\partial_1\Theta(\eta(r),\theta)\cdot\eta'(r)
\end{matrix}\right).
\end{equation}
Note that radial monotonicity of $f\circ\rho\circ f^{-1}$ is precisely saying that the left hand side of equation \eqref{eq:formula_radial_derivative} is positive. Thus \eqref{eq:formula_radial_derivative} implies that
\begin{equation*}
dR(\rho(r,\theta))
\left(\begin{matrix}
\partial_2\Theta(r,\theta) \\
-\partial_1\Theta(r,\theta)
\end{matrix}\right) >0.
\end{equation*}
By choosing $\epsilon>0$ sufficiently small, we can guarantee that
\begin{equation*}
dR(\rho(r,\theta))
\left(\begin{matrix}
\partial_2\Theta(\eta(r),\theta) \\
-\partial_1\Theta(\eta(r),\theta)
\end{matrix}\right) >0
\end{equation*}
whenever $r\in (\frac{1}{5}-\epsilon,\frac{1}{5}+\epsilon)$. Using the fact that $\partial_1R>0$ and $\partial_2\Theta>0$, we see that
\begin{equation*}
dR(\rho(r,\theta))
\left(\begin{matrix}
\partial_2\Theta(\eta(r),\theta) \\
0
\end{matrix}\right) >0.
\end{equation*}
Therefore, the linear functional $dR(\rho(r,\theta))$ is positive on any convex linear combination of the vectors $(\partial_2\Theta(\eta(r),\theta) ,
-\partial_1\Theta(\eta(r),\theta))$ and $(\partial_2\Theta(\eta(r),\theta),0)$. Using $\eta'(r)\in [0,1]$, we can hence deduce that
\begin{equation*}
dR(\rho(r,\theta))
\left(\begin{matrix}
\partial_2\Theta(\eta(r),\theta) \\
-\partial_1\Theta(\eta(r),\theta)\cdot\eta'(r)
\end{matrix}\right)>0.
\end{equation*}
Together with \eqref{eq:formula_radial_derivative_tilde} this implies radial monotonicity of $\widetilde{f}\circ\rho\circ\widetilde{f}^{-1}$ in the annulus $\widetilde{f}(\{\frac{1}{5}-\epsilon<r<\frac{1}{5}+\epsilon\})$.
\end{proof}

\begin{proof}[Proof of Corollary \ref{cor:positivity_criterion_for_diffeomorphisms_close_to_the_identity}]
Let us first prove the corollary under the additional assumption that $\phi$ possesses a fixed point $p$ such that $\phi$ is smoothly conjugated to a rotation in some neighbourhood of $p$. In this situation we may apply Lemma \ref{lem:moving_fixed_point_to_origin}. Let $\psi$ be the resulting diffeomorphism of $\BD$. The $2$-form $\psi^*\omega$ and the diffeomorphism $\psi^{-1}\circ\widetilde{\phi}\circ\psi\in\widetilde{\Diff}(\BD,\psi^*\omega)$ satisfy all assumptions in Theorem \ref{theorem:positivity_criterion_for_radially_monotone_diffeomorphisms}, except for possibly rotation invariance of $\psi^*\omega$ near $0$ and $\partial\BD$. Using an equivariant version of Moser's argument, we may construct a diffeomorphism of $\BD$ which is supported near $0$ and $\partial \BD$, commutes with $\psi^{-1}\circ \phi\circ\psi$ and pulls back $\psi^*\omega$ to a $2$-form which is rotation invariant near $0$ and $\partial\BD$. After replacing $\psi$ by its composition with this diffeomorphism, we may assume w.l.o.g. that $\psi^*\omega$ is rotation invariant near the origin and the boundary and apply Theorem \ref{theorem:positivity_criterion_for_radially_monotone_diffeomorphisms} to the tuple $(\psi^*\omega,\psi^{-1}\circ\widetilde{\phi}\circ\psi)$. Let $H$ denote the resulting Hamiltonian. Then the Hamiltonian $H\circ\psi^{-1}$ satisfies all assertions of Corollary \ref{cor:positivity_criterion_for_diffeomorphisms_close_to_the_identity}.\\

Let us now consider the general case. We claim that $\phi$ must possess an interior fixed point which is either degenerate or elliptic. This clearly is the case if $\phi$ is equal to the identity near the boundary $\partial\BD$. So assume that $\phi$ is not equal to the identity near the boundary. Since $\phi$ is conjugated to a rotation near the boundary, this implies that there are no fixed points at all near the boundary. Let us assume that all fixed points of $\phi$ are non-degenerate. Then there are only finitely many fixed points and their signed count equals the Euler characteristic of $\BD$, which is equal to $1$. The Lefschetz sign of positive hyperbolic fixed points is $-1$ and the Lefschetz sign of negative hyperbolic and elliptic fixed points $1$. Thus there must exist a fixed point which is negative hyperbolic or elliptic. Since $\widetilde{\phi}$ is assumed to be $C^1$-close to the identity, there are no negative hyperbolic fixed points, which implies that there must exist an elliptic one. This concludes the proof that there must be a degenerate or elliptic interior fixed point. We choose such a fixed point $p$. By Lemma \ref{lem:perturbing_symplectomorphism_near_elliptic_fixed_point} there exists a Hamiltonian $G$, supported in an arbitrarily small neighbourhood of $p$ and with arbitrarily small $C^2$-norm $\|X_G\|_{C^2}$, such that $\phi'\coloneqq \phi\circ \phi_G^1$ is smoothly conjugated to an rotation in a neighbourhood of $p$. We define the lift $\widetilde{\phi'}$ of $\phi'$ by $\widetilde{\phi'}\coloneqq \widetilde{\phi}\circ\phi_G$ where $\phi_G\in\widetilde{\Diff}(\BD,\omega)$ is represented by the arc $(\phi_G^t)_{t\in [0,1]}$. After shrinking $\|X_G\|_{C^2}$ if necessary, we can assume that the action $\sigma_{\widetilde{\phi'}}$ is positive on all fixed points of $\phi'$. Thus Corollary \ref{cor:positivity_criterion_for_diffeomorphisms_close_to_the_identity} holds for $\widetilde{\phi'}$ by the above discussion. Let $H'$ be a Hamiltonian generating $\widetilde{\phi'}$ and satisfying all assertions in Corollary \ref{cor:positivity_criterion_for_diffeomorphisms_close_to_the_identity}. In fact, it follows from assertion (2) in Theorem \ref{theorem:positivity_criterion_for_radially_monotone_diffeomorphisms} that $H'$ can be chosen to satisfy $dH'_t(p)=0$ and $H'_t(p) = \sigma_{\widetilde{\phi'}}(p)$ for all $t$. The action $\sigma_{\widetilde{\phi'}}(p)$ is close to $\sigma_{\widetilde{\phi}}(p)>0$. Thus we can bound $H'_t(p)$ from below by a positive constant which can be chosen uniform among all sufficiently small $G$. After reparametrizing the Hamiltonian flow of $G$, we can assume w.l.o.g. that $G$ vanishes for $t$ near $0$ and $1$. We define $H$ by
\begin{equation}
\label{eq:proof_of_positivity_criterion_close_to_identity_def_of_ham}
H_t\coloneqq (H'\#\overline{G})_t = H'_t - G_t\circ\phi_G^t\circ (\phi_{H'}^t)^{-1}
\end{equation}
for $t\in [0,1]$. This smoothly extends to a $1$-periodic Hamiltonian and generates $\widetilde{\phi}$. Since $H$ agrees with $H'$ outside a small neighbourhood of $p$, it satisfies assertions (2) and (3) of Corollary \ref{cor:positivity_criterion_for_diffeomorphisms_close_to_the_identity}. If we choose $G$ with sufficiently small norm $\|X_G\|_{C^2}$, we can also guarantee that $H$ is strictly positive in the interior of $\BD$. This follows from \eqref{eq:proof_of_positivity_criterion_close_to_identity_def_of_ham} and the fact that we have a strictly positive lower bound on $H'_t(p)$ independent of $G$. Thus $H$ has all desired properties.
\end{proof}

\section{From Reeb flows to disk-like surfaces of section and approximation results}
\label{section:from_reeb_flows_to_disk_like_surfaces_of_section_and_approximation_results}

Let $\alpha_0$ denote the restriction of the standard Liouville $1$-form $\lambda_0$ on $\BR^4$ defined in \eqref{eq:liouville_vector_field_and_form} to the unit sphere $S^3$. We will refer to $\alpha_0$ as the standard contact form. The goal of this section is to prove the following result.

\begin{prop}
\label{prop:contact_forms_in_C_3_neighbourhood_may_be_approx_by_forms_satisfying_criterion}
Every tight contact form $\alpha$ on $S^3$ which is sufficiently $C^3$-close to the standard contact form $\alpha_0$ can be $C^2$-approximated by contact forms $\alpha'$ with the following properties: There exists a unique Reeb orbit $\gamma$ of minimal action. The local first return map of a small disk transversely intersecting $\gamma$ is smoothly conjugated to an irrational rotation. There exists a smooth embedding $f:\BD\rightarrow S^3$ parametrizing a $\partial$-strong disk-like global surface of section with boundary orbit $\gamma$ such that the $2$-form $\omega\coloneqq f^*d\alpha'$, the first return map $\phi\in \Diff(\BD,\omega)$ and the lift $\widetilde{\phi}_1\in\widetilde{\Diff}(\BD,\omega)$ of $\phi$ with respect to a trivialization of degree $1$ satisfy assertions (1)-(3) below.
\begin{enumerate}
\item $\widetilde{\phi}_1$ is $C^1$-close to the identity $\identity_\BD$.
\item $\phi$ is smoothly conjugated to an irrational rotation in some neighbourhood of the boundary $\partial\BD$.
\item All fixed points $p$ of $\phi$ have positive action $\sigma_{\widetilde{\phi}_1}(p)$.
\end{enumerate}
\end{prop}

\subsection{Approximation results}
\label{subsection:approximation_results}

The next result says that we may perturb the contact form near an elliptic orbit such that the local first return map of the perturbed Reeb flow is smoothly conjugated to a rotation. We can choose the perturbation such that both the contact form and Reeb vector field are close to the original contact form and Reeb vector field with respect to the $C^2$-topology.

\begin{prop}
\label{prop:making_local_return_map_of_elliptic_orbit_a_rotation}
Let $\alpha$ be a contact form on a $3$-manifold $Y$. Let $\gamma$ be an elliptic, simple, closed Reeb orbit of $\alpha$. For every open neighbourhood $V$ of $\gamma$ and every $\epsilon>0$ there exists a contact form $\alpha'$ on $Y$ such that:
\begin{enumerate}
\item $\alpha'$ agrees with $\alpha$ outside $V$.
\item $\|\alpha'-\alpha\|_{C^2} < \epsilon$
\item $\|R_{\alpha'}-R_\alpha\|_{C^2} < \epsilon$
\item Up to reparametrization, $\gamma$ is a simple closed Reeb orbit of $\alpha'$. The local first return map of a small disk transversely intersecting $\gamma$ is smoothly conjugated to an irrational rotation.
\end{enumerate}
\end{prop}

Our proof of Proposition \ref{prop:making_local_return_map_of_elliptic_orbit_a_rotation} requires some preparation. Consider $\BR^2$ equipped with the standard symplectic form $\omega_0$. Any compactly supported symplectomorphism $\phi$ which is sufficiently $C^1$-close to the identity can be represented by a unique compactly supported generating function $W$ (see chapter 9 in \cite{MS17}). If $(X,Y)$ denote the components of $\phi$, the defining equations for the generating function are:
\begin{equation*}
\begin{cases}
X-x = \enspace\partial_2 W(X,y)\\
Y-y = -\partial_1 W(X,y)
\end{cases}
\end{equation*}
Conversely, any compactly supported function $W$ which is sufficiently $C^2$-close to the identity uniquely determines a compactly supported symplectomorphism $\phi$.

\begin{lem}
\label{lem:continuity_of_correspondence_between_symplectomorphisms_and_generating_functions}
Let $k\geq 0$. We equip the space of compactly supported diffeomorphisms with the $C^k$-topology and the space of compactly supported generating functions with the topology induced by the norm $\|\nabla\cdot\|_{C^k}$. The correspondence between symplectomorphisms and generating functions is continuous in both directions with respect to these topologies.
\end{lem}

\begin{proof}
If $\phi$ is a symplectomorphism with corresponding generating function $W$, then the graph $\Gamma(\phi)$ of $\phi$ inside $\BR^2\times\BR^2$ can be parametrized via
\begin{equation*}
\iota : \BR^2\ni (X,y)\mapsto
\left(\begin{matrix}
X-\partial_2W(X,y)\\
y\\
X\\
y-\partial_1W(X,y)
\end{matrix}\right)\in\Gamma(\phi)\subset\BR^2\times\BR^2.
\end{equation*}
For $j\in\{1,2\}$ let $\pi_j:\BR^2\times\BR^2\rightarrow\BR^2$ denote the projection onto the $j$-th factor. Then $\phi$ can be written as
\begin{equation*}
\phi = (\pi_2\circ\iota)\circ (\pi_1\circ\iota)^{-1}.
\end{equation*}
Clearly, the assignment $W\mapsto \pi_j\circ\iota$ is continuous with respect to the topologies on the spaces of functions and diffeomorphisms specified above. Since composition and taking inverses of diffeomorphisms are continuous operations with respect to the $C^k$-topology, this shows that the symplectomorphism $\phi$ depends continuously on $W$. Conversely, given $\phi$ we define the diffeomorphism $\psi(x,y)\coloneqq (X(x,y),y)$. The assignment $\phi\mapsto \psi$ is continuous with respect to the $C^k$-topology. Define the function
\begin{equation*}
V(x,y) \coloneqq (y - Y(x,y), X(x,y)-x).
\end{equation*}
Then $\nabla W$ is given by the composition $\nabla W = V\circ\psi^{-1}$. Both $V$ and $\psi^{-1}$ depend continuously on $\phi$ with respect to the $C^k$-topology. Hence the same is true for $\nabla W$.
\end{proof}

\begin{lem}
\label{lem:small_symplectomorphism_generated_by_small_hamiltonian}
There exist a $C^1$-open neighbourhood $\MU \subset \Symp_c(\BR^2,\omega_0)$ of the identity and a map
\begin{equation*}
\MU\rightarrow C^\infty_c([0,1]\times \BR^2)\qquad \phi\mapsto H_\phi
\end{equation*}
such that:
\begin{enumerate}
\item $H_\phi$ generates $\phi$.
\item $H_{\identity} = 0$
\item For every integer $k\geq 1$ the following is true: Equip $\MU$ with the $C^k$-topology and $C^\infty_c([0,1]\times \BR^2)$
with the topology induced by the norm $H\mapsto \|X_H\|_{C^{k-1}}$ where we view $X_H$ as an element of $C^\infty_c([0,1] \times \BR^2,\BR^2)$. The map $\phi\mapsto H_\phi$ is continuous with respect to these topologies.
\end{enumerate}
\end{lem}

\begin{proof}
Let $\phi$ be a compactly supported symplectomorphism sufficiently $C^1$-close to the identity and let $W$ be the associated generating function. For $t\in [0,1]$, let $\phi_t$ be the symplectomorphism associated to the generating function $t\cdot W$. Let $H_\phi:[0,1]\times \BR^2\rightarrow\BR$ be the unique compactly supported Hamiltonian generating the flow $(\phi_t)_{t\in [0,1]}$. This yields a map $\phi\mapsto H_\phi$ defined on a $C^1$-open neighbourhood of the identity. Clearly, assertions (1) and (2) hold. It remains to check continuity. Let $k\geq 1$. Consider the parametrization
\begin{equation*}
\iota_t : \BR^2\ni (X,y)\mapsto
\left(\begin{matrix}
X-t\cdot \partial_2W(X,y)\\
y\\
X\\
y- t\cdot \partial_1W(X,y)
\end{matrix}\right)\in\Gamma(\phi_t)\subset\BR^2\times\BR^2
\end{equation*}
of $\Gamma(\phi_t)$. We have
\begin{equation*}
\phi_t = (\pi_2\circ\iota_t)\circ (\pi_1\circ\iota_t)^{-1}
\end{equation*}
where $\pi_j:\BR^2\times\BR^2\rightarrow\BR^2$ denotes the projection onto the $j$-th factor. Clearly, the map
\begin{equation*}
(C^\infty_c(\BR^2),\|\nabla \cdot\|_{C^k})\enspace\ni\enspace W \mapsto \pi_j\circ\iota_t \enspace\in\enspace (C^\infty([0,1]\times \BR^2,\BR^2),\|\cdot\|_{C^k})
\end{equation*}
is continuous. Since composition and inversion of diffeomorphisms are continuous operations with respect to the $C^k$-topology, this implies that
\begin{equation}
\label{eq:continuity_of_generating_function_mapsto_flow}
(C^\infty_c(\BR^2),\|\nabla \cdot\|_{C^k})\enspace\ni\enspace W \mapsto \phi_t \enspace\in\enspace (C^\infty([0,1] \times \BR^2,\BR^2),\|\cdot\|_{C^k})
\end{equation}
is continuous. We have $X_{H_\phi} = (\partial_t\phi_t)\circ\phi_t^{-1}$. Combining Lemma \ref{lem:continuity_of_correspondence_between_symplectomorphisms_and_generating_functions} with continuity of \eqref{eq:continuity_of_generating_function_mapsto_flow}, we obtain that $\phi\mapsto X_{H_\phi}$ is continuous with respect to the $C^k$-topology on symplectomorphisms and the $C^{k-1}$-topology on vector fields.
\end{proof}

\begin{lem}
\label{lem:perturbing_symplectomorphism_near_elliptic_fixed_point}
Let $\omega$ be an area form and $\phi$ a symplectomorphism defined near the origin of $\BR^2$. Assume that $0$ is a fixed point of $\phi$ and that it is either elliptic or degenerate. Then there exists a Hamiltonian $H\in C^\infty_c([0,1]\times \BR^2,\BR)$ supported inside an arbitrarily small open neighbourhood of $0$ and with arbitrarily small norm $\|X_H\|_{C^2}$ such that $0$ is a fixed point of $\phi\circ \phi_H^1$ and such that $\phi\circ \phi_H^1$ is smoothly conjugated to an irrational rotation in some neighbourhood of $0$.
\end{lem}

\begin{proof}
After a change of coordinates, we can assume w.l.o.g. that $\omega=\omega_0$. Since $0$ is elliptic or degenerate as a fixed point of $\phi$, we may choose a $C^\infty$-small Hamiltonian $H$ supported in a small neighbourhood of $0$ such that $0$ is an elliptic fixed point of $\phi\circ\phi_H^1$ with rotation number an irrational multiple of $2\pi$ and such that $\phi\circ\phi_H^1$ is real analytic in some open neighbourhood of $0$. It follows from \cite[Chapter 23, p. 172-173]{MoSi95} that there exists a symplectomorphism $\psi$ defined in an open neighbourhood of $0$ and fixing $0$ such that
\begin{equation*}
\psi^{-1}\circ\phi\circ\phi_H^1\circ\psi (x,y) =
\left(\begin{matrix}
\cos(\theta(x,y)) & -\sin(\theta(x,y)) \\
\sin(\theta(x,y)) & \cos(\theta(x,y))
\end{matrix}\right)\cdot 
\left(\begin{matrix}
x\\y
\end{matrix}\right) + O_4(x,y)
\end{equation*}
where $\theta(x,y)=\theta_0 + \theta_1(x^2+y^2)$ for real constants $\theta_0$ and $\theta_1$ and $O_4(x,y)$ is a real analytic map vanishing up to order $3$ at the origin. Since $d(\phi\circ\phi_H^1)(0)$ is conjugated to an irrational rotation, the constant $\theta_0$ is an irrational multiple of $2\pi$. There exists a symplectomorphism $\xi$ arbitrarily $C^3$-close to the identity and supported in an arbitrarily small neighbourhood of $0$ such that
\begin{equation*}
\psi^{-1}\circ\phi\circ\phi_H^1\circ\psi \circ \xi (x,y) =
\left(\begin{matrix}
\cos(\theta(x,y)) & -\sin(\theta(x,y)) \\
\sin(\theta(x,y)) & \cos(\theta(x,y))
\end{matrix}\right)\cdot 
\left(\begin{matrix}
x\\y
\end{matrix}\right)
\end{equation*}
near $0$. Lemma \ref{lem:small_symplectomorphism_generated_by_small_hamiltonian} yields a Hamiltonian $G$ supported in a small neighbourhood of $0$ such that $\phi_G^1 = \xi$ and such that $\|X_G\|_{C^2}$ is controlled by $\|\xi-\identity\|_{C^3}$. We may choose an autonomous Hamiltonian $K$ such that $\|K\|_{C^3}$ is arbitrarily small, $K$ is supported in an arbitrarily small neighbourhood of $0$ and $K(x,y) = \frac{\theta_1}{4}(x^2+y^2)^2$ near $0$. The time-$1$-map $\phi_K^1$ is given by
\begin{equation*}
\phi_K^1(x,y) = \left(\begin{matrix}
\cos(-\theta_1(x^2+y^2)) & -\sin(-\theta_1(x^2+y^2)) \\
\sin(-\theta_1(x^2+y^2)) & \cos(-\theta_1(x^2+y^2))
\end{matrix}\right)\cdot 
\left(\begin{matrix}
x\\y
\end{matrix}\right)
\end{equation*}
in a neighbourhood of $0$. Thus we have
\begin{equation*}
\psi^{-1}\circ\phi\circ\phi_H^1\circ\psi \circ \phi_G^1\circ \phi_K^1 (x,y)=
\left(\begin{matrix}
\cos\theta_0 & -\sin\theta_0 \\
\sin\theta_0 & \cos\theta_0
\end{matrix}\right)\cdot 
\left(\begin{matrix}
x\\y
\end{matrix}\right).
\end{equation*}
Hence
\begin{equation*}
\phi\circ\phi_H^1\circ\psi \circ \phi_G^1\circ \phi_K^1\circ\psi^{-1} = \phi\circ \phi_H^1\circ \phi_{G\circ\psi^{-1}}^1\circ \phi_{K\circ\psi^{-1}}^1
\end{equation*}
is smoothly conjugated to an irrational rotation. Let $\eta:[0,1]\rightarrow [0,1]$ be a smooth function such that $\eta(t)=0$ for $t$ near $0$ and $\eta(t)=1$ for $t$ near $1$ and $\eta'\geq 0$. Define the Hamiltonian $F$ by
\begin{equation*}
F:[0,1] \times \BR^2\rightarrow\BR\qquad F(t,z)\coloneqq
\begin{cases}
3\cdot\eta'(3t)\cdot K(\eta(3t),\psi^{-1}(z)) & \text{for}\enspace 0\leq t\leq \frac{1}{3} \\
3\cdot\eta'(3t-1)\cdot G(\eta(3t-1),\psi^{-1}(z)) & \text{for}\enspace \frac{1}{3}\leq t\leq \frac{2}{3} \\
3\cdot\eta'(3t-2)\cdot H(\eta(3t-2),z) & \text{for}\enspace \frac{2}{3}\leq t\leq 1.
\end{cases}
\end{equation*}
$F$ is compactly supported, vanishes for $t$ near $0$ and $1$ and generates the symplectomorphism $\phi_H^1\circ \phi_{G\circ\psi^{-1}}^1\circ \phi_{K\circ\psi^{-1}}^1$. Thus $\phi\circ\phi_F^1$ is smoothly conjugated to an irrational rotation. By shrinking the supports of $H$, $\xi$ and $K$ and the norms $\|H\|_{C^\infty}$, $\|\xi-\identity\|_{C^3}$ and $\|K\|_{C^3}$, we can make the support of $F$ and the norm $\|X_F\|_{C^2}$ arbitrarily small.
\end{proof}

\begin{lem}
\label{lem:lift_modified_return_map_to_contact_form}
Let $\alpha$ be a contact form on a $3$-manifold $Y$ and let $\gamma$ be a simple closed Reeb orbit. Let $p$ be a point on $\gamma$ and let $D$ be a small disk intersecting $\gamma$ transversely in $p$. We denote $\omega\coloneqq d\alpha|_D$. Let $\phi:(U,\omega)\rightarrow (D,\omega)$ be the local first-return-map of the Reeb flow, defined in some open neighbourhood $U\subset D$. There exist a $C^1$-open neighbourhood $\MU$ of zero inside $C^\infty_c([0,1]\times U,\BR)$ and a map
\begin{equation*}
\MU\rightarrow \Omega^1(Y)\qquad H \mapsto \alpha_H
\end{equation*}
such that:
\begin{enumerate}
\item $\alpha_H$ is a contact form and agrees with $\alpha$ outside a small neighbourhood of $\gamma$.
\item The local first return map of the Reeb flow of $\alpha_H$ is given by $\phi\circ \phi_H^1$.
\item $\alpha_0 = \alpha$
\item For every interger $k\geq 0$ the following is true: Equip $\MU$ with the topology induced by the norm $\|X_H\|_{C^k}$. Equip $\Omega^1(Y)$ and $\operatorname{Vect}(Y)$ with the $C^k$-topologies. Then the maps $H\mapsto \alpha_H$ and $H\mapsto R_{\alpha_H}$ are continuous with respect to these topologies.
\end{enumerate}
\end{lem}

\begin{proof}
Denote $\lambda\coloneqq \alpha|_{D}$. For $T>0$ sufficiently small, there exists an embedding
\begin{equation*}
F:[0,T]\times D\rightarrow Y
\end{equation*}
such that the restriction of $F$ to $\{0\}\times D$ is the inclusion of $D$ and such that $F^*\alpha= dt + \lambda$. Let $\eta:[0,T]\rightarrow [0,1]$ be a smooth function which is equal to $0$ near $0$, equal to $1$ near $T$ and satisfies $\eta'\geq 0$. Given a Hamiltonian $H\in\MU$, we define $H'$ by
\begin{equation*}
H'(z,t)\coloneqq \eta'(t)\cdot H(z,\eta(t)).
\end{equation*}
$H'$ vanishes for $t$ near $0$ and $T$ and its time-$T$-map agrees with the time-$1$-map of $H$. The map $H\mapsto H'$ is continuous with respect to the topology induced by the norm $\|X_H\|_{C^k}$. We define $\alpha_H$ by
\begin{equation}
\label{eq:lift_modified_return_map_to_contact_form}
\alpha_H\coloneqq (1+H')dt + \lambda
\end{equation}
in the coordinate chart $F$. We extend $\alpha_H$ to all of $Y$ by setting it equal to $\alpha$ outside $\operatorname{im}(F)$. If $H$ is sufficiently $C^1$-small, then $\alpha_H$ is a contact form. The Reeb vector field inside the coordinate chart $F$ is given by
\begin{equation}
\label{eq:lift_modified_return_map_to_contact_form_reeb_vector_field}
R_{\alpha_H} = \frac{1}{1 + H' + \lambda(X_{H'})}\cdot (\partial_t + X_{H'}).
\end{equation}
This is positively proportional to $\partial_t + X_{H'}$. Thus the local first return map of the disk $D$ induced by the Reeb flow of $\alpha_H$ is given by $\phi\circ\phi_H^1$. It is immediate from formulas \eqref{eq:lift_modified_return_map_to_contact_form} and \eqref{eq:lift_modified_return_map_to_contact_form_reeb_vector_field} that $\alpha_H$ and $R_{\alpha_H}$ depend continuously on $H$ with respect to the topologies specified in assertion (4).
\end{proof}

We are finally ready to prove Proposition \ref{prop:making_local_return_map_of_elliptic_orbit_a_rotation}.

\begin{proof}[Proof of Proposition \ref{prop:making_local_return_map_of_elliptic_orbit_a_rotation}]
Let $D$ be a small disk transversely intersecting $\gamma$ in a point $p$. Denote $\omega\coloneqq d\alpha|_D$. For a sufficiently small open neighbourhood $p\in U\subset D$ we have a well-defined local first return map $\phi:(U,\omega)\rightarrow (D,\omega)$. The map $\phi$ has an elliptic fixed point at $p$. By Lemma \ref{lem:perturbing_symplectomorphism_near_elliptic_fixed_point}, there exists a Hamiltonian $H:[0,1]\times U\rightarrow\BR$, supported in an arbitrarily small neighbourhood of $p$ and with arbitrarily small norm $\|X_H\|_{C^2}$, such that $\phi\circ\phi_H^1$ is smoothly conjugated to an irrational rotation in some neighbourhood of $p$. Lemma \ref{lem:lift_modified_return_map_to_contact_form} yields a contact form $\alpha_H$ on $Y$ which agrees with $\alpha$ outside a small neighbourhood of $\gamma$ such that the local first return map of the Reeb flow of $\alpha_H$ is given by $\phi\circ\phi_H^1$. By assertion (4) in Lemma \ref{lem:lift_modified_return_map_to_contact_form}, we can make $\|\alpha_H-\alpha\|_{C^2}$ and $\|R_{\alpha_H}-R_\alpha\|_{C^2}$ arbitrarily small by shrinking $\|X_H\|_{C^2}$.
\end{proof}

\subsection{Proof of Proposition \ref{prop:contact_forms_in_C_3_neighbourhood_may_be_approx_by_forms_satisfying_criterion}}
\label{subsection:proof_of_proposition_on_contact_forms_in_c_3_neighbourhood}

This section is devoted to the proof of Proposition \ref{prop:contact_forms_in_C_3_neighbourhood_may_be_approx_by_forms_satisfying_criterion}.

\begin{lem}
\label{lem:viterbo_near_standard_contact_form_if_local_return_map_is_rotation}
There exists $\epsilon>0$ with the following property: Let $\alpha$ be a tight contact form on $S^3$ satisfying the following conditions:
\begin{enumerate}
\item $\|R_\alpha-R_{\alpha_0}\|_{C^2}<\epsilon$
\item $R_\alpha = c\cdot R_{\alpha_0}$ on the great circle $\Gamma\coloneqq \{(z_1,0)\mid |z_1|=1\}$ for some constant $c>0$.
\item $\Gamma$ is the unique shortest Reeb orbit of $\alpha$.
\item The local return map of a small disk transversely intersecting $\Gamma$ is smoothly conjugated to an irrational rotation.
\end{enumerate}
Then there exists a smooth embedding $f:\BD\rightarrow S^3$ parametrizing a $\partial$-strong disk-like surface of section with boundary orbit $\Gamma$ such that the $2$-form $\omega\coloneqq f^*d\alpha$ and the lift of the first return map $\widetilde{\phi}_1\in\widetilde{\Diff}(\BD,\omega)$ with respect to a trivialization of degree $1$ satisfy assertions (1)-(3) in Proposition \ref{prop:contact_forms_in_C_3_neighbourhood_may_be_approx_by_forms_satisfying_criterion}.
\end{lem}

\begin{proof}
Our proof is based on Proposition 3.6 in \cite{ABHS18}. We define
\begin{equation*}
f:\BR/\BZ\times\BD\rightarrow S^3 \quad f(t,e^{i\theta})\coloneqq \left(\sin\left(\frac{\pi}{2}r\right)e^{i(\theta+2\pi t)},\cos\left(\frac{\pi}{2}r\right)e^{2\pi it}\right).
\end{equation*}
Up to replacing $\BR/\pi\BZ$ by $\BR/\BZ$, this agrees with the map $f$ defined in \cite{ABHS18}. By assertion (iii) in \cite[Proposition 3.6]{ABHS18}, the pull-back of the Reeb vector field $R_\alpha$ via $f|_{\BR/\BZ\times\interior(\BD)}$ extends to a smooth vector field $R$ on the closed solid torus $\BR/\BZ\times\BD$. Moreover, the $C^1$-norm $\|R-\partial_t\|_{C^1}$ is controlled by the $C^2$-norm $\|R_\alpha - R_{\alpha_0}\|_{C^2}$. This shows that if $R_\alpha$ is sufficiently $C^2$-close to $R_{\alpha_0}$, then $R$ is positively transverse to the fibres $t\times\BD$ of the solid torus. In particular, the restriction of $f$ to $0\times\BD$ parametrizes a $\partial$-strong disk-like global surface of section of the Reeb flow of $\alpha$. Let $z$ be a point in the boundary $\partial\BD$. The map
\begin{equation*}
\BR/\BZ\rightarrow \Gamma \quad t \mapsto f(t,z)
\end{equation*}
has degree $1$. Thus the flow of $R$ on $\BR/\BZ \times \BD$ induces the lift $\widetilde{\phi}_1$ of the first return map of $f|_{0\times \BD}$ with respect to a trivialization of degree $1$. We see that the $C^1$-distance between $\widetilde{\phi}_1$ and the identity $\identity_\BD$ is controlled by the $C^2$-distance between $R_\alpha$ and $R_{\alpha_0}$, which yields assertion (1) of Proposition \ref{prop:contact_forms_in_C_3_neighbourhood_may_be_approx_by_forms_satisfying_criterion}. The hypothesis that the local first return map of the orbit $\Gamma$ is smoothly conjugated to an irrational rotation implies that the global first return map $\phi$ is smoothly conjugated to an irrational rotation near the boundary. Let $p$ be a fixed point of $\phi$ corresponding to a closed Reeb orbit $\gamma$ of $\alpha$. By assumption, $\Gamma$ is the unique shortest Reeb orbit of $\alpha$. Thus $\int_\gamma\alpha > \int_\Gamma\alpha$. Let $\widetilde{\phi}_0$ denote the lift of $\phi$ with respect to a trivialization of degree $0$. It follows from Lemma \ref{lem:action_equals_first_return_time} that $\int_\gamma\alpha = \sigma_{\widetilde{\phi}_0}(p)$. The actions of $\widetilde{\phi}_0$ and $\widetilde{\phi}_1$ are related via $\sigma_{\widetilde{\phi}_0}(p)=\sigma_{\widetilde{\phi}_1}(p)+\int_\Gamma\alpha$. Thus we can conclude that $\sigma_{\widetilde{\phi}_1}(p)$ is positive.
\end{proof}

\begin{lem}
\label{lem:move_orbit_to_great_circle}
For every $\epsilon >0$ there exists $\delta >0$ such that for all tight contact forms $\alpha$ on $S^3$ which satisfy $\| R_\alpha - R_{\alpha_0}\|_{C^2}< \delta$ and for all simple closed Reeb orbits $\gamma$ of action less than $\frac{3}{2}\cdot\pi$ there exists a diffeomorphism $\psi$ of $S^3$ such that:
\begin{enumerate}
\item $R_{\psi^*\alpha} = c\cdot R_{\alpha_0}$ on the great circle $\Gamma \coloneqq \{(z_1,0)\mid |z_1|=1\}$ for some constant $c>0$.
\item $\psi(\Gamma) = \operatorname{im} (\gamma)$
\item $\|R_{\psi^*\alpha} - R_{\alpha_0}\|_{C^2} < \epsilon$
\end{enumerate}
\end{lem}
\begin{proof}
This is immediate from Proposition 3.10 in \cite{ABHS18}.
\end{proof}

\begin{proof}[Proof of Proposition \ref{prop:contact_forms_in_C_3_neighbourhood_may_be_approx_by_forms_satisfying_criterion}]
Choose $\epsilon >0$ as in Lemma \ref{lem:viterbo_near_standard_contact_form_if_local_return_map_is_rotation}. Choose corresponding $\delta>0$ as in Lemma \ref{lem:move_orbit_to_great_circle}. Let $\MU$ be a small $C^3$-open neighbourhood of $\alpha_0$ such that all $\alpha\in\MU$ satisfy $\|R_\alpha-R_{\alpha_0}\|_{C^2}<\delta$. We also demand that $(S^3,\alpha)$ is strictly contactomorphic to the boundary of a strictly positively curved domain. We prove that the Proposition holds for all $\alpha\in\MU$. It suffices to consider $C^\infty$-generic $\alpha$. In particular we can assume that all periodic orbits are non-degenerate and that there exists a unique orbit $\gamma$ of minimal action. Since $(S^3,\alpha)$ is strictly contactomorphic to the boundary of a strictly positively curved domain, it follows from \cite{Ek90} (see in particular Theorem 3 and Proposition 9 in chapter V) that $\gamma$ must have Conley-Zehnder index $3$ with respect to a global trivialization of the contact structure. This implies that $\gamma$ must be elliptic or negative hyperbolic. By shrinking $\MU$ we can guarantee the linearized return map of $\gamma$ to be arbitrarily close to the identity. Hence we can guarantee that $\gamma$ is elliptic. We apply Proposition \ref{prop:making_local_return_map_of_elliptic_orbit_a_rotation}. This yields a contact form $\alpha'$ approximating $\alpha$ in the $C^2$-topology and agreeing with $\alpha$ outside a small neighbourhood of $\gamma$ such that the local return map of $\gamma$ generated by the Reeb flow of $\alpha'$ is smoothly conjugated to an irrational rotation. We can also demand $\|R_{\alpha'} - R_{\alpha}\|_{C^2}$ to be arbitrarily small. In particular we can guarantee that $\|R_{\alpha'} - R_{\alpha_0}\|_{C^2}<\delta$. We apply Lemma \ref{lem:move_orbit_to_great_circle} to the contact form $\alpha'$ and the Reeb orbit $\gamma$ of minimal action. Let $\psi$ be a diffeomorphism of $S^3$ satisfying properties (1)-(3) in Lemma \ref{lem:move_orbit_to_great_circle}. Then the contact form $\psi^*\alpha'$ satisfies all assumptions in Lemma \ref{lem:viterbo_near_standard_contact_form_if_local_return_map_is_rotation}. Hence there exists a smooth embedding $f:\BD\rightarrow S^3$ parametrizing a $\partial$-strong disk-like surface of section of the Reeb flow of $\psi^*\alpha'$ with boundary orbit $\Gamma$ such that the $2$-form $\omega\coloneqq f^*d(\psi^*\alpha')$ and the first return map $\widetilde{\phi}_1\in\widetilde{\Diff}(\BD,\omega)$ satisfy assertions (1)-(3) in Proposition \ref{prop:contact_forms_in_C_3_neighbourhood_may_be_approx_by_forms_satisfying_criterion}. Since $\psi^*\alpha'$ and $\alpha'$ are strictly contactomorphic, the same is true for $\alpha'$.
\end{proof}

\section{Proofs of the main results}
\label{section:proofs_of_the_main_results}

\begin{proof}[Proof of Theorem \ref{theorem:strong_viterbo_near_round_ball}]
Let $X\subset\BR^4$ be a convex domain such that $\partial X$ is $C^3$-close to the unit sphere $S^3$. Let $g:S^3\rightarrow\BR_{>0}$ be the unique function such that
\begin{equation*}
\partial X = \{\sqrt{g(x)}\cdot x\mid x\in S^3\}.
\end{equation*}
The function $g$ is $C^3$-close to the constant function $1$. The pull-back of the contact form $\lambda_0|_{\partial X}$ via the radial map
\begin{equation*}
S^3\rightarrow\partial X \qquad x\mapsto \sqrt{g(x)}\cdot x
\end{equation*}
is given by $\alpha \coloneqq g\cdot \alpha_0$ and is $C^3$-close to $\alpha_0$. Let $\alpha'$ be a contact form which is $C^2$-close to $\alpha$ and satisfies all assertions of Proposition \ref{prop:contact_forms_in_C_3_neighbourhood_may_be_approx_by_forms_satisfying_criterion}. We claim that there exists a star-shaped domain $X'$ whose boundary $\partial X'$ is $C^1$-close to $\partial X$ and strictly contactomorphic to $(S^3,\alpha')$. Indeed, arguing as in the proof of Proposition 3.11 in \cite{ABHS18}, we conclude that there exists a $C^1$-open neighbourhood $\MU\subset \Omega^1(S^3)$ of $\alpha$ and a map
\begin{equation*}
\MU \rightarrow C^\infty(S^3) \quad \beta \mapsto g_\beta
\end{equation*}
which is continuous with respect to the $C^{k+1}$ topology on $\MU$ and the $C^k$-topology on $C^\infty(S^3)$, maps $\alpha$ to the constant function $1$ and has the property that every $\beta\in\MU$ is a contact form strictly contactomorphic to $g_\beta\cdot \alpha$. Since $\alpha'$ is $C^2$-close to $\alpha$, the function $g_{\alpha'}$ is $C^1$-close to the constant function $1$. We define $X'$ to be the star-shaped domain with boundary
\begin{equation*}
\partial X' = \{ \sqrt{g_{\alpha'}(x)\cdot g(x)}\cdot x\mid x\in S^3\}.
\end{equation*}
$\partial X'$ is $C^1$-close to $\partial X$. The pull-back to $S^3$ of the contact form $\lambda_0|_{\partial X'}$ via the radial projection is given by $g_{\alpha'}\cdot g\cdot \alpha_0 = g_{\alpha'}\cdot \alpha$. This is strictly contactomorphic to $\alpha'$. We claim that $\cG(X') = \cZ(X')$. Let $f:\BD\rightarrow\partial X'$ be a surface of section satisfying the assertions of Proposition \ref{prop:contact_forms_in_C_3_neighbourhood_may_be_approx_by_forms_satisfying_criterion}. This means that we may apply Corollary \ref{cor:positivity_criterion_for_diffeomorphisms_close_to_the_identity} to $\widetilde{\phi}_1$, the lift of the first return map with respect to a trivialization of degree $1$. Let $H:\BR/\BZ\times\BD\rightarrow\BR$ denote a Hamiltonian satisfying all assertions of Corollary \ref{cor:positivity_criterion_for_diffeomorphisms_close_to_the_identity}. This Hamiltonian satisfies all hypotheses for the second part of Theorem \ref{theorem:embedding_result}. We conclude that $B(a)\overset{s}{\hookrightarrow}X'\overset{s}{\hookrightarrow} Z(a)$ where $a>0$ is the symplectic area of the surface of section. In particular, this implies that $\cG(X')=\cZ(X')$. We can make the $C^1$-distance between $\partial X$ and $\partial X'$ arbitrarily small by letting the $C^2$-distance between $\alpha$ and $\alpha'$ go to zero. This shows that $X$ may be approximated in the $C^1$-topology by star-shaped domains $X'$ whose Gromov width and cylindrical embedding capacity agree. It is an easy consequence of the monotonicity and conformality of symplectic capacities that any symplectic capacity is continuous on the space of all star-shaped domains with respect to the $C^0$-topology. Therefore $\cG(X)=\cZ(X)$.
\end{proof}

\begin{proof}[Proof of Theorem \ref{theorem:area_of_surface_of_section_bounds_cylindrical_capacity}]
We apply Proposition \ref{prop:modify_hypersurface_such_that_return_map_is_generated_by_positive_hamiltonian}. This yields a star-shaped domain $X'$ such that $X\overset{s}{\hookrightarrow} X'$ and a $\partial$-strong disk-like global surface of section $\Sigma'\subset\partial X'$ such that $\Sigma'$ has the same symplectic area as $\Sigma$ and such that $(X',\Sigma')$ satisfies all hypotheses in the first part of Theorem \ref{theorem:embedding_result}. Let $a>0$ denote the symplectic area of $\Sigma$. By Theorem \ref{theorem:embedding_result}, there exist symplectic embeddings $X\overset{s}{\hookrightarrow} X'\overset{s}{\hookrightarrow} Z(a)$. In particular, we have $\cZ(X)\leq a$.
\end{proof}

\begin{proof}[Proof of Theorem \ref{theorem:a_hopf_equals_cylindrical_capacity_for_convex_domains}]
It was proved by Hofer-Wysocki-Zehnder in \cite{HWZ96} that every star-shaped domain $X$ possesses a Hopf orbit, i.e. that $\operatorname{A_{Hopf}}(X)<\infty$. Although not explicitly stated in this form, the proof of the existence result of Hopf orbits in \cite{HWZ96} shows more, namely that if $X$ symplectically embeds into the cylinder $Z(a)$, then $\operatorname{A_{Hopf}}(X)\leq a$. The reason is the following. Since $X$ embeds into $Z(a)$, it also embeds into the product $M\coloneqq S^2(a+\varepsilon)\times T^2(b)$. Here $S^2(a+\varepsilon)$ is the $2$-sphere equipped with an area form of total area $a+\varepsilon$ for an arbitrarily small $\varepsilon>0$ and $T^2(b)$ is the $2$-torus of total area $b$ for some sufficiently large $b$. For an arbitrary compatible almost complex structure $J$, the symplectic manifold $M$ is foliated by $J$-holomorphic spheres in the homology class $[S^2\times *]$. Hofer-Wysocki-Zehnder produce a Hopf orbit by neck-stretching $J$-holomorphic spheres along the boundary of $X$. Carrying out this procedure with $J$-holomorphic spheres of symplectic area $a+\varepsilon$ yields a Hopf orbit with action at most $a+\varepsilon$. Therefore $\operatorname{A_{Hopf}}(X)\leq a+\varepsilon$. Since $\varepsilon>0$ is arbitrary, this shows that $\operatorname{A_{Hopf}}(X)\leq a$.

Let us now assume that $X$ is dynamically convex. It is proved by Hryniewicz-Hutchings-Ramos \cite{HHR21} that the infimum in the definition of $\operatorname{A_{Hopf}}(X)$ is attained in this case. This means that there exists a Hopf orbit $\gamma$ with action $\MA(\gamma)=\operatorname{A_{Hopf}}(X)$. As mentioned in the introduction, it follows from work of Hryniewicy-Salom\~{a}o \cite{HS11} and Hryniewicz \cite{Hry14} that $\gamma$ bounds a disk-like global surface of section. In fact, Florio-Hryniewicz proved in \cite[Proposition 2.8]{FH21} that $\gamma$ bounds a $\partial$-strong disk-like surface of section. Thus it is a direct consequence of Theorem \ref{theorem:area_of_surface_of_section_bounds_cylindrical_capacity} that $X$ symplectically embeds into $Z(\operatorname{A_{Hopf}}(X))$.
\end{proof}

\bibliographystyle{plain}
\bibliography{surf_sect_sympl_cap_v2}

\end{document}